\documentclass[reqno,a4paper]{article}

\usepackage{amssymb,amsmath,epsfig,graphics,graphicx,color,subfigure,blindtext,slashbox,diagbox}

\definecolor{lightblue}{rgb}{0.22,0.45,0.70}
\usepackage[colorlinks=true,breaklinks=true,linkcolor=lightblue,citecolor=lightblue]{hyperref}
\usepackage{float}
\usepackage{tikz}

\usetikzlibrary{arrows.meta}
\usepackage[font=small,skip=0pt]{caption}

\usepackage{multicol}
\usepackage{mathrsfs}
\usepackage{comment}

\usepackage{caption}

\newtheorem{remark}{Remark}[section]
\newtheorem{lemma}{Lemma}[section]
\newtheorem{theorem}{Theorem}[section]
\newtheorem{proposition}{Proposition}[section]

\usepackage{soul}

\renewcommand{\O}{\Omega}

\newcommand\E{{E}}

\newcommand\mS{\mathcal S}
\newcommand\mP{\mathcal P}

\def\qon{{\quad\hbox{on}\quad}}

\newcommand\R{\mathbb{R}}
\newcommand\N{\mathbb{N}}
\renewcommand{\P}{{\mathbb P}}  % polynomials
\newcommand{\PP}{ \mathbf{P} }  % polynomials^2
\newcommand{\M}{ \mathbb{M}}
\def\S{\mathrm{S}}
\renewcommand{\SS}{\mathbf{S}}

\def\dof{{\rm dof}}
\def\Tol{{\rm Tol}}

\def\Cb{\boldsymbol{C}}

\def\gb{{\bf g}}

\def\vb{{\bf v}}
\def\fb{{\bf f}}
\def\nb{{\bf n}}

\def\Zb{{\bf Z}}
\def\H{{\rm H}}
\def\HH{\mathbf{H}}
\def\L{{\rm L}}
\def\LL{\mathbf{L}}
\def\WW{\mathbf{W}}
\def\W{{\rm W}}
\def\X{{\rm X}}

\def\nbs \boldsymbol{n}
\def\0{\boldsymbol{0}}
\def\div{{\rm div} \:}
\def\calD{\mathcal D}
\def\curl{{\mathbf{curl}} \:}

\def\bc{\mathbf{c}}

\def\D{ \mathbf{D}}

\def\HdoO{{\H_0^2(\O)}}
\def\HuoO{{\H_0^1(\O)}}

\def\CT{\mathcal{T}}

\def\CT{\O}
\def\WtK{\widetilde{\W}_k^{h}(\E)}
\def\WK{\W_k^{h}(\E)}
\def\Wh{\W_k^{h}}
\def\Xh{\X_{k,\ell}^{h}}

\def\HtE{\widetilde{\H}_{\ell}^h(\E)}
\def\HE{\H_{\ell}^h(\E)}
\def\Hh{\H_{\ell}^h}

\def\HdoK{{\H^{2}(\E)}}

\newcommand\bn{{\bf n}}
\def\dn{\partial_{\bn}}

\def\PiK{\Pi_{\E}^{\D,k}}
\def\PiK{\Pi_{\E}^{\D,k}}
\def\PinablaK{\Pi_{\E}^{\nabla, \ell}}
\def\PicurlK{\Pi_{\E}^{\mathbf{c},k}}
\def\PioK{\Pi^{k-2}_{\E}}
\def\PimunoK{\boldsymbol{\Pi}_{\E}^{k-1}}
\def\PiellK{\boldsymbol{\Pi}^{\ell-1}_{\E}}

\def\etanpsi {\eta^n_{\psi}}
\def\varnpsi {\varphi^n_{\psi}}
\def\etantheta {\eta^n_{\theta}}
\def\varntheta {\varphi^n_{\theta}}	
\def\aAFh{ \widehat{\alpha}_{A_F}  }
\def\aATh{ \widehat{\alpha}_{A_T}  }
\def\aMFh{ \widehat{\alpha}_{M_F}  }
\def\aMTh{ \widehat{\alpha}_{M_T}  }

\def\CMFh{ \widehat{C}_{M_F}  }
\def\CMTh{ \widehat{C}_{M_T}  }

\def\CBF{ C_{B_F} }
\def\CBT{C_{B_T} }
\def\CBFh{ \widehat{C}_{B_F}  }
\def\CBTh{\widehat{C}_{B_T} }

\newcommand\rot{\mathop{\mathrm{rot}}\nolimits}
\renewcommand\skew{\mathop{\mathrm{skew}}\nolimits}

\newcommand\bv{\boldsymbol{v}}
\newcommand\bu{\boldsymbol{u}}
\newcommand\bt{{\bf t}}
\def\dofs{{\rm dofs}}

\def\Ra{{\rm Ra}}
\def\Pr{{\rm Pr}}
\def\Nu{{\rm Nu}_{local}}
\def\DXu{{\bf D_W1}}
\def\DXd{{\bf D_W2}}
\def\DXt{{\bf D_W3}}
\def\DXc{{\bf D_W4}}
\def\DXf{{\bf D_W5}}
\def\DHu{{\bf  D_H1}}
\def\DHd{{\bf  D_H2}}
\def\DHt{{\bf  D_H3}}
\def\CI{C_{I} }
\def\Cinv{C_{{\tt inv}}}
\def\CR{C_{{\tt reg}}}

\newenvironment{proof}{\noindent{\it Proof.}}{\hfill$\square$\\}

\begin{document}
	
%***********************************************************************************
\title{A fully-discrete virtual element method for the nonstationary Boussinesq equations}
%***********************************************************************************
	% Louren\c co
\author{L. Beir\~ao da Veiga\thanks{Dipartimento di Matematica e Applicazioni,
Universit\`a degli Studi di Milano Bicocca, Via Roberto Cozzi 55, Milano, Italy and
IMATI-CNR, Via Ferrata 1, Pavia, Italy. 
E-mail: {\tt lourenco.beirao@unimib.it}.},\quad 
D. Mora\thanks{GIMNAP, Departamento de Matem\'atica, Universidad
del B\'io-B\'io, Concepci\'on, Chile and
CI$^2$MA, Universidad de Concepci\'on, Concepci\'on, Chile.
E-mail: {\tt dmora@ubiobio.cl}.},\quad 
A. Silgado\thanks{GIMNAP, Departamento de Matem\'atica,
Universidad del B\'io-B\'io, Concepci\'on, Chile. E-mail:
{\tt alberth.silgado1701@alumnos.ubiobio.cl}.}}

	\date{}
	\maketitle
	
	%***********************************************************************************
	\begin{abstract}
		
In the present work we propose and analyze a fully coupled virtual 
element method of high order for solving the two dimensional nonstationary Boussinesq 
system in terms of the stream-function and temperature fields. 
The discretization for the spatial variables is based on the coupling $C^1$- and $C^0$-conforming 
virtual element approaches, 
while a backward Euler scheme is employed for the temporal variable.
Well-posedness and unconditional stability of the fully-discrete problem is provided. 
Moreover,  error estimates in $\H^2$- and $\H^1$-norms are derived for the stream-function 
and temperature, respectively. Finally, a set of benchmark tests are reported to confirm 
the theoretical error bounds and illustrate the behavior of the fully-discrete scheme.
\end{abstract}
	
\noindent
{\bf Key words}: Virtual element method, nonstationary Boussinesq equations, stream-function form, 
error estimates.
	
\smallskip\noindent
{\bf Mathematics subject classifications (2000)}:  65M12, 65M15, 65M60, 76D05

\maketitle
	%************************************************************************************************
\section{Introduction}	

The Boussinesq system is typically used to describe the natural convection 
in a viscous incompressible fluid, which consists of coupling
between the Navier-Stokes equations with a convection-diffusion equation.
Such coupling is done by means of a buoyancy term  (in the momentum equation 
of the Navier-Stokes system) and convective heat transfer 
(in the  energy equation). Applications of this fluid-thermal system appears in 
several  engineering processes, such as, industrial ovens,
cooling procedures (cooling of steel industries, electronic and electric equipments, 
nuclear reactors, etc). Moreover, this physical phenomena appears in oceanography and 
geophysics  when studying oceanic flows and climate predictions.

Due its relevance and presence in different applications, several works have been devoted to 
study these equations (and some variants). For the analysis of existence, uniqueness 
and regularity of the solution, we refer to~\cite{Morimoto92,LB99}.
Besides, over the last decades several discretizations
have been employed to solve this system; see for instance ~\cite{BMP1995,BL90,TT2005,ZHZ2014_camwa,OQS2014,ABS2015,DA2016,CGO2016,DGN2019,AHMS2018} 
and the references therein, where the steady and unsteady regimens, temperature-dependent 
parameters problems have been studied, considering the classical 
velocity--pressure--temperature and pseudostress--velocity--temperature formulations.
	
Typically, in the existing literature, the majority of the discretizations for the 
fluid part involve the standard velocity--pressure formulation for the Boussinesq system. 
However, some researchers have developed numerical methods
by using the stream-function--vorticity and pure stream-function approaches to approximate 
this system. For instance, in~\cite{stevens1982} a finite element 
discretization is considered to solve the problem in stream-function--vorticity--temperature 
form, numerical solutions are obtained for the natural convection in a square cavity and compared 
with some results available in the literature.
In \cite{TG2003} a fourth-order compact finite difference scheme is formulated for 
solving the steady regimen, by using also the stream-function--vorticity--temperature 
formulation. Numerical experiments are also presented. More recently, 
in~\cite{WLJ2003,WLJ2004}, the authors present the analysis of stability and convergence  
for a fourth-order finite difference method for the unsteady regimen of 
Boussinesq equations with the stream-function--vorticity--temperature approach 
is established. Numerical results are provided in \cite{WLJ2003}.
On the other hand, in~\cite{BH1983}, the authors employed a $C^1$ finite element method 
to approximate the stream-function  variable. Numerical solution for the $2$D natural convection
in a square cavity are presented and compared with benchmark results~\cite{Vahl83}. 

For two dimensional fluid problems, the formulation in terms of the stream-function 
presents several attractive features, among these we can mention: the velocity vector 
and pressure fields are not present in the formulation, instead only one scalar variable 
(the stream-function) is the main unknown to approximate. By construction the  
incompressibility constraint is automatically satisfied. 
Moreover, the resulting trilinear form in the 
momentum equation is  naturally skew-symmetric, which allows more direct stability and 
convergence arguments.
On the other hand, in comparison with the stream-function--vorticity 
form, our approach avoid the difficulties related with the definition of the 
boundary values for the vorticity field, present in such formulation.

 Nevertheless, the construction of subspaces  of $\H^2$ (space where the 
stream-function belongs) by using finite element method involve high order polynomials and a 
 large number of degrees of freedom, which are considered a difficult task principally 
    from the computational viewpoint, even for triangular decompositions. As an alternative to avoid 
    the aforementioned drawback, we consider the approach presented in \cite{BM13,ChM-camwa} to introduce $C^1$-virtual element schemes of arbitrary order $k \geq 2$, to approximate the stream-function variable of the Boussinesq system.

	The Virtual Element Methods (VEM) were introduced in the seminal work~\cite{BBCMMR2013} 
	as an extension of Finite Elements Methods (FEM) to polygonal or polyhedral decompositions. 
	In this first work the Poisson equation is used to illustrate the main ideas of VEM approach.
	The virtual element spaces are constituted by polynomial and nonpolynomial functions,  
	the degrees of freedom must be chosen appropriately so that the stiffness matrix and load term 
	can be computed without computing these nonpolynomial functions. 
	Later on, in \cite{BM13} is introduced a new family of $C^1$-virtual element of high order $k \geq 2$,	to solve  Kirchhoff-Love plate problems, which in the lowest order polynomial degree employed only $3$ degrees of freedom  per mesh vertex (the function and its gradient values vertex). 
	This fact represents a very significant advantage over $C^1$ schemes based on FEM. Moreover, in~\cite{BMan13,AMSV2021}, the authors discuss 
	the application of VEM to construct finite dimensional spaces of arbitrarily regular $C^{\alpha}$, 
	with $\alpha \geq 1$, where promising results have been observed to solve equations involving  
	high-order PDEs. In the last year a wide variety of second- and fourth-order problems have been  
	discretized by using VEM. Due to the large number of papers available in the literature, we here limit ourselves in citing some representative articles within the area of 
	fluid mechanics, where several models have been addressed with the {\it conforming}
	VEM approach: the Stokes equations~\cite{ABMV2014,ChM2020-ACM,BLV-M2AN,VK2021}, 
	the Brinkman model~\cite{CGS17,MRS2021-IMAJNA}, Navier-Stokes and incompressible flows~\cite{BLV-NS18,BMV2019,FM2020-IMAJNA,BDV2020,BPV2021,DGR2022}, the 
	Quasi-Geostrophic equations of the ocean~\cite{MS2021-camwa} and 
	Boussinesq system~\cite{GMS2021,AVV2022},  
	where different formulations have been considered.

	According to the previously discussed, in the present contribution, we are
	interested in further exploring the ability of VEM to approximate 
	coupled nonlinear fluid flow problems considering the stream-function approach.
    More precisely, we develop and analyze a fully-discrete VE scheme for solving 
    the nonstationary Boussinesq system. Under assumption that the domain is 
	simply connected and by using the incompressibility condition of the velocity field,
	we write a equivalent variational formulation in terms of the stream-function and 
	temperature unknowns.
    The discretization for the spatial variables is based on the \textit{coupling} of $C^1$- and $C^0$- conforming    virtual element approaches~\cite{BM13,BBCMMR2013}, for the stream-function and temperature 
    fields, respectively, and  we handle the time derivatives with a classical backward Euler implicit method.
    Employing the discretizations mentioned above, we propose a fully-discrete scheme of high-order, 
    which is fully coupled, implicit in the nonlinear terms and unconditionally stable. 
    By using the fixed point theory, we establish the corresponding existence of a discrete solution and, 
    under a small time step assumption, we prove that such discrete solution is also unique.
    Moreover, employing the natural skew-symmetry property of the resulting discrete trilinear 
    form (in the momentum equation) we provide optimal error estimates in $\H^2$- and $\H^1$-norms 
    for the stream-function and temperature, respectively.

   	%% - Outline: 
		
	The remainder of this paper has been organized as follows: In Section~\ref{preli:cont:problem}
    we provide preliminaries notations and recall the unsteady Boussinesq equations in its standard
    velocity--pressure--temperature formulation. Moreover, we write a weak form of the system in terms
    of the stream-function and temperature variables. We finish this section by recalling the corresponding  
    stability and well-posedness results for the continuous problem. In Section~\ref{VEM:section} we present the VE discretization, introducing the polygonal decomposition and mesh notations, the constructions of stream-function and temperature VE spaces along with their corresponding degrees of freedom,  the polynomial projections  and the construction of the multilinear forms. In Section~\ref{fullydiscrete:BSE} we present the 
    fully-discrete VE formulation and provide its stability and well-posedness. In Section~\ref{SEC:convergence}
    we derive error estimates for the stream-function and temperature fields. 
    Finally,  three numerical experiments, including the solution of the $2$D natural convection
    benchmark problem, are presented in Section~\ref{Numerical:experiments}, to illustrate the good performance of 
    the scheme and confirm our theoretical predictions.

%-----------------------------------------------------------
\section{Preliminaries and the continuous problem}\label{preli:cont:problem}
%-----------------------------------------------------------

We start this section introducing some preliminary notations 
that will be used throughout this work.
Thenceforth, $\O$ will  denote a polygonal simply connected  
bounded domain of $\R^2$, with boundary $\Gamma:=\partial\O$ and 
$\bn=(n_i)_{1\le i\le2}$ is the outward unit normal vector 
to the boundary $\Gamma$ and  $\bt=(t_i)_{i=1,2}:=(-n_2,n_1)$ 
is the unit tangent vector to $\Gamma$. Moreover, we denote by $\partial_\bn$ 
to the normal derivative. 
According to \cite{AF2003}, for any open measurable bounded domain $\calD \subseteq \O$, 
we will employ the usual notation for the Banach spaces $\L^p(\calD)$ 
and the Sobolev spaces $\W_p^{s}(\calD)$, with $s \geq 0$ and $p \in [1,+ \infty]$, 
with the corresponding seminorms 
and norms are denoted by $|\cdot|_{\W_p^{s}(\calD)}$ 
and  $\|\cdot\|_{\W_p^{s}(\calD)}$, respectively.
We adopt the convention $\W_p^{0}(\calD):=\L^p(\calD)$ and in particular when $p=2$, 
we  write $\H^{s}(\calD)$  instead to  $\W_2^{s}(\calD)$, the corresponding
seminorm and norm  of these space will be denoted by $|\cdot|_{s,\calD}$ 
and  $\|\cdot\|_{s,\calD}$, respectively. 
Furthermore, we denote by $\SS$ the corresponding vectorial version of a generic scalar $\S$,
examples of this are: $\LL^{p}(\calD):=[\L^{p}(\calD)]^2$ and $\WW_p^{s}(\calD):=[\W_p^{s}(\calD)]^2$.
%Moreover, 

We denote by $t$ the temporal variable with values
in the interval $I:=(0,T]$, where $T>0$ is a given final time.
Moreover, given a Banach space $V$ endowed with the
norm $\|\cdot \|_V$, we define the space  $\L^p(0,T;V)$ as the space of classes of functions 
$\phi:(0,T) \to V$ that are Bochner measurable and such that $\|\phi\|_{\L^p(0,T;V)} < \infty$, with	
\begin{equation*}
	\begin{aligned}
		\|\phi\|_{\L^p(0,T,V)}:=\left(\int_0^T \|\phi(t)\|_{V}^p \mathrm{d} t \right)^{1/p} \quad \text{and} \quad
		\|\phi\|_{\L^{\infty}(0,T,V)}:=\underset{t \in [0,T]}{\text{ess sup}}\|\phi(t)\|_{V}.
	\end{aligned}
\end{equation*} 	
\subsection{The time dependent Boussinesq system} 
In this work we are interested in approximating the solution 
of the nonstationary Boussinesq system, modeling
incompressible nonisothermal fluid flows. 
The system consists of a coupling between the Navier-Stokes
equations with a convection-diffusion equation for the temperature variable.
The coupling is by means of a buoyancy term (in the momentum equation of the Navier-Stokes system)
and convective heat transfer (in the  energy equation). 
More precisely,
given suitable initial data $(\bu_0,\theta_0)$, the  aforementioned system of equations are given by  (see \cite{Morimoto92}):

\begin{equation}\label{Bouss:Eq}
	\begin{split}
		\partial_t \bu - \nu \Delta\bu + (\bu\cdot\nabla)\bu + \nabla p-\gb \theta
		&=\fb_{\psi}\qquad\textrm{in}\quad\Omega \times (0,T),\\
		\div\bu&= 0\qquad \: \:\textrm{in}\quad\Omega \times (0,T),\\
		\bu&=\0 \qquad \: \:\textrm{on}\quad\Gamma\times (0,T),\\
		\bu(0)&=\bu_0  \qquad\textrm{in}\quad\Omega \text{ at } t=0,\\
		(p,1)_{0,\O} &=0\\
		\partial_t \theta- \kappa  \Delta\theta + \bu\cdot\nabla\theta &=f_{\theta} \: \qquad\textrm{in}\quad\Omega \times (0,T),\\
		\theta&=0  \: \: \: \qquad\textrm{on}\quad\Gamma \times (0,T),\\
		\theta(0)&=\theta_0   \: \:\qquad\textrm{in}\quad\Omega \text{ at } t=0,
	\end{split}
\end{equation} 
where $\bu: \Omega \times (0,T) \to  \R^2$, $p: \Omega \times (0,T) \to  \R$ and $\theta: \Omega \times (0,T) \to  \R $ denote the  velocity, pressure and temperature fields.
The parameters $\nu >0$ and $\kappa >0$ are the viscosity fluid and the thermal conductivity, respectively.   The functions $\fb_{\psi}: \O \times (0,T) \to \R^2$, $f_{\theta}: \O \times (0,T) \to \R$ is a set of external forces and  $\gb: \O \times (0,T) \to \R^2$ is a force per unit mass. 

In next subsection, by using the incompressibility property of the velocity field, 
we will write an equivalent weak formulation of the system~\eqref{Bouss:Eq} in terms
of the stream-function and temperature variables.

\subsection{The time dependent stream-function--temperature formulation}\label{weak_form_section}
Let us introduce the following space of functions belonging to $\HH_0^1(\O)$
with vanishing divergence:
\[
\Zb:= \left\{\bv \in\HH_0^1(\O): \div \bv =0 \right\}.
\]
Since $\O\subset \R^2$ is simply connected, a well known result
states that a vector function $\bv \in \Zb$ if and only if there exists a scalar 
function $\varphi \in \H^2(\O)$ (called \textit{stream-function}), such that
$$\bv =\curl \varphi \in \HH_0^1(\O).$$
The function $\varphi$ is defined up to a constant (see \cite{GR}). 
Thus, we consider the following space 
\begin{equation*}
\HdoO= \left \{\varphi \in \H^2(\O) : \varphi = \dn\varphi = 0  \qon \Gamma  \right \}.
\end{equation*}
Then, choosing  $\psi(t)\in \HdoO$  the stream-function of the velocity field $\bu(t)\in\Zb$ 
(i.e. $\bu(t)=\curl \psi(t)$)  in the momentum equation of system~\eqref{Bouss:Eq}, testing  against 
a function $\bv=\curl \phi$ with $\phi \in  \HdoO$  and applying twice an integration by part, we have 
\begin{equation*}
\int_{\O}\curl(\partial_t \psi)\cdot\curl \phi +	
\nu \int_{\O}\D^2 \psi:\,\D^2 \phi+\int_{\O} \Delta \psi \,   \curl \psi \cdot \nabla\phi
		-\int_{\O}  \gb \theta \cdot \curl \phi =\int_{\O} \fb_{\psi} \cdot \curl \phi  \qquad \forall \phi \in  \HdoO. 
	\end{equation*}
On other hand, multiplying by $v \in \H_0^{1}(\O)$ and  integrating by parts in
the energy equation of system~\eqref{Bouss:Eq}, we obtain 
\begin{equation*}
\int_{\O}\partial_t\theta v +	\kappa \int_{\O}\nabla \theta \cdot \nabla v
+\int_{\O}  (\curl \psi \cdot \nabla \theta )v = \int_{\O} f_{\theta} v \qquad  \forall v \in \H_0^{1}(\O). 
\end{equation*}
From the above identities, we obtain the following weak formulation 
for system~\eqref{Bouss:Eq}: given $\psi_0 \in  \H_0^1(\O)$, $\theta_0 \in \L^2(\O)$,  
$\gb \in \L^{\infty}(0,T;\LL^{\infty}(\O))$, and  the external forces $\fb_{\psi} \in \L^{2}(0,T;\LL^{2}(\O)),
f_{\theta} \in \L^{2}(0,T;\L^{2}(\O)) $, find $(\psi,\theta)\in \L^2(0,T;\H_0^2(\O))  \times\L^2(0,T;\H_0^1(\O))$ such that 
\begin{equation}\label{coupled-stream1}
\begin{split}
M_F(\partial_t \psi, \phi) +\nu A_F(\psi,\phi) +B_F(\psi;\psi,\phi)	- C(\theta, \phi) &=F_{\psi}(\phi) \qquad \forall \phi \in \HdoO, \quad \text{for a.e.  $t \in (0,T)$},\\
M_T(\partial_t\theta,v)   + \kappa A_T(\theta,v) +B_{T}(\psi;\theta,v)&=F_{\theta}(v)  \:\qquad \forall v \in \HuoO, \quad \text{for a.e.  $t \in (0,T)$},\\
\psi(0)=\psi_0, \quad \qquad  \theta(0)&=\theta_0,
\end{split}
\end{equation}
where the bilinear forms  $M_F(\cdot,\cdot)$, $M_T(\cdot,\cdot)$, $A_F(\cdot,\cdot)$ and $A_T(\cdot,\cdot)$ are given by 
\begin{align}
M_F(\cdot,\cdot): \HdoO \times \HdoO &\to \R,  \qquad
M_F(\varphi,\phi) :=\int_{\O} \curl \varphi\cdot\curl \phi,   \label{MF-cont}\\
M_T(\cdot,\cdot): \H_0^1(\O) \times \H_0^1(\O) &\to \R,  \qquad
M_T(v,w) :=\int_{\O}  v w, \label{MT-cont}\\
A_{F}: \HdoO \times \HdoO &\to \R, \qquad A_{F}(\varphi,\phi) 
:= \int_{\O}\D^2 \varphi:\,\D^2 \phi,\label{AF-cont}\\
A_T: \HuoO \times \HuoO &\to \R, \qquad
A_{T}(v,w) := \int_{\O}\nabla v \cdot \nabla w, \label{AT-cont}		
\end{align}
whereas the convective trilinear forms $B_F(\cdot;\cdot,\cdot)$  and $B_T(\cdot;\cdot,\cdot)$ are defined by
\begin{align}
B_F: \HdoO \times \HdoO \times \HdoO &\to \R, \qquad B_F(\zeta;\varphi,\phi) :=\int_{\O} \Delta \zeta \,   \curl \varphi \cdot \nabla\phi, \label{BF-cont}\\
B_T: \HdoO \times \HuoO \times \HuoO &\to \R, \qquad	B_T(\varphi;v,w) 
:=\int_{\O}  (\curl \varphi \cdot \nabla v )w. \label{BT-cont}
\end{align}
The bilinear form $C(\cdot,\cdot)$ associated to the buoyancy term is given by
\begin{equation}\label{D-cont}
C:  \HuoO \times \HdoO \to \R, \qquad	C(v, \phi) := \int_{\O}  \gb v \cdot \curl \phi   
\end{equation}
and the functionals $F_{\psi}(\cdot)$ and $F_{\theta}(\cdot)$ are give by 
\begin{align}
	F_{\psi}	&:   \HdoO \to \R, \qquad	F_{\psi}(\phi) := \int_{\O}  \fb_{\psi} \cdot \curl \phi, \label{Fstream:cont}\\
	F_{\theta}	&:  \HuoO \to \R, \qquad	F_{\theta}(v) := \int_{\O}  f_{\theta} v.\label{Ftemp:cont}	
\end{align}
We can observe by a direct computation that the trilinear form $B_T(\cdot;\cdot,\cdot)$  defined  in \eqref{BT-cont} is skew-symmetric, i.e.,
\begin{equation*}
B_T(\varphi;v,w) = - B_T(\varphi;w,v) \qquad  \forall \varphi \in  \HdoO \quad \text{and} \quad \forall v,w \in  \H_0^1(\O). 
\end{equation*}
Therefore, the bilinear form $B_T(\cdot;\cdot,\cdot)$ is equal to its skew-symmetric part, defined by 
\begin{equation}\label{BT-cont-skew}
B_{\skew}(\varphi;v,w) := \frac{1}{2}(B_{T}(\varphi;v,w) - B_{T}(\varphi;w,v)) \qquad  \forall \varphi \in  \HdoO \quad \text{and} \quad \forall v,w \in  \H_0^1(\O).
\end{equation}
According with the above discussion, we rewrite system \eqref{coupled-stream1}  in the following equivalent 
formulation: given the initial conditions $(\psi_0,\theta_0) \in  \H_0^1(\O)\times \L^2(\O)$  
and  the forces $\fb_{\psi} \in \L^{2}(0,T;\LL^{2}(\O)),
f_{\theta} \in \L^{2}(0,T;\L^{2}(\O))$ and $\gb \in \L^{\infty}(0,T;\LL^{\infty}(\O))$,
find $(\psi,\theta)\in \L^2(0,T;\H_0^2(\O))  \times\L^2(0,T;\H_0^1(\O))$ such that
\begin{equation}\label{BSE-stream2}
\begin{split}
M_F(\partial_t \psi, \phi) +\nu A_F(\psi,\phi) +B_F(\psi;\psi,\phi)	- C(\theta, \phi) &=F_{\psi}(\phi) \qquad \forall \phi \in \HdoO, \quad \text{for a.e.  $t \in (0,T)$},\\
M_T(\partial_t\theta,v)   + \kappa A_T(\theta,v) +B_{\skew}(\psi;\theta,v)&=F_{\theta}(v)  \:\qquad \forall v \in \HuoO, \quad \text{for a.e.  $t \in (0,T)$},\\
\psi(0)=\psi_0, \quad \qquad  \theta(0)&=\theta_0. 
\end{split}
\end{equation} 

\subsection{Well-posedness of the weak formulation}

In this subsection we recall some basic properties of the continuous forms and 
the existence and uniqueness properties of the solution to problem~\eqref{BSE-stream2}. 
	
\begin{lemma}\label{ACBF-bound}
For all $\zeta,\varphi,\phi \in \H_0^2(\O)$ and for each $v,w \in \H_0^1(\O)$, the forms defined in \eqref{MF-cont}-\eqref{BT-cont-skew} satisfy the following properties:
\begin{align*}
|M_F(\varphi,\phi)| & \leq  \,  C_{M_F}\|\varphi\|_{2,\O} \|\phi\|_{2,\O} \qquad  \text{and} \qquad
M_F(\phi,\phi) \ge |\phi|_{1,\O}^2,\\ %\alpha_{M_F}		
			|M_T(v,w)| & \leq  \,  C_{M_T}\|v\|_{1,\O} \|w\|_{1,\O} \qquad \text{and} \qquad
			M_T(v,v) \ge \|v\|_{0,\O}^2,\\ % \alpha_{M_T}
			|A_F(\varphi,\phi)| & \leq  \,  C_{A_F}\|\varphi\|_{2,\O} \|\phi\|_{2,\O} \qquad \text{and} \qquad
			A_F(\phi,\phi) \ge\alpha_{A_F}\|\phi\|_{2,\O}^2,\\
			|A_T(v,w)| & \leq  \,  C_{A_T}\|v\|_{1,\O} \|w\|_{1,\O} \qquad \text{and} \qquad
			A_T(v,v) \ge \alpha_{A_T}\|v\|_{1,\O}^2,\\
			|B_F(\zeta;\varphi,\phi)| &\leq \CBF \, \|\zeta\|_{2,\O} \|\varphi\|_{2,\O} \|\phi\|_{2,\O} \qquad \text{and} \qquad B(\zeta;\phi,\phi)= 0,\\
			|B_{\skew}(\zeta;v,w)| &\leq \CBT \, \|\zeta\|_{2,\O} \|v\|_{1,\O} \|w\|_{1,\O} 
			\qquad \text{and} \qquad B_{\skew}(\zeta;v,v)= 0,\\
			|C(v,\phi)|&\leq \|\gb\|_{\infty,\O}\|v\|_{0,\O}\|\phi\|_{1,\O}, \quad  |F_{\psi}(\phi)|  \leq  C_{F_{\psi}} \|\fb_{\psi}\|_{0,\O} \|\phi\|_{1,\O}, \quad |F_{\theta}(v)| \leq  C_{F_{\theta}} \|f_{\theta}\|_{0,\O} \|v\|_{0,\O}.
		\end{align*}
\end{lemma}

The equivalence between the (weak form of) problem \eqref{Bouss:Eq} and its stream formulation \eqref{BSE-stream2} is well known and easy to check. The couple $(\psi,\theta)$ satisfies \eqref{BSE-stream2} if and only if it exists a unique $p$ such that the triple $(\bu, \theta, p)$ in 
$\L^2(0,T;\HH_0^1(\O))  \times \L^2(0,T;\H_0^1(\O)) \times \L^2(0,T;L_0^2(\O))$ 
solves (the variational formulation of) \eqref{Bouss:Eq}, where $\bu = \curl \psi$.
Therefore the following well-posedness results for problem \eqref{BSE-stream2} follow immediately from known results for \eqref{Bouss:Eq}, see~\cite{Morimoto92}.

\begin{theorem}\label{esti:psi:theta}
Problem~\eqref{BSE-stream2} admits a unique solution $(\psi,\theta)$, satisfying
 $\psi \in \L^2(0,T;\H_0^2(\O)) \cap\L^{\infty}(0,T;\H_0^1(\O))$ and $\theta \in \L^2(0,T;\H_0^1(\O)) \cap\L^{\infty}(0,T;\L^2(\O))$.
Furthermore there exists a positive constant $C$, such that
\begin{equation*}
\begin{split}
\|\psi\|_{\L^{\infty}(0,T;\H_0^1(\O))}+ \|\psi\|_{\L^{2}(0,T;\H_0^2(\O))} &+\|\theta \|_{\L^{\infty}(0,T;\L^2(\O))}+ \|\theta\|_{\L^{2}(0,T;\H_0^1(\O))} \\
&\leq 
C\left( \|\fb_{\psi}\|_{\L^{2}(0,T;\LL^2(\O))}
+\|f_{\theta}\|_{\L^{2}(0,T;\L^2(\O))}+\|\theta_0 \|_{0,\O}+|\psi_0|_{1,\O}\right).
\end{split}
\end{equation*}	 
\end{theorem}

We close this section by recalling a useful Sobolev inequality (\cite[Lemma 2.2]{AHMS2018}), needed in the sequel:
\begin{equation}\label{sobolev:ineq}
	\|\bv\|_{\L^4(\O)} \leq 2^{\frac{1}{4}} \|\bv\|^{\frac{1}{2}}_{1,\O}\|\bv\|^{\frac{1}{2}}_{0,\O} 
	\qquad \forall \bv \in  \HH^1_0(\O).
\end{equation}

%_________________________________________________________________________________
 \setcounter{equation}{0}
\section{Virtual elements discretization}\label{VEM:section}
%_________________________________________________________________________________
In this section we will introduce $C^1$- and $C^0$-conforming schemes
of arbitrary order $k \geq 2$ and $\ell\geq 1$, 
for the numerical approximation of the stream-function and temperature 
unknowns of problem~\eqref{BSE-stream2}, respectively.
First, we start by introducing some mesh notations  
together with the respective local and global virtual spaces and their 
degrees of freedom. Moreover, we introduce the classical VEM polynomial 
projections and present the discrete multilinear forms. 
\subsection{Polygonal decompositions and  notations}
\label{meshassump}
Henceforth, we will denote by $\E$ a general polygon, $e$ a general edge of $\partial\E$, $h_\E$ 
the diameter of the element  $\E$ and by $h_e$ the length of edge. 
Let $\{\CT_h\}_{h>0}$ be a sequence of decompositions of $\O$ into non-overlapping polygons $\E$, 
where $h:=\max_{\E\in\CT_h}h_\E$. 

Moreover,  $N_\E$ denotes the number of vertices of 
$\E$  and we define the unit normal vector $\bn_{\E}$, that points outside of $\E$  
and the unit tangent vector $\bt_{\E}$ to $\E$ obtained by a counterclockwise rotation of $\bn_{\E}$. Furthermore, for each $\E$ and any integer $n\geq 0$, we introduce the following spaces:
	\begin{itemize}
\item For every open bounded subdomain $\mathcal{D} \subset\R^2$ we define  
$\P_{n}(\mathcal{D})$ as the space of polynomials on $\mathcal{D}$  of degree up to 
$n$ and we denote by $\PP_{n}(\calD)$ 
its vectorial version, i.e.,  $\PP_{n}(\calD):=[\P_{n}(\calD)]^2$; 
\item  We define the discontinuous piecewise $n$-order polynomial by 
\begin{equation*}
\P_{n}(\O_h):= \big\{ q \in \L^2(\O): q|_{\E} \in \P_{n}(\E)  \quad \forall \E \in \O_h \big\}.
\end{equation*}
\end{itemize}

Besides, for $s>0$, we consider the broken spaces
	\begin{equation*}
		\H^{s}(\O_h):= \big\{ \phi \in \L^2(\O): \phi|_{\E} \in \H^{s}(\E)  \quad \forall \E \in \O_h \big\} 
	\end{equation*}
endowed with the following broken seminorm: $|\phi|_{s,h}:=\Big(\sum_{\E\in\CT_h}|\phi|_{s,\E}^{2}\Big)^{1/2}$.

For the theoretical convergence analysis, we suppose that for all $h$, each element $\E$ in the mesh family $\{\CT_h \}_{h>0}$ satisfies the following assumptions~\cite{BBCMMR2013,ChM-camwa} for a uniform constant $\rho>0$:	
\begin{itemize}
\item[${\bf A1}:$] $\E$ is star-shaped with respect to every point of a  ball
		of radius  greater or equal to $\rho h_\E$;
		\item[${\bf A2}:$] every  edge $e \subset \partial \E$ has the length greater or equal to $ \rho h_\E$.
	\end{itemize}

\subsection{Virtual element space for the stream-function}
\label{stream:spaces}
In the present section we introduce  a virtual space of order $k$ used to approximate the stream-function unknown.
 
For each polygon $\E\in\O_h$  and every integer $k\ge 2$, let $\widehat{k}:=\max\{k,3\}$ 
and $\WtK$  be the finite  dimensional space introduced  in~\cite{ChM-camwa}:
\begin{align*}
\WtK:=\left\{\phi_h\in \HdoK : \Delta^2\phi_h\in\P_{k-2}(\E), \phi_h|_{\partial\E}\in C^0(\partial\E),
		\phi_h|_e\in\P_{\widehat{k}}(e)\,\,\forall e\in\partial\E,\right.\\
		\left.\nabla \phi_h|_{\partial\E}\in \Cb^0(\partial\E),
		\partial_{\nb^e_{\E}}\phi_h \in\P_{k-1}(e)\,\,\forall e\in\partial\E\right\}.
	\end{align*}
Next, for  $\phi_{h}\in\WtK$, we introduce the following set of linear operators:
\begin{itemize}
		\item $\DXu:$ the values of $\phi_{h}(\vb_i)$, for all  vertex $\vb_i$ of the polygon $\E$;
		\item $\DXd:$ the values of $h_{\vb_i}\nabla\phi_{h}(\vb_i)$, for all  vertex $\vb_i$ of the polygon $\E$;
		\item $\DXt:$ for $k \geq 3$, the  moments on edges   up to degree $k-3$:
		$$
		\quad (q, \partial_{\nb^e_{\E}}\phi_h)_{0,e}  \qquad \forall q \in \M_{k-3}(e), 
		\quad \forall \, \text{edge} \, \, e;
		$$
		\item $\DXc:$ for $\widehat{k}\geq 4$, the moments on edges  up to degree $\widehat{k}-4$: 
		$$h^{-1}_e(q,\phi_h)_{0,e} \qquad \qquad \forall q \in \M_{\widehat{k}-4}(e), 
		\quad \forall \, \text{edge} \, \, e;$$
		
		\item $\DXf:$ for $k \geq 4$, the moments on polygons   up to degree $k-4$:
		$$ \qquad h^{-2}_\E(q,  \phi_h)_{0,\E}\,
		\qquad \qquad \forall q \in \M_{k-4}(\E), 
		\quad \forall \, \text{polygon} \, \, \E,$$
\end{itemize}
where for each vertex $\vb_i$, we chose $h_{\vb_i}$ as the average of the diameters of the elements having $\vb_i$ 
as a vertex and $\M_{n}(\E)$ denote the scaled monomials of degree $n$, for each $n\geq0$ (for further details see~\cite{BM13}). 

In order to construct  an approximation for the bilinear form  $A_F(\cdot,\cdot)$, we define the operator
$
{\tt P_0}: C^{0}(\partial \E) \longrightarrow \P_0(\E)
$
defined by 
	as the following average: 
	\begin{equation}\label{average}
		{\tt P_0}\phi_h = \frac{1}{N_\E} \sum_{i=1}^{N_\E} \phi_h (\vb_i),
	\end{equation}
	where $\vb_i, 1 \leq i \leq N_\E$, are the vertices of  $\E$. Then, for each polygon $\E$,  we define the  projector: 
	\begin{equation*}
		\PiK:\WtK  \longrightarrow \P_k(\E)\subseteq\WtK,
	\end{equation*}
	as the solution of the local problems:
	\begin{align*}
		A^{\E}_{F}(\phi_h- \PiK \phi_h&, q_k)  = 0 \qquad  \forall q_k \in \P_{k}(\E), \\
		{\tt P_0}(\phi_h-\PiK \phi_h)&=0, \quad 
		{\tt P_0}(\nabla (\phi -\PiK \phi_h))=0, 
	\end{align*}
where $A^{\E}_{F}(\cdot,\cdot)$ is the restriction of the  global bilinear form 
$A_{F}(\cdot,\cdot)$ (cf. \eqref{AF-cont}) on each polygon $\E$.
  	
	\begin{remark}\label{lemma-ProyDelta}
		The operator $\PiK :\WtK \to \P_{k}(\E)$ is explicitly
		computable for every $\phi_h \in \WtK$, using only the
		information of the linear operators $\DXu-\DXf$; see for instance \cite{ChM-camwa,MRS2021-IMAJNA}.	
	\end{remark} 
	
Now, we will present the local stream-function virtual space. 
For any $\E \in \O_h$ and  each integer $k\ge 2$, we consider 
the following local enhanced virtual space
% the local  enhanced virtual space is given by: 
	\begin{align}\label{localspace}
		\WK 
		:= \left \{ \phi_h \in \WtK : 
		(q^{\ast} \,, \phi_h-\PiK \phi_h)_{0,\E}=0 \quad \forall  q^{\ast} \in \M^{\ast}_{k-3}(\E) \cup \M^{\ast}_{k-2}(\E)\right \},
	\end{align}
	where $\M^{\ast}_{k-3}(\E)$ and $\M^{\ast}_{k-2}(\E)$ are scaled monomials of degree $k-3$ 
	and $k-2$, respectively  (see \cite{AABMR13}),  with  the convention that 
	$\M^{\ast}_{-1}(\E):=\emptyset$.  For further details,
	see  for instance \cite{ChM-camwa} (see also \cite{BM13,ABSV2016,MRS2021-IMAJNA}).	
	
For $k\ge2$, we  introduce an additional projector,
which will be used to build an approximation of the bilinear form $M_F(\cdot,\cdot)$.
Such projector $\PicurlK:\WtK \longrightarrow \P_k(\E)\subseteq\WtK$ is defined
as the solution of the local problems:
\begin{align*}
	M^{\E}_{F}(\phi_h- \PicurlK \phi_h, q_k)  &= 0 \qquad  \forall q_k \in \P_{k}(\E), \\
	{\tt P_0}(\nabla (\phi_h -\PicurlK \phi_h))&=0,
\end{align*}
where $M^{\E}_{F}(\cdot,\cdot)$ is the restriction of the  global bilinear form 
$M_{F}(\cdot,\cdot)$ (cf. \eqref{MF-cont}) on each polygon $\E$.

We summarize the main properties of
the local virtual space $\WK$ defined in \eqref{localspace}
(for the proof, we refer to \cite{AABMR13,BM13,ChM-camwa,MRS2021-IMAJNA}). 

	\begin{itemize}
		\item $\P_k(\E) \subset  \WK \subset \WtK$;
		\item The sets of linear operators $\DXu-\DXf$ constitutes a set of degrees of freedom for $\WK$;
		\item The operators $\PiK :\WK \to \P_{k}(\E)$ and  $\PicurlK :\WK \to \P_{k}(\E)$ are 
		computable using only the degrees of freedom $\DXu-\DXf$.
	\end{itemize}

Now, we present our global virtual space to approximate the stream-function of 
the Boussinesq system~\eqref{BSE-stream2}. For each decomposition $\O_h$ of 
$\O$ into simple polygons $\E$, we define 
\begin{equation*}
	\Wh:=\left\{ \phi_h \in \HdoO: \phi_h|_{\E} \in \WK \quad \forall \E \in \O_h \right\}. 
\end{equation*}

\subsection{Virtual element space for the temperature}
\label{temp:spaces}	
In this subsection we will introduce a $C^0$-virtual element space of high order $\ell\geq 1$ 
to approximate the temperature field of problem \eqref{BSE-stream2}. 
To this end, for each polygon $\E \in \CT_h$, we consider the following finite 
dimensional space (see \cite{AABMR13,BBMRm3as2016,CMS2016}):  
\begin{align*}
\HtE:= \left\{ w_h \in \H^1(\E) \cap C^{0}(\partial \E) : {\Delta w_h} \in \P_{\ell}(\E), \, \, w_h|_{e}
\in \P_{\ell}(e) \quad \forall e \subset \partial \E  \right\}.
\end{align*}
For each $q_h \in \HtE$ we consider the following set of linear operators:
\begin{itemize}
\item  $\DHu:$  the values of   $w_h(\vb_i)$, for all vertex $\vb_i$ of the polygon $\E$. 
\item  $\DHd:$ for $\ell \geq 2$, the moments on edges up to degree $\ell-2$:  
\begin{equation*}
h^{-1}_e(w_h\, p_{\ell})_{0,e}  \qquad  \forall p_{\ell} \in \M_{\ell-2}(e), 	\quad \forall \, \text{edge} \, \, e;
\end{equation*} 
\item  $\DHt:$ for $\ell\geq 2$, the moments on element $\E$ up to degree $\ell-2$: 
\begin{equation*}
h^{-2}_{\E}(w_h\, p_{\ell})_{0,\E}  \qquad  \forall p_{\ell} \in \M_{\ell-2}(E), \quad \forall \, \text{polygon} \, \, \E,
		\end{equation*}
	\end{itemize}
where $\M_{n}(\E)$ denote the scaled monomials of degree $n$, for each $n\geq0$ (for further details see~\cite{AABMR13,CMS2016}).	
	Now, we define the projector $\PinablaK:\HtE \to\P_{\ell}(\E)\subseteq \HtE$,
	 as the solution of the local problems
	\begin{align*}
		A_T^{E}(w_h-\PinablaK  w_h,r_{\ell}) &=0 
		\qquad\forall r_{\ell}\in\P_{\ell}(\E),\\
		{\tt P}_0(\PinablaK w_h-w_h)&= 0,
	\end{align*}
where $A^{\E}_{T}(\cdot,\cdot)$ is the restriction of the  global bilinear form 
$A_{T}(\cdot,\cdot)$ (cf. \eqref{AT-cont}) on each polygon $\E$ and
the operator ${\tt P}_0(\cdot)$ is defined in \eqref{average}. 
We have that the operator $\PinablaK:\HtE \to\P_{\ell}(\E)$ is computable 
using the set $\DHu-\DHt$ (see for instance, \cite{AABMR13,BBMRm3as2016,CMS2016}). 
In addition,  by using this projection and the definition of space $\HtE$, 
we introduce our local virtual space to approximate the temperature field: 
\begin{equation*}
\HE:= \left\{  w_h \in \HtE:   (r^{\ast},w_h-\PinablaK w_h)_{0,\E}= 0 \qquad \forall r^{\ast} \in  \M^{\ast}_{\ell}(\E) \cup \M^{\ast}_{\ell-1}(\E) \right\},
\end{equation*}
where $\M^{\ast}_{\ell}(\E)$ and $\M^{\ast}_{\ell-1}(\E)$ are scaled monomials of degree $\ell$ and $\ell-1$, respectively,  with  the convention that $\M^{\ast}_{-1}(\E):=\emptyset$ (see \cite{AABMR13,CMS2016}).
	
Now, we summarize the main properties of the local virtual spaces $\HE$	(for a proof we refer  to \cite{AABMR13,BBMRm3as2016,CMS2016}):
	\begin{itemize}
		\item $\P_{\ell}(\E) \subset \HE \subset \HtE$;
		\item The sets of linear operators $\DHu-\DHt$ constitutes a set of degrees of freedom
		for $\HE$;
		\item The operator $\PinablaK :\HE \to \P_{\ell}(\E)$ is also 
		computable using the degrees of freedom $\DHu-\DHt$.
	\end{itemize}
	
Next, we present our global virtual space to approximate the  fluid temperature of the Boussinesq system~\eqref{BSE-stream2}.
For each decomposition $\O_h$ of $\O$ into simple polygons $\E$, we define 
	\begin{equation*}
		\Hh:=\left\{ w_h \in \HuoO: w_h|_{\E} \in \HE \quad \forall \E \in \O_h \right\}. 
	\end{equation*}

\subsection{$\L^2$-projections and the discrete forms}
\label{projectors:dicrete:form}
	
In this subsection we introduce some functions built from the classical 
$\L^2$-polynomial projections,  which  will be useful to construct an 
approximation of the continuous multilinear forms defined in Section~\ref{weak_form_section}.
We start recalling  the usual $\L^2(\E)$-projection onto the scalar polynomial space 
	$\P_{n}(\E)$, with $n \in \N \cup \{0\}$: 
	for each $\phi \in \L^2(\E)$, the function $\Pi_{\E}^{n} \phi \in \P_{n}(\E)$ 
	is defined as the unique function, such that 
	\begin{equation}\label{proymdos}
		\left( q_{n}, \phi-\Pi_{\E}^{n} \phi \right)_{0,\E}=0 
		\qquad\forall q_{n}\in\P_{n}(\E).
	\end{equation}

An analogous definition holds for the  $\LL^2(\E)$-projection onto the vectorial polynomial space 
$\PP_{n}(\E)$,  which we will denote by $\boldsymbol{\Pi}_{\E}^{n}$.  

The following lemma establishes that certain polynomial projections  
are computable on $\WK$, using only the information of the 
degrees of freedom $\DXu-\DXf$ (see for instance \cite{ChM-camwa,MRS2021-IMAJNA}).
	
\begin{lemma}\label{lemm-PiK-stream}
For $k \geq 2$, let $\PioK:\L^2(\E)\to\P_{k-2}(\E)$ and $\PimunoK: \LL^2(\E)\to \PP_{k-1}(\E)$ be the operators defined
by the relation \eqref{proymdos} and  by its vectorial version. 
Then,  for each $\phi_h \in \WK$ the polynomial functions
\[
\PioK \phi_h, \quad \: 
\PioK \Delta \phi_h, \quad \:  \PimunoK\nabla \phi_h \quad \text{and} \quad \PimunoK\curl \phi_h
\]
are computable using only the information of the degrees of freedom $\DXu-\DXf$.
\end{lemma}

For the space $\HE$  and its degrees of freedom $\DHu-\DHt$, we have the following result  
(see for instance \cite{BBMRm3as2016,CMS2016}).
\begin{lemma}\label{lemm-PiK-temp}
For $\ell \geq 1$, let $\Pi_{\E}^{\ell-1}:\L^2(\E)\to\P_{\ell-1}(\E)$, $\Pi_{\E}^{\ell}:\L^2(\E)\to\P_{\ell}(\E)$ and $\PiellK: \LL^2(\E)\to \PP_{\ell-1}(\E)$ 
be the operators defined by the relation \eqref{proymdos} and  by its vectorial version, respectively. 
Then,  for each $w_h \in \HE$ the polynomial functions
\[\Pi_{\E}^{\ell-1} w_h, \quad \: \Pi_{\E}^{\ell} w_h \quad \:  \text{and} \quad \PiellK\nabla w_h
\]
are computable using only the information of the degrees of freedom $\DHu-\DHt$.
\end{lemma}
		
Now, using the functions introduced above, we will construct the discrete version
of the forms defined in Section~\ref{weak_form_section}. 
First, let  $s^{\bc}_{\E}: \WK \times \WK \to \R$ and $s^{\D}_{\E} : \WK \times \WK \to \R$
 be any symmetric positive definite 
bilinear forms to be chosen  to satisfy:
	\begin{equation}\label{term:stab:SK}
\begin{split}
c_0 M_F^{\E}(\phi_h,\phi_h)&\leq s_{\E}^{\bc}(\phi_h,\phi_h)\leq c_1 M_F^{\E}(\phi_h,\phi_h)\qquad \quad
\forall \phi_h \in  \mathrm{Ker}(\PicurlK),\\
c_2 A_F^{\E}(\phi_h,\phi_h)&\leq s^{\D}_{\E}(\phi_h,\phi_h)\leq c_3 A_F^{\E}(\phi_h,\phi_h)\qquad \quad \:
\forall \phi_h \in  \mathrm{Ker}(\PiK),
\end{split}
%\label{stab:projector}
\end{equation}
with $c_0,c_1,c_2$ and $c_3$ are positive constants independent
of $h$ and $\E$. We will choose the following representation satisfying \eqref{term:stab:SK} 
 (see \cite[Proposition 3.5]{MRS2021-IMAJNA}):
	\begin{equation*}
s^{\D}_{\E}(\varphi_h,\phi_h):= h^{-2}_{\E} \sum_{i=1}^{N_\E^{\dof}}\dof^{\WK}_i(\varphi_h)\dof^{\WK}_i(\phi_h) \quad
\text{and} \quad 
s^{\bc}_{\E}(\varphi_h,\phi_h):=  \sum_{i=1}^{N_\E^{\dof}} \dof^{\WK}_i(\varphi_h)\dof^{\WK}_i(\phi_h),
	\end{equation*}
where $N_\E^{\dof}:=\dim(\WK)$ and the operator $\dof^{\WK}_j(\phi)$ associates to each  smooth enough function 
$\phi$  the $j$th local degree of freedom $\dof^{\WK}_j(\phi)$, with $1 \leq j \leq N_\E^{\dof}$.

On each polygon $\E$, we define the local discrete bilinear forms $M_{F}^{h,\E}(\cdot, \cdot)$ and 
$A_{F}^{h,\E}(\cdot,\cdot)$ as follows
\begin{align}	
M_{F}^{h,\E}(\varphi_h,\phi_h)&:=
M_{F}^{\E} \left(\PicurlK \varphi_h,\PicurlK \phi_h\right) +s^{\bc}_{\E}\big(({\rm I}-\PicurlK) \varphi_h,({\rm I}-\PicurlK) \phi_h\big)  \qquad \quad\forall  \varphi_h,\phi_h\in\WK, \label{localdisc:MF}\\
A_{F}^{h,\E}(\varphi_h,\phi_h)&:=
A_{F}^{\E} \left(\PiK \varphi_h,\PiK \phi_h\right) +s^{\D}_{\E}\big(({\rm I}-\PiK) \varphi_h,({\rm I}-\PiK) \phi_h\big) \qquad \: \forall  \varphi_h,\phi_h\in\WK. \label{localdisc:AF}
\end{align}

For the approximation of the local trilinear form $B_{F}^{\E}(\cdot; \cdot, \cdot)$, we consider set   
\begin{equation}\label{disc-formB}
B_{F}^{h,\E}(\zeta_h; \varphi_h,\phi_h)
:=\int_{\E}\left[ \left(\PioK \Delta \zeta_h \right)  \left(\PimunoK \curl \varphi_h  \right)\right] \cdot  \PimunoK \nabla \phi_h \qquad  \forall \zeta_h, \varphi_h,\phi_h\in\WK.
\end{equation}

For the treatment of the right-hand side associate to the fluid equation, 
we set the following local load term:
\begin{equation*}
	F_{\psi}^{h,\E}(\phi_h)= \int_{\E} \PimunoK \fb_{\psi}(t) \cdot \curl \phi_h \equiv \int_{\E} \fb_{\psi}(t) \cdot\PimunoK  \curl \phi_h \qquad \forall \phi_h \in \WK, \quad \text{for a.e.  $t \in (0,T)$}. 
\end{equation*}

Thus, for all  $ \zeta_h, \varphi_h,\phi_h\in\Wh$, we define the associated global forms 
$M_{F}^h, A_{F}^h, B_{F}^h, F_{\psi}^h$ in the usual way, by summing the local forms on all mesh elements. 
For instance
$$
M_{F}^h : \Wh \times  \Wh \to \R, \qquad 
M_{F}^{h}(\varphi_h, \phi_h) :=\sum_{\E \in \CT_h} M_{F}^{h,\E} (\varphi_h, \phi_h).
$$

We recall that the forms defined above are computable using the degrees of freedom $\DXu-\DXf$.
In addition, we have that the  trilinear form $B_F^h(\cdot;\cdot,\cdot)$ is immediately extendable to the whole $\HdoO$.

The following result establishes the usual $k$-consistency
and stability properties for the discrete local forms 
$M_{F}^{h,\E}(\cdot,\cdot)$ and $A_{F}^{h,\E}(\cdot,\cdot)$. 
\begin{proposition}\label{cons-stab-forms}
The local bilinear forms    defined in \eqref{MF-cont}, \eqref{AF-cont}, \eqref{localdisc:MF} and 
\eqref{localdisc:AF},  satisfy the following properties:
		\begin{itemize}
	\item \textit{$k$-consistency}: for all $\E\in\CT_h$, we have that
\begin{equation*}
M_{F}^{h,\E}(q,\phi_h)=M_{F}^{\E}(q,\phi_h), \qquad \qquad
A_{F}^{h,\E}(q,\phi_h)
=A_{F}^{\E}(q,\phi_h) \qquad \forall q\in\P_k(\E), \quad\forall \phi_h\in\WK.%\label{consis-AF}
	\end{equation*}

\item \textit{stability and boundedness}: there exist positive constants
$\alpha_i, i=1,\ldots,4,$ independent of $\E$, such that:
\begin{align*}
\alpha_1 M_{F}^{\E}(\phi_h,\phi_h)&\leq M_{F}^{h,\E}(\phi_h,\phi_h)
\leq\alpha_2 M_{F}^{\E}(\phi_h,\phi_h)&\forall \phi_h\in\WK,\\%\label{stab-MF}\\
\alpha_3 A_{F}^{\E}(\phi_h,\phi_h)&\leq A_{F}^{h,\E}(\phi_h,\phi_h)
\leq\alpha_4 A_{F}^{\E}(\phi_h,\phi_h)
&\forall \phi_h\in\WK.%\label{stab-AF}
\end{align*}		
\end{itemize}
\end{proposition}
\begin{proof}
The proof follows standard arguments in the VEM literature (see \cite{ABSV2016,BBCMMR2013,BBMRm3as2016}).
\end{proof}
		
Now, we continue with the construction of the forms associated to the energy equation. 
First, let $s^0_{\E}(\cdot, \cdot)$ and $s^{\nabla}_{\E}(\cdot, \cdot)$ be any symmetric positive definite bilinear forms such that 
\begin{equation}\label{term:stab:SKTemp}
\begin{split}
c_4 M_T^{\E}(v_h,v_h)&\leq s^{0}_{\E}(v_h,v_h)\leq c_5 M_T^{\E}(v_h,v_h)\qquad \qquad
\forall v_h \in  \mathrm{Ker}(\Pi_{\E}^{\ell}),\\
c_6 A_T^{\E}(v_h,v_h)&\leq s_{\E}^{\nabla}(v_h,v_h)\leq c_7 A_T^{\E}(v_h,v_h)
\qquad \qquad \forall v_h \in  \mathrm{Ker}(\PinablaK),
	\end{split}
\end{equation}
for some positive constants $c_4$, $c_5$, $c_6$ and $c_7$, independent of $h$ and $\E$. 
We will choose the classical representation for these stabilizing forms satisfying  
property \eqref{term:stab:SKTemp}   (see \cite{BLR-M3AS,BS18-M3AS,CMS2016}):
\begin{equation*}
s^{0}_{\E}(v_h,w_h):=  h_{\E}^2\sum_{j=1}^{\dim(\HE)} \dof^{\HE}_j(v_h)\dof^{\HE}_j(w_h),  
\qquad 
s^{\nabla}_{\E}(v_h,w_h):=  \sum_{j=1}^{\dim(\HE)} \dof^{\HE}_j(v_h)\dof^{\HE}_j(w_h),  
	\end{equation*}
	where the operator $\dof^{\HE}_j(v)$ associates to each  smooth enough function 
	$v$  the $j$th local degree of freedom $\dof^{\HE}_j(v)$, with $1 \leq j \leq \dim(\HE)$. 
	Then, we set the following approximation for the forms $M_{T}^{\E}(\cdot,\cdot)$
and $A_{T}^{\E}(\cdot,\cdot)$ (cf. \eqref{MT-cont} and \eqref{AT-cont})
\begin{align*}	
M_{T}^{h,\E}(v_h,w_h)&:=
M_{T}^{\E}  \left(\Pi_{\E}^{\ell} v_h,\Pi_{\E}^{\ell} w_h\right) +s^{0}_{\E}\big(({\rm I}-\Pi_{\E}^{\ell}) v_h,({\rm I}-\Pi_{\E}^{\ell}) w_h\big)  &\forall  v_h,w_h\in\HE, \\ 
A_{T}^{h,\E}(v_h,w_h)&:= \int_{\E}  \PiellK\nabla v_h \cdot \PiellK\nabla w_h +s^{\nabla}_{\E}\big(({\rm I}-\PinablaK) v_h,({\rm I}-\PinablaK) w_h\big)  & \forall  v_h,w_h\in\HE. 
\end{align*}	

We have that the bilinear forms $M_{T}^{h,\E}(\cdot,\cdot)$ and $A_{T}^{h,\E}(\cdot,\cdot)$ satisfy the classical  $\ell$-consistency and stability properties (analogous to Proposition~\eqref{cons-stab-forms}). For further details,
see \cite{BBCMMR2013,BBMRm3as2016,CMS2016}.

To approximate of bilinear form $C^{\E}(\cdot,\cdot)$, we set
\begin{equation*}
C^{h,\E}(w_h,\phi_h):= \int_{\E} \gb \Pi^{\ell-1}_{\E} w_h\cdot \PimunoK \curl \phi_h	 
\qquad \forall w_h \in \HE,  \: \forall\phi_h \in \WK. 
\end{equation*} 

Now, we consider the following discrete trilinear form %$B_T^{h,\E} (\cdot;\cdot,\cdot)$
\begin{equation*}
B_T^{h,\E} (\varphi_h;v_h,w_h) := \int_{\E}  \left( \PimunoK \curl \varphi_h \cdot \PiellK \nabla v_h  \right) \Pi^{\ell-1}_{\E} w_h
\qquad \forall \varphi_h \in  \WK, \: \forall w_h, v_h \in \HE. 
\end{equation*}
Then, for the skew-symmetric trilinear form $B_{\skew}^{\E} (\cdot;\cdot,\cdot)$ (cf. \eqref{BT-cont-skew}),
we set the following approximation:
\begin{equation*}
B_{\skew}^{h,\E}(\varphi_h;v_h,w_h) := \frac{1}{2} (B_T^{h,\E} (\varphi_h;v_h,w_h) - B_T^{h,\E} (\varphi_h;w_h,v_h)) 
\qquad \forall  \varphi_h \in  \WK, \: \forall w_h, v_h \in \HE. 
\end{equation*}

For the treatment of the right-hand side associated to the temperature discretization, 
we set following local load term
\begin{equation*}
	F_{\theta}^{h,\E}(v_h):= \int_{\E} \Pi_{\E}^{\ell-1} f_{\theta}(t) v_h \equiv \int_{\E} f_{\theta}(t)\Pi_{\E}^{\ell-1}  v_h \qquad \forall v_h \in \HE  \quad \text{for a.e.  $t\in (0,T)$}. 
\end{equation*}	
	
Thus, for all $ \zeta_h \in \Wh$ and for all $v_h,w_h\in\Hh$, we define the associated global forms 
$M_{T}^h, C^h, B_{\skew}^h, F_{\theta}^h$ in the usual way, by summing the local forms on all mesh elements. 
For instance
$$
M_{T}^h : \Hh \times  \Hh \to \R, \qquad M_{T}^{h}(v_h, w_h) :=\sum_{\E \in \CT_h} M_{T}^{h,\E} (v_h, w_h).
$$

We finish this section summarizing  some properties of the discrete global forms  defined above.

\begin{lemma}\label{lemma:disc:form}
For each $\zeta_h,\varphi_h,\phi_h  \in \Wh$ and each $v_h,w_h \in \Hh$, the global forms defined above
satisfy the following properties:
\begin{align*}
|M^h_F(\varphi_h,\phi_h)| & \leq  \,  \widehat{C}_{M_F}\|\varphi_h\|_{1,\O} \|\phi_h\|_{1,\O} \qquad \text{and} \qquad
M^h_F(\phi_h,\phi_h) \ge \widehat{\alpha}_{M_F}\|\phi_h\|_{1,\O}^2,\\		
|M^h_T(v_h,w_h)| & \leq  \,  \widehat{C}_{M_T}\|v_h\|_{0,\O} \|w_h\|_{0,\O} \qquad \text{and} \qquad
M^h_T(v_h,v_h) \ge \widehat{\alpha}_{M_T}\|v_h\|_{0,\O}^2,\\
|A^h_F(\varphi_h,\phi_h)| & \leq  \,  \widehat{C}_{A_F}\|\varphi_h\|_{2,\O} \|\phi_h\|_{2,\O} \qquad \text{and} \qquad
A^h_F(\phi_h,\phi_h) \ge\widehat{\alpha}_{A_F}\|\phi_h\|_{2,\O}^2,\\
|A^h_T(v_h,w_h)| & \leq  \,  \widehat{C}_{A_T}\|v_h\|_{1,\O} \|w_h\|_{1,\O} \qquad \text{and} \qquad
A^h_T(v_h,v_h) \ge \widehat{\alpha}_{A_T}\|v_h\|_{1,\O}^2,\\
|B^h_F(\zeta_h;\varphi_h,\phi_h)| &\leq \CBFh \, \|\zeta_h\|_{2,\O} \|\varphi_h\|_{2,\O} \|\phi_h\|_{2,\O} 
\qquad \text{and} \qquad B^h_F(\zeta_h;\phi_h,\phi_h)= 0,\\
|B^h_{\skew}(\zeta;v_h,w_h)| &\leq \CBTh \, \|\zeta_h\|_{2,\O} \|v_h\|_{1,\O} \|w_h\|_{1,\O}\qquad \text{and} \qquad B^h_{\skew}(\zeta;v_h,v_h)= 0,\\
|C^h(v,\phi)|\leq \|\gb\|_{\infty,\O}&\|v\|_{0,\O}\|\phi\|_{1,\O}, \quad |F^h_{\psi}(\phi_h)|  \leq  \widehat{C}_{F_{\psi}} \|\fb_{\psi}\|_{0,\O} \|\phi_h\|_{1,\O} \quad \text{and} \quad |F^h_{\theta}(\phi)| \leq  \widehat{C}_{F_{\theta}} \|f_{\theta}\|_{0,\O} \|v_h\|_{0,\O},
	\end{align*}
where all the constants involved are positive and independent of mesh size $h$.
\end{lemma} 

\begin{remark}
If $\fb_{\psi}$ is given as an explicit function, then we can consider the following alternative
discrete  load term 
\begin{equation*}
F_{\psi}^{h}(\phi_h) :=	\sum_{\E \in \CT_h} \int_{\E} \rot \fb_{\psi}(t) \PioK \phi_h \qquad \forall \phi_h \in \Wh,   
\end{equation*} 
 which is also computable using the degrees of freedom ${\bf D_1}-{\bf D_5}$. 
\end{remark}

%%____________________________________________________________________________________	
\setcounter{equation}{0}
\section{Fully-discrete formulation and its well posedness}\label{fullydiscrete:BSE}	
%%____________________________________________________________________________________	

In order to present a full discretization of problem \eqref{BSE-stream2}
we introduce a sequence of time steps  $t_n = n \Delta t$, $n=0,1,2,\ldots,N$, where $\Delta t=T/N$ is the time step. 
Moreover, we consider the following approximations at each time $t_n$: 
$\psi_h^n \approx\psi_h(t_n)$ and $\theta_h^n\approx\theta_h(t_n)$. 
For the external forces, we introduce the following notation: 
$\fb_{\psi}^n:=\fb_{\psi}(t_n)$, $f^n_{\theta}:=f_{\theta}(t_n)$ and $\gb^n:=\gb(t_n)$. 
 
We consider the backward Euler method coupled with the VE 
discretization presented in Section~\ref{VEM:section}, which read as follows:  
given $(\psi_h^0,\theta_h^{0})$, find $\{(\psi_h^n,\theta_h^{n})\}_{n=1}^{N} \in \Wh \times \Hh$, such that
% find  $(\psi_h^n, \theta_h^n)\in \Wh \times \Hh$, such that
\begin{equation}\label{eq:VEM2}
\begin{split}
M^h_F \left( \frac{\psi_{h}^n-\psi_{h}^{n-1}}{\Delta t},\phi_h \right) + \nu A^h_F(\psi^n_h,\phi_h) +B^h_F(\psi^{n}_h;\psi^n_h,\phi_h)- C^h(\theta^{n}_h, \phi_h) &=F_{\psi}^{h}(\phi_h) \qquad \forall \phi_h \in \Wh, \\
M^h_T \left( \frac{\theta_{h}^n-\theta_{h}^{n-1}}{\Delta t},v_h \right)+ \kappa A^h_T(\theta^n_h,v_h) +B^h_{\skew}(\psi^{n}_h;\theta^{n}_h,v_h)&=F_{\theta}^{h}(v_h) \qquad \forall v_h \in \Hh.
\end{split}
\end{equation}

The functions $(\psi_h^0,\theta_h^{0})$  are initial approximations of 
$(\psi_h,\theta_h)$ at $t=0$. 
For instance, we will consider $\psi_h^0:=\mS_h \psi_0$ (see~\eqref{dis:engy:bilap}) and $\theta_h^0:=\mP_h \theta_0$, with $\mP_h(\cdot)$ being the energy operator associated to the 
$\H^1$-inner product (for further details, see for instance \cite{VB15}).

In what follows, we will provide the well-posedness of the fully-discrete formulation~\eqref{eq:VEM2}.
\begin{theorem}\label{existence}
Let $\widehat{\alpha}:= \min\left\{\aMFh,\aMTh\right\}$ and  $\gamma :=\min\left\{\aAFh\nu,\aATh \kappa \right\}$, where $\aMFh,\aMTh, \aAFh$  and $\aATh$ are the constants in Lemma~\ref{lemma:disc:form}. Assume that	
\begin{equation}\label{assump:existence}
\widehat{\alpha} +\Delta t\left(\gamma - C_{\gb} \right) >0,
\end{equation}
where $C_{\gb} := \|\gb\|_{\L^{\infty}(0,T;\LL^{\infty}(\O))}$.
Then the fully-discrete scheme~\eqref{eq:VEM2} admits at least one solution 
$(\psi_h^n,\theta^n_h) \in \Wh \times \Hh$ at every time step $t_n$, with $n=1, \ldots, N$.
\end{theorem}
 
\begin{proof}
For simplicity we set $\Xh:=\Wh \times \Hh$ and we endow this space  with the following equivalent norm:
\begin{equation*}
|||(\phi_h,w_h) |||:=( \|\phi_h\|^2_{1,\O}+\|w_h\|^2_{0,\O})^{\frac{1}{2}} \qquad \forall (\phi_h,w_h) \in \Xh.
\end{equation*}

Next, for $ 1 \leq n \leq N$, let  $(\psi^{n-1}_{h},\theta^{n-1}_{h}) \in \Xh$ and for any $(\psi_h,\theta_h) \in \Xh$, we consider the operator $\Phi: \Xh \to (\Xh)^{\ast}$ defined by
\begin{equation}\label{operator}
\begin{split}
\langle \Phi(\psi_h,\theta_h),(\phi_h,w_h)  
\rangle &:=M^h_F(\psi_h,\phi_h) - M^h_F(\psi_{h}^{n-1},\phi_h) 
+ \nu \Delta t A^h_F(\psi_h,\phi_h)+ \Delta t B^h_F(\psi_h;\psi_h,\phi_h)  \\
& \quad - \Delta t F_{\psi}^{h}(\phi_h) +M^h_T(\theta_h,w_h )  -M^h_T(\theta_{h}^{n-1},w_h)  + \kappa\Delta t A^h_T(\theta_h,w_h)\\
& \quad+\Delta t B^h_{\skew}(\psi_h;\theta_h,w_h) 
- \Delta t F_{\theta}^{h}(w_h)- \Delta tC^h(\theta_h, \phi_h) \qquad \forall (\phi_h,w_h) \in \Xh.
\end{split}
\end{equation}

By using the definition of operator $\Phi$ and Lemma~\ref{lemma:disc:form}, we easily have that
\begin{equation*}
\|\Phi(\psi_h,\theta_h)- \Phi(\psi^{\star}_h,\theta^{\star}_h)\|_{(\Xh)^{\ast}} \longrightarrow 0, \quad \text{when} \quad  (\psi_h,\theta_h)\xrightarrow[]{|||\cdot|||}	(\psi^{\star}_h,\theta^{\star}_h), 
\end{equation*}
i.e., $\Phi$ is continuous.

On the other hand,  employing again Lemma~\ref{lemma:disc:form} and the Young inequality, 
for all $(\psi_h,\theta_h) \in \Xh$, we obtain
\begin{equation*}
\begin{split}
\langle \Phi(\psi_h&,\theta_h),(\psi_h,\theta_h)  \rangle 
\geq\aMFh \|\psi_h\|^2_{1,\O}- \frac{\CMFh^2}{2\aMFh}\|\psi^{n-1}_h\|^2_{1,\O} 
- \frac{\aMFh}{2}\|\psi_h\|^2_{1,\O} + \aAFh \nu \Delta t \|\psi_h\|^2_{2,\O}
- \frac{\Delta t }{2\aAFh \nu} \|\fb^{n}_{\psi}\|^2_{0,\O}\\
&\quad - \frac{\aAFh \nu \Delta t}{2}\|\psi_h\|^2_{2,\O}+
\aMTh \|\theta_h\|^2_{0,\O}- \frac{\CMTh^2}{2\aMTh}\|\theta^{n-1}_h\|^2_{0,\O} 
- \frac{\aMTh}{2}\|\theta_h\|^2_{0,\O} + \aATh \kappa \Delta t\|\theta_h\|^2_{1,\O}\\
& \quad- \frac{\Delta t }{2\aATh \kappa} \|f^{n}_{\theta}\|^2_{0,\O}
- \frac{\aATh \kappa\Delta t }{2}\|\theta_h\|^2_{1,\O}
- \frac{\Delta t C_{\gb}}{2} (\|\psi_h\|^2_{1,\O}+\|\theta_h\|^2_{0,\O})\\
& \geq\min\left\{\aMFh,\aMTh\right\} (\|\psi_h\|^2_{1,\O} +\|\theta_h\|^2_{0,\O}) 
+\frac{\Delta t}{2} \min\left\{\aAFh\nu,\aATh \kappa \right\} (\|\psi_h\|^2_{1,\O} +\|\theta_h\|^2_{0,\O})\\
& \quad  - \frac{\Delta t C_{\gb}}{2} (\|\psi_h\|^2_{1,\O}+\|\theta_h\|^2_{0,\O})
- \frac{\CMFh^2}{2\aMFh}\|\psi^{n-1}_h\|^2_{1,\O}+ \frac{\CMTh^2}{2\aMTh}\|\theta^{n-1}_h\|^2_{0,\O} \\
&\quad	- \frac{\Delta t }{2\aAFh \nu} \|\fb^{n}_{\psi}\|^2_{0,\O}- \frac{\Delta t }{2\aATh \kappa} \|f^{n}_{\theta}\|^2_{0,\O}\\
& \geq \frac{1}{2}\left(\widehat{\alpha} +\Delta t \left(\gamma - C_{\gb} \right) \right)  (\|\psi_h\|^2_{1,\O}+\|\theta_h\|^2_{0,\O})- \frac{\CMFh^2}{2\aMFh}\|\psi^{n-1}_h\|^2_{1,\O}- \frac{\CMTh^2}{2\aMTh}\|\theta^{n-1}_h\|^2_{0,\O}\\  
&\quad
 - \frac{\Delta t }{2\aAFh \nu} \|\fb^{n}_{\psi}\|^2_{0,\O}- \frac{\Delta t }{2\aATh \kappa} \|f^{n}_{\theta}\|^2_{0,\O}.
	\end{split}
\end{equation*}
 
Thus, from assumption \eqref{assump:existence}, we can set 
\begin{equation*}
\rho := \left(\widehat{\alpha} +\Delta t\left(\gamma - C_{\gb} \right) \right) ^{-\frac{1}{2}} \left(
	\frac{\CMFh^2}{\aMFh}\|\psi^{n-1}_h\|^2_{1,\O} + \frac{\CMTh^2}{\aMTh}\|\theta^{n-1}_h\|^2_{0,\O}
	+ \frac{\Delta t }{\aAFh \nu} \|\fb^{n}_{\psi}\|^2_{0,\O}+ \frac{\Delta t }{\aATh \kappa} \|f^{n}_{\theta}\|^2_{0,\O}\right)^{\frac{1}{2}},
\end{equation*}
and  $\mathcal{S}:= \left\{(\varphi_h,w_h) \in \Xh :|||(\varphi_h,w_h) ||| \leq \rho 
\right \}$. Then, we have that
\begin{equation*}
	\langle \Phi(\psi_h,\theta_h),(\psi_h,\theta_h)  \rangle \geq 0  \quad \text{for any} \quad (\psi_h,\theta_h) \in  \partial \mathcal{S}.
\end{equation*}
Then, by employing the  fixed-point Theorem~\cite[Chap. IV, Corollary 1.1]{GR}, 
there exists $(\psi^n_{h},\theta^n_{h}) \in \mathcal{S}$, such that  
$\Phi(\psi_h^n,\theta_h^n)=\0$, i.e., the fully-discrete problem~\eqref{eq:VEM2} admits 
at least  one solution $(\psi_h^n,\theta_h^n) \in \mathcal{S}$ at every time step $t_n$.
\end{proof} 
\begin{remark} 
From assumption~\eqref{assump:existence} it follows that if $C_{\gb} > \gamma$, that is when the buoyancy term is strong when compared to the diffusion terms, a ``small time step condition" is needed in order to guarantee the existence of a discrete solution.   
\end{remark}

The following result establishes that the fully-discrete scheme~\eqref{eq:VEM2} 
is unconditionally stable.  
\begin{theorem}\label{uncond:stable}
Assume that $\fb_{\psi} \in \L^2(0,T;\LL^2(\O))$, $f_{\theta} \in \L^2(0,T;\L^2(\O))$, $\gb \in \L^{\infty}(0,T;\LL^{\infty}(\O))$. Moreover, suppose that the initial data satisfy 
$\psi_0 \in \H^2_0(\O)$ and $\theta_0 \in \H_0^1(\O)$. Then,  the fully-discrete 
scheme~\eqref{eq:VEM2} are unconditionally stable and satisfy the following estimate for 
any $0<m \leq N$
\begin{equation*}
\begin{split}
\|(\psi^m_{h},\theta^m_{h})\|_{\H^1(\O) \times \L^2(\O)} 
&+\Big(\Delta t\sum_{n=0}^m  \|(\psi^n_{h},\theta^n_{h})\|^2_{\H^2(\O) \times \H^1(\O)}  \Big)^{\frac{1}{2}}\\
&\leq C \Big(\Big(\Delta t\sum_{n=0}^m  \|(\fb_{\psi}^n,f_{\theta}^n)\|^2_{\LL^2(\O) \times \L^2(\O)} \Big)^{\frac{1}{2}} + \|(\psi_0,\theta_0)\|_{\H^2(\O) \times \H^1(\O)}\Big)=:\delta, 
\end{split}
\end{equation*}
where $C>0$ is independent of $h$ and $\Delta t$.
\end{theorem}
\begin{proof}
Let $(\psi_h^n,\theta^n_h) \in \Wh \times \Hh$ be a solution of fully-discrete problem~\eqref{eq:VEM2}. 
We consider the following equivalent norms:
\begin{equation}\label{norm:heatfluid}
||| \phi_h|||_{F,h} := (M^h_F(\phi_h,\phi_h))^{1/2} , \qquad		
||| v_h|||_{T,h} := (M^h_T(v_h,v_h))^{1/2} 
\quad \phi_h \in \Wh, \ \forall v_h \in \Hh. 	
\end{equation}

Taking $v_h =  \theta_h^n \in \Hh$ in the second equation of~\eqref{eq:VEM2}, 
using Lemma~\ref{lemma:disc:form}, the Young inequality and some identities of real numbers,
we obtain
\begin{equation*}
\frac{1}{2 \Delta t}(|||\theta^n_h|||^2_{T,h}-|||\theta^{n-1}_h|||^2_{T,h}) 
+ \aATh \kappa \|\theta^n_h\|^2_{1,\O}
\leq C_{F_{\theta}}\|f^n_{\theta}\|_{0,\O} \|\theta^n_h\|_{1,\O} \leq C\|f^n_{\theta}\|^2_{0,\O} 
 + \frac{1}{2}\aATh \kappa \|\theta^n_h\|^2_{1,\O}.  	
\end{equation*}

Then, multiplying by $2\Delta t$, using the equivalence of norms and summing 
for $n=1,\ldots, m$, we have that 
\begin{equation}\label{stab:theta}
\|\theta^m_h\|^2_{0,\O} + \Delta t \sum^m_{n=1}\|\theta^n_h\|^2_{1,\O} 
\leq C \Big(\Delta t \sum^m_{n=1}\|f^n_{\theta}\|^2_{0,\O} +\|\theta^{0}_h\|^2_{0,\O}\Big).
\end{equation}

Analogously, taking $\phi_h= \psi^n_h \in \Wh$ in the first equation of \eqref{eq:VEM2} 
and repeating the same arguments, we obtain
\begin{equation}\label{pre:stab:psi}
|||\psi^n_h|||^2_{F,h}-|||\psi^{n-1}_h|||^2_{F,h} + \aAFh \nu \Delta t \|\psi^n_h\|^2_{2,\O}
\leq  C \Delta t C_{\gb}\|\theta^n_h\|^2_{0,\O} + C\Delta t\|\fb^n_{\psi}\|^2_{0,\O}, 
\end{equation}
where the constant $C_{\gb}$ is defined in Theorem \ref{existence}.

Now,  summing for $n=1,\ldots, m$, inserting \eqref{stab:theta} in \eqref{pre:stab:psi} and 
using the equivalence of norms and, we get
 \begin{equation}\label{stab:psi}
 	\|\psi^m_h\|^2_{1,\O} + \Delta t \sum^m_{n=1}\|\psi^n_h\|^2_{2,\O} 
 	\leq C \Big(\Delta t \sum^m_{n=1}\big(\|\fb^n_{\psi}\|^2_{0,\O}+\|f^n_{\theta}\|^2_{0,\O}\big) +\|\psi^{0}_h\|^2_{1,\O} +\|\theta^{0}_h\|^2_{0,\O}\Big),
 \end{equation}
where the constant $C_{\gb}$ was included in the constant $C$ to shorten the bound. 

Finally, the desired result follows adding \eqref{stab:theta} and \eqref{stab:psi}.
\end{proof}

We now recall local inverse inequalities for the virtual spaces $\WK$ and $\HE$  
(see \cite{BM15,CH2017}):
\begin{equation}\label{inverse:ineq}
	|\phi_h|_{2,\E} \leq \Cinv h^{-1}_{\E} |\phi_h|_{1,\E} \quad \forall \phi_h \in \WK \quad \text{and} \quad |v_h|_{1,\E} \leq \Cinv h^{-1}_{\E} \|v_h\|_{0,\E} 
\quad \forall v_h \in \HE,\\
\end{equation}

The following result establishes that  the solution of scheme~\eqref{eq:VEM2} 
is unique for small values of  $\Delta t$.

\begin{theorem}\label{uniqueness}
Let $\aMFh,\aMTh,\CBFh$ and $\CBTh$ be the constants in Lemma~\ref{lemma:disc:form}. 
Moreover, let $\delta$ be the upper bound  in Theorem~\ref{uncond:stable}, 
$C_{\gb}$ be the constant defined in Theorem \ref{existence} and $\Cinv$ be the constant 
in \eqref{inverse:ineq}.
Assume  that
	\begin{equation}\label{assump:stepsmall}
	\Delta t < \min\{\aMFh, \aMTh\} \min\left\{ \frac{ h_{min}^2}{2\Cinv^2(\CBFh+\CBTh)\delta},
	\frac{ h_{min}}{2\Cinv C_{\gb}} \right\}.
\end{equation}

Then, for each $n=1, \ldots, N$ the solution of the fully-discrete scheme~\eqref{eq:VEM2} is unique.
\end{theorem}
\begin{proof}
Let $1 \leq n \leq N$ and $(\psi_{h1}^n,\theta^n_{h1}),(\psi_{h2}^n,\theta_{h2}^n)\in \Wh \times \Hh$ 
be two solutions of  problem~\eqref{eq:VEM2}.  
Then, setting $\widetilde{\psi^n_h} :=\psi_{h1}^n-\psi_{h2}^n$, 
$\widetilde{\theta_{h}^n}:=\theta_{h1}^n-\theta_{h2}^n$ and using the definition 
of operator~\eqref{operator}, for all $(\phi_h,v_h) \in \Wh \times \Hh$, we have that
\begin{equation}\label{unique1}
\begin{split}
&M^h_F(\widetilde{\psi^n_h},\phi_h)+M^h_T(\widetilde{\theta^n_h},v_h ) 
+ \nu \Delta t A^h_F(\widetilde{\psi^n_{h}},\phi_h)+ \kappa \Delta t A^h_T(\widetilde{\theta^n_h},v_h)- \Delta t C^h(\widetilde{\theta^n_h}, \phi_h) \\
&\quad+\Delta t(B^h_F(\psi^n_{h1};\psi^n_{h1},\phi_h)-B^h_F(\psi^n_{h2};\psi^n_{h2},\phi_h))
+\Delta t (B^h_{\skew}(\psi^n_{h1};\theta^n_{h1},v_h)-B^h_{\skew}(\psi^n_{h2};\theta_{h2}^n,v_h))
=0.
\end{split}
\end{equation}

Adding and subtracting $B^h_F(\psi^n_{h2};\psi^n_{h1},\phi_h)$ and  $B^h_{\skew}(\psi^n_{h2};\theta^n_{h1},v_h)$ we obtain
\begin{align*}
	B^h_F(\psi^n_{h1};\psi^n_{h1},\phi_h)-B^h_F(\psi^n_{h2};\psi^n_{h2},\phi_h) &= B^h_F(\widetilde{\psi^n_h};\psi^n_{h1},\phi_h)+B^h_F(\psi^n_{h2};\widetilde{\psi^n_h},\phi_h)\\
	B^h_{\skew}(\psi^n_{h1};\theta^n_{h1},v_h)-B^h_{\skew}(\psi^n_{h2};\theta^n_{h2},v_h)&=
	B^h_{\skew}(\widetilde{\psi^n_h};\theta^n_{h1},v_h)+B^h_{\skew}(\psi^n_{h2};\widetilde{\theta^n_{h}},v_h).
\end{align*}

Next, taking $\phi_h=\widetilde{\psi^n_h}$  and $v_h=\widetilde{\theta^n_h}$ in~\eqref{unique1},  from
the above identities, the skew-symmetry of trilinear forms, the continuity and coercivity properties of  
the multilinear forms involved (cf. Lemma~\ref{lemma:disc:form}), it follows 
\begin{equation*}
\begin{split}
& \aMFh\|\widetilde{\psi^n_h}\|^2_{1,\O}+ \aMTh\|\widetilde{\theta^n_h}\|^2_{0,\O} 
+ \aAFh\nu \Delta t \|\widetilde{\psi^n_h}\|^2_{2,\O}+ \aATh\kappa \Delta t \|\widetilde{\theta^n_h}\|^2_{1,\O} \\
&\leq \Delta t \CBFh\|\widetilde{\psi^n_h}\|_{2,\O} \|\psi^n_{h1}\|_{2,\O} \|\widetilde{\psi^n_h}\|_{2,\O} 
+\Delta t \CBTh\|\widetilde{\psi^n_h}\|_{2,\O} \|\theta^n_{h1}\|_{1,\O} \|\widetilde{\theta^n_h}\|_{1,\O} + \Delta tC_{\gb}\|\widetilde{\theta^n_h}\|_{1,\O}  \|\widetilde{\psi^n_h}\|_{2,\O}\\
&\leq \Delta t \CBFh\|\widetilde{\psi^n_h}\|^2_{2,\O} \|\psi^n_{h1}\|_{2,\O} 
+\frac{1}{2} \Delta t \CBTh \|\theta^n_{h1}\|_{1,\O} (\|\widetilde{\psi^n_h}\|^2_{2,\O} 
+ \|\widetilde{\theta^n_h}\|^2_{1,\O})+ \frac{1}{2}\Delta t C_{\gb} (\|\widetilde{\psi^n_h}\|^2_{2,\O} 
+ \|\widetilde{\theta^n_h}\|^2_{1,\O})\\
&\leq \Delta t\left(\CBFh\|\psi^n_{h1}\|_{2,\O}+\CBTh\|\theta^n_{h1}\|_{1,\O}+ C_{\gb} \right)\|(\widetilde{\psi^n_h},\widetilde{\theta^n_h})\|^2_{\H^2(\O) \times \H^1(\O)}.
\end{split}
\end{equation*}

Now, employing  local inverse inequalities \eqref{inverse:ineq} in the above estimate 
and Theorem~\ref{uncond:stable}, we get
\begin{equation*}	
\begin{split}
\min\{\aMFh,\aMFh\} \|(\widetilde{\psi^n_h},\widetilde{\theta^n_h})\|^2_{\H^1(\O) \times \L^2(\O)}
\leq   \Cinv h^{-1}_{min} \Delta t\left((\CBFh+ \CBTh) \Cinv h_{min}^{-1}\delta  + C_{\gb} \right)\|(\widetilde{\psi^n_h},\widetilde{\theta^n_h})\|^2_{\H^1(\O) \times \L^2(\O)}.
\end{split}
 \end{equation*}

From the assumption~\eqref{assump:stepsmall}, we have that
\begin{equation}\label{lipsc}
\frac{\Cinv h^{-1}_{min}}{\min\{\aMFh,\aMFh\}} \left((\CBFh+ \CBTh)\Cinv h_{min}^{-1}\delta  + C_{\gb} \right)	 <1. 
\end{equation}
Thus, $\widetilde{\psi_h^n}=0$ and $\widetilde{\theta^n_h}=0$, which implies
 $\psi^n_{h1}=\psi^n_{h2}$ and 
$\theta^n_{h1}=\theta^n_{h2}$.  The proof is complete.
\end{proof}
\begin{remark}
Exploiting the fact that we are in the two dimensional case and using 
sharper Sobolev bounds for the convective terms, 
we could get a power $h_{min}^{-\epsilon}$, for all $\epsilon>0$,  instead of $h_{min}^{-1}$ 
in the term $h_{min}^{-1}\delta$ (see equation \eqref{lipsc}).
\end{remark}

%%%%_______________________________________________________________________________
\setcounter{equation}{0}
\section{Convergence analysis}\label{SEC:convergence}
%%%%_______________________________________________________________________________	
This section is devoted to the  convergence analysis of the fully-discrete 
formulation~\eqref{eq:VEM2} introduced in the previous section. 
We  start recalling some preliminary results of approximation in the polynomial 
and virtual spaces. Moreover, we introduce  an energy operator associated to the 
$\H^2$-inner product with its corresponding 
approximation properties. Later on, we state technical results, which will be 
useful to provide the convergence result of our fully-discrete virtual scheme.
 
\subsection{Preliminary results}
First, we recall the following polynomial approximation result (see for instance \cite{BS-2008}).
Here below $E$ represents as usual a generic element of $\{ {\Omega}_h \}_{h>0}$, 
which we recall satisfies assumptions  ${\bf A1}$, ${\bf A2}$ in Section \ref{meshassump}.

\begin{proposition}\label{app-pol}
For each $\phi\in \H^{m}(\E)$,
there exist $\phi_{\pi}\in\P_{n}(\E)$, $n \geq 0$ and $C>0$ independent of $h_{\E}$,  such that
\begin{equation*}
\|\phi-\phi_{\pi}\|_{t,\E} \leq Ch_\E^{m- t}|\phi|_{m,\E},
\quad t,m \in {\mathbb R} , \ 0 \leq t \leq m \leq n+1 .
\end{equation*} 
\end{proposition}

We continue with the following approximation for the stream-function and temperature 
virtual element spaces, which can be found  in~\cite{G2021,BMR2019,BM13} and \cite{MRR2015,CMS2016,BBMRm3as2016}, respectively. 

\begin{proposition}\label{approx-virtual}
For each  $\phi\in \H^{m}(\Omega)$, 
there exist $\phi_{I}\in\Wh$ and $\CI >0$, independent of $h$, such that
\begin{equation*}
\|\phi-\phi_{I}\|_{t,\Omega} \leq \CI h^{m-t}|\phi|_{m,\Omega},  \quad t =0,1,2,
\quad 2 < m \leq k+1, \quad k \geq 2.
\end{equation*}
\end{proposition}

For the temperature variable, we present local and global approximation properties. 
\begin{proposition}\label{approx-virtual-temp}
For each  $v\in \H^{m}(\Omega)$,  there exist $v_{I}\in\Hh$ and $\CI >0$,
independent of $h$, such that
\begin{equation*}
\|v-v_{I}\|_{t,\E} \leq \CI h_{\E}^{m-t}|v|_{m,\E} \quad \forall \E \in \CT_h; \quad\|v-v_{I}\|_{t,\Omega} \leq \CI h^{m-t}|v|_{m,\Omega},  \quad t =0,1,
\quad 1 < m \leq \ell+1, \quad \ell \geq 1.
	\end{equation*}
\end{proposition}

Now, we will introduce  the following  discrete biharmonic projection associated 
with the stream-function discretization. For each $\varphi \in \H_0^2(\O)$, 
we consider  the operator $\mS_h:\H_0^2(\O)\rightarrow  \Wh$, defined as 
the solution of problem:
\begin{equation}\label{dis:engy:bilap}
A_{F}^h(\mS_h \varphi,\phi_h)=A_{F}(\varphi,\phi_h) \qquad \forall \phi_h \in \Wh,
\end{equation}
where $A_{F}(\cdot,\cdot)$ was defined in \eqref{AF-cont} and we recall that
$A_{F}^h(\cdot,\cdot)$ is the global version of the form defined in \eqref{localdisc:AF}.

By using Propositions~\ref{cons-stab-forms}, \ref{app-pol} and~\ref{approx-virtual}, 
the following approximation result for the energy projection $\mS_h(\cdot)$ holds 
true (see \cite[Lemma 5.3]{AMNS2021-M2AN}).
\begin{proposition}\label{prop:engy:bihar}
For each $\varphi \in \H_0^2(\O)$, there exists a unique function  
$\mathcal{S}_h\varphi\in \Wh$  satisfying \eqref{dis:engy:bilap}. 
Moreover, if  $\varphi \in \H^{2+s}(\O)$, with $\frac{1}{2} < s \leq k-1$, 
then the  following approximation property holds:
\begin{equation*}%\label{estima:H2Ru}
\|\varphi-\mathcal{S}_h\varphi\|_{1,\O} +
h^{\tilde{s}}\|\varphi-\mathcal{S}_h\varphi\|_{2,\O} \leq Ch^{\tilde{s}+s}|\varphi|_{2+s,\O},
\end{equation*}
where  $C$ is a positive constant, independent of $h$ and  
$\tilde{s} \in (\frac{1}{2},1]$ depends on the domain $\O$.
\end{proposition}

In what follows, we will establish four technical lemmas involving the trilinear 
forms associated to transport/convection and the bilinear form associated to the buoyancy term;
these results will be useful in subsection~\ref{error:fully:scheme}. 

\begin{lemma}\label{boundBFh:two}
For all $\zeta_h;\varphi_h,\phi_h \in \Wh$, there exists $\CBFh>0$, independent of $h$, such that 
\begin{equation*}
|B^h_F(\zeta_h;\varphi_h,\phi_h)| \leq \CBFh \, \|\zeta_h\|_{2,\O} \|\varphi_h\|_{2,\O} \|\phi_h\|^{\frac{1}{2}}_{2,\O}\|\phi_h\|^{\frac{1}{2}}_{1,\O}.
\end{equation*}
\end{lemma}  
\begin{proof}
We use the definition of the trilinear form $B_F^h(\cdot;\cdot,\cdot)$ (cf. \eqref{disc-formB}), the H\"{o}lder inequality, the continuity of the operators $\PioK$ and $\PimunoK$ with  respect to the
$\L^2$- and $\L^4$-norms (see \cite{BLV-NS18}) respectively, and the H\"{o}lder inequality for sequences, to obtain
\begin{equation*}
\begin{split}
B_F^h(\zeta_h;\varphi_h,\phi_h) 
& \leq \sum_{\E \in \CT_h} \|\PioK \Delta \zeta_h\|_{0,\E} \|\PimunoK \curl \varphi_h\|_{\L^{4}(\E)}\|\PimunoK \nabla \phi_h\|_{\L^{4}(\E)}\\
& \leq C \|\Delta \zeta_h\|_{0,\O} \|\curl \varphi_h\|_{\L^{4}(\O)} \|\nabla \phi_h\|_{\L^{4}(\O)}\\
& \leq C \|\Delta \zeta_h\|_{0,\O} \|\varphi_h\|_{2,\O} \|\nabla \phi_h\|_{\L^{4}(\O)}, 
\end{split}
\end{equation*} 
where we have used  the Sobolev inclusion $\H^1(\O) \hookrightarrow \L^4(\O)$. 
Now, applying the Sobolev inequality~\eqref{sobolev:ineq} with $\vb=\nabla \phi_h$
we obtain the desired result.
\end{proof}
\begin{lemma}\label{lemma-tec2}
For all $\zeta, \varphi, \phi \in \HdoO$, we have that 
\begin{equation*}
\begin{split}	
B_{F}^h(\varphi;\varphi,\phi)-B_{F}^h(\zeta;\zeta,\phi)
& = B_{F}^h(\varphi;\varphi-\zeta+\phi,\phi)+ B_{F}^h(\varphi-\zeta+\phi;\zeta,\phi)
	- B_{F}^h(\phi;\zeta,\phi).
	\end{split} 
	\end{equation*}
\end{lemma}	
\begin{proof}
The proof follows by adding and subtracting
suitable terms, and using the trilineality and skew-symmetric properties of 
form $B_{F}^h(\cdot;\cdot,\cdot)$. 
\end{proof}

Next lemmas give us the measure 
of the variational crime in the discretization of 
the trilinear forms $B_{F}(\cdot;\cdot,\cdot)$ and 
$B_{\skew}(\cdot;\cdot,\cdot)$ and 
the bilinear form $C(\cdot,\cdot)$.
\begin{lemma}\label{lemma:crimenBF}
Let $\varphi(t)\in \H^2_0(\O) \cap \H^{2+s}(\O)$, with $\frac{1}{2}<s \leq k-1$, for almost all $t \in (0,T)$.
Then, there exists  $C>0$, independent of mesh size $h$, such that
\begin{equation*}
|B_F(\varphi;\varphi,\phi_h)-B_F^h(\varphi;\varphi,\phi_h)|
\leq Ch^{s}\big(\|\varphi\|_{1+s,\O}+\|\varphi\|_{2,\O} \big)
\|\varphi\|_{2+s,\O}\|\phi_h\|_{2,\O} \qquad\forall\phi_h\in \Wh. 
\end{equation*}
\end{lemma}
\begin{proof}
The proof has been established in \cite[Lemma 5.4]{AMNS2021-M2AN}.
\end{proof}
\begin{lemma}\label{lemma:crimen}
Let $\frac{1}{2}<\gamma \leq \min\{k-1,\ell\}$. Assume that  $\varphi(t)\in \H^2_0(\O) \cap \H^{2+\gamma}(\O)$ and  $v(t)\in \H^1_0(\O) \cap \H^{1+\gamma}(\O)$, for almost all $t \in (0,T)$.
Then, there exists  $C>0$, independent of mesh size $h$, such that, a.e. $t \in (0,T)$,
\begin{equation} \label{prop:Bskew-Bskewh} 
|B_{\skew}(\varphi;v,w_h)-B_{\skew}^h(\varphi;v,w_h)|
\leq Ch^{\gamma}\|\varphi\|_{2+\gamma,\O}\|v\|_{1+\gamma,\O}\|w_h\|_{1,\O}
\qquad \forall w_h\in \Hh. 
\end{equation}
Moreover, assume that $\gb(t) \in\WW^{\gamma}_{\infty}(\O)$, for almost all $t \in (0,T)$. Then, a.e. $t \in (0,T)$,
\begin{equation}
|C(v,\phi_h)-C^h(v,\phi_h)|\leq Ch^{\gamma}\|\gb\|_{\W^{\gamma}_{\infty}(\O)}\|v\|_{1+\gamma,\O}\|\phi_h\|_{1,\O}\qquad \qquad 
\forall \phi_h\in \Wh. \label{prop:C-Ch}
\end{equation}
\end{lemma}

\begin{proof}
To prove estimate~\eqref{prop:Bskew-Bskewh},
 we split the consistency error as follow:
\begin{equation}\label{two:terms}
B_{\skew}(\varphi;v,w_h)-B_{\skew}^h(\varphi;v,w_h) 
= \frac{1}{2}\left(\beta_1(w_h) +\beta_2(w_h) \right),
\end{equation}
 where 
\begin{equation*}
\beta_1(w_h) := \sum_{\E \in \CT_h} \left(B^{\E}_T(\varphi; v, w_h )
	-B^{h,\E}_T(\varphi; v, w_h )  \right)\quad \text{and} \quad 
	\beta_2(w_h) := \sum_{\E \in \CT_h} \left(B^{\E}_T(\varphi; w_h, v )
	-B^{h,\E}_T(\varphi; w_h, v ) \right).
\end{equation*}

In what follows, we will establish bounds for the terms $\beta_1(w_h)$ and $\beta_2(w_h)$. 
Indeed, for the term $\beta_1(w_h)$ we have 
\begin{equation}\label{beta1}
\begin{split}
\beta_1(w_h) &= \sum_{\E \in \CT_h}\int_{\E} (\curl \varphi \cdot \nabla v) w_h -  
\int_{\E} (\PimunoK\curl \varphi \cdot \PiellK\nabla v) \Pi_{\E}^{\ell-1} w_h\\
& = \sum_{\E \in \CT_h}\int_{\E} (\curl \varphi \cdot \nabla v) (w_h-\Pi_{\E}^{\ell-1} w_h) +
\sum_{\E \in \CT_h}\int_{\E} \left(\curl \varphi \cdot (\nabla v- \PiellK \nabla v)\right) \Pi_{\E}^{\ell-1} w_h\\
& \quad + \sum_{\E \in \CT_h} \int_{\E} \left((\curl \varphi- \PimunoK \curl \varphi) \cdot  \PiellK \nabla v\right) \Pi_{\E}^{\ell-1} w_h\\
&=: T_1 +T_2 +T_3.	
\end{split}
\end{equation}

In order to bound the terms $T_1$, first we consider the case $1/2<\gamma \leq 1$. 
Then, by using approximation property of $\Pi_{\E}^{\ell-1}$ and the H\"{o}lder inequality,  it follows
\begin{equation*}
\begin{split}
T_1 &\leq \sum_{\E \in \CT_h} \|\curl \varphi\|_{\L^4(\E)} \|\nabla \varphi\|_{\L^4(\E)} 
\|w_h-\Pi_{\E}^{\ell-1} w_h\|_{0,\E}\\
&\leq C\sum_{\E \in \CT_h} \|\curl \varphi\|_{\L^4(\E)} \|\nabla \varphi\|_{\L^4(\E)} 
h_{\E}|w_h|_{1,\E}\\
& \leq  Ch\|\varphi\|_{1+\gamma,\O} \|v\|_{1+\gamma,\O} \|w_h\|_{1,\O}.		
\end{split}
\end{equation*}

On the other hand, for the case $1<\gamma \leq \ell$, we use orthogonality 
and approximation properties  of $\Pi_{\E}^{\ell-1}$, 
the H\"{o}lder inequality (for sequences), to obtain  
\begin{equation*}
\begin{split}
T_1 &
= \sum_{\E \in \CT_h}\int_{\E} (\curl \varphi \cdot \nabla v- \Pi_{\E}^{\ell-1}(\curl \varphi \cdot \nabla v)) (w_h-\Pi_{\E}^{\ell-1} w_h)
\leq Ch^{\gamma}|\curl \varphi \cdot \nabla v|_{\gamma-1,\O}\|w_h\|_{1,\O}.
\end{split}
\end{equation*}
Then,  applying the  H\"{o}lder inequality and Sobolev embedding, we get
\begin{equation*}
\begin{split}
|\curl \varphi \cdot \nabla v|_{\gamma-1,\O} &\leq C \|\curl \varphi\|_{\W_{4}^{\gamma-1}(\O)} \|\nabla v\|_{\W_{4}^{\gamma-1}(\O)} 
\leq C \|\varphi\|_{1+\gamma,\O} \|v\|_{1+\gamma,\O}.
\end{split}
\end{equation*}

Collecting the above inequalities, for $\frac{1}{2}< \gamma \leq \ell$, we have
\begin{equation}\label{T1-beta1}
T_1 \leq  Ch^{\gamma}\|\varphi\|_{1+\gamma,\O} \|v\|_{1+\gamma,\O} \|w_h\|_{1,\O}.		
\end{equation}

Now, for the term $T_2$ we proceed as follows. First, we apply the H\"{o}lder inequality, 
then by using stability and approximation properties of the $\L^2$-projectors, 
Sobolev embedding and  the H\"{o}lder inequality for sequences, we get
\begin{equation}\label{T2-beta1}
\begin{split}
T_2 & \leq \sum_{\E \in \CT_h} \|\curl \varphi\|_{\L^4(\E)} \| \nabla v- \PiellK \nabla v\|_{0,\E} \|\Pi^{\ell-1}_{\E} w_h \|_{\L^4(\E)}
 \leq Ch^\gamma \|\varphi\|_{1+\gamma,\O}\|v\|_{1+\gamma,\O} \|w_h\|_{1,\O}.
\end{split}
\end{equation}
For the term $T_3$, we follow similar arguments,  to obtain
\begin{equation}\label{T3-beta1}
\begin{split}
T_3 \leq Ch^{\gamma} \|\varphi\|_{1+\gamma,\O}\|v\|_{1+\gamma,\O} \|w_h\|_{1,\O}.
	\end{split}
\end{equation} 
From the bounds \eqref{beta1}, \eqref{T1-beta1}, \eqref{T2-beta1} and 
\eqref{T3-beta1}, we conclude that 
\begin{equation}\label{final:beta1}
\beta_1(w_h) \leq Ch^{\gamma}\|\varphi\|_{1+\gamma,\O}\|v\|_{1+\gamma,\O} \|w_h\|_{1,\O}.	
\end{equation}
Now, we will focus on the term $\beta_2(w_h)$. 
To estimate this term, first we add and 
subtract suitable expressions to obtain 
\begin{equation*}
\begin{split}
\beta_2(w_h) &= \sum_{\E \in \CT_h}\int_{\E} (\curl \varphi \cdot \nabla w_h) v -  
\int_{\E} (\PimunoK\curl \varphi \cdot \PiellK \nabla w_h) \Pi_{\E}^{\ell-1} v\\
& = \sum_{\E \in \CT_h}\int_{\E} v(\curl \varphi) \cdot  (\nabla w_h-\PiellK \nabla w_h) 
+ \sum_{\E \in \CT_h} \int_{\E} \left(\curl \varphi- \PimunoK \curl \varphi\right) \cdot  v \PiellK \nabla w_h\\
& \quad +\sum_{\E \in \CT_h}\int_{\E} \left(\PimunoK \curl \varphi \cdot  \PiellK \nabla w_h\right) (v-\Pi_{\E}^{\ell-1} v)\\
&=: I_1 +I_2 +I_3.	
	\end{split}
\end{equation*}
Applying orthogonality and approximation properties of $\PiellK$,  we have 
\begin{equation*}
	\begin{split}
I_1&= \sum_{\E \in \CT_h}\int_{\E} (v(\curl \varphi)  - \PiellK(v(\curl \varphi) )) \cdot  ( \nabla w_h-\PiellK \nabla w_h)\\
&\leq C\sum_{\E \in \CT_h} h_{\E}^{\gamma}|v(\curl \varphi)|_{\gamma,\E} |w_h|_{1,\E}
\leq Ch^{\gamma}|v(\curl \varphi)|_{\gamma,\O}\|w_h\|_{1,\O}.
	\end{split}
\end{equation*}
Then, employing the  H\"{o}lder inequality and  Sobolev embedding, we get
\begin{equation*}
\begin{split}
|v(\curl \varphi)|_{\gamma,\O} &\leq C  \|v\|_{\W_{4}^{\gamma}(\O)}\|\curl \varphi\|_{\W_{4}^{\gamma}(\O)} % \leq C \|\varphi\|_{2+r,\O} \|v\|_{1+r,\O}\\& 
\leq C \|\varphi\|_{2+\gamma,\O} \|v\|_{1+\gamma,\O}.
\end{split}
\end{equation*}
From the two bounds above, we obtain  
\begin{equation*}
I_1 \leq Ch^{\gamma}\|\varphi\|_{2+\gamma,\O} \|v\|_{1+\gamma,\O}\|w_h\|_{1,\O}.	
\end{equation*}
The terms $I_2$ and $I_3$ can be estimated using similar arguments.  We conclude that
\begin{equation}\label{final:beta2}
\beta_2(w_h) \leq Ch^{\gamma}\|\varphi\|_{2+\gamma,\O}\|v\|_{1+\gamma,\O} \|w_h\|_{1,\O}.	
\end{equation}

The proof  of \eqref{prop:Bskew-Bskewh} follows from~\eqref{two:terms}, \eqref{final:beta1} and \eqref{final:beta2}.

Next, we will prove property \eqref{prop:C-Ch}. Let $\phi_h \in \Wh$, 
then adding and subtracting the term  $\gb v \cdot \PimunoK \curl \phi_h$ 
and by using orthogonality and approximations properties of projection 
$\PimunoK$, we have 
\begin{equation*}
\begin{split}
C(v,\phi_h)-C^h(v,\phi_h) 
&= \sum_{\E \in \CT_h} \int_{\E} (\gb v-\PimunoK(\gb v)) \cdot (\curl \phi_h -   \PimunoK\curl \phi_h) 
+ \int_{\E} \gb (v- \Pi_{\E}^{\ell-1} v)\cdot  \PimunoK\curl \phi_h  \\ 
&\leq C\sum_{\E \in \CT_h} (h_{\E}^{\gamma} |\gb v|_{\gamma,\E} \|\curl \phi_h\|_{0,\E} 
+ h_{\E}^{\gamma}\|\gb\|_{\L^{\infty}(\E)} \|v\|_{\gamma,\E}  \|\curl \phi_h\|_{0,\E})\\
&\leq Ch^{\gamma}\|\gb\|_{\W^{\gamma}_{\infty}(\O)}\|v\|_{\gamma,\O}\|\phi_h\|_{1,\O},
\end{split}
\end{equation*} 
where we have used the H\"{o}lder inequality. The proof is complete.
\end{proof}

We finish this subsection recalling a discrete Gronwall inequality, which will be useful to derive the error estimate of the fully-discrete virtual scheme~\eqref{eq:VEM2}. 
\begin{lemma}\label{discrete:gronwall}
Let $D\geq 0$, $a_j$, $b_j$, $c_j$ and $\lambda_j$ be non negative numbers for any integer $j \geq 0$, such that 
\begin{equation*}
a_n+ \Delta t \sum_{j=0}^n b_j \leq \Delta t \sum_{j=0}^n \lambda_j a_j+ \Delta t \sum_{j=0}^n c_j+D,  \quad n \geq 0.
\end{equation*}
Suppose that $\Delta t \lambda_j <1$ for all $j$, and set $\sigma_j:=(1-\Delta t \lambda_j)^{-1}$. 
Then, the following bound holds
\begin{equation*}
a_n+ \Delta t \sum_{j=0}^n b_j \leq \exp \Big( \Delta t \sum_{j=0}^n \sigma_j \lambda_j \Big) \Big( \Delta t \sum_{j=0}^n c_j+D\Big).
\end{equation*}
\end{lemma}
\begin{proof}
See \cite[Lemma 5.1]{HR1990}.
\end{proof}
\subsection{Error estimates for the fully-discrete scheme}\label{error:fully:scheme}
In this subsection we will provide a convergence result for the fully-discrete 
problem~\eqref{eq:VEM2} under suitable regularity conditions for the exact solution. 

We start denoting  $(\psi(t_n), \theta(t_n))$ as $(\psi^n, \theta^n)$ 
at each time level $t_n$, and splitting the stream-function error as follows:
\begin{equation*}
\begin{split}
\psi^n-\psi^n_h &= (\psi^n-\mS_h \psi^n) - (\psi^n_h- \mS_h\psi^n)=: \etanpsi - \varnpsi.
\end{split}	
\end{equation*}  
For the temperature variable we will exploit the virtual interpolant presented in  
Proposition~\ref{approx-virtual-temp}, to split the error as:
\begin{equation*}
\theta^n-\theta^n_h = (\theta^n- \theta_I^n) - (\theta^n_h- \theta_I^n)=: \etantheta - \varntheta,
\end{equation*}  
where $\theta^n_I$ is the interpolant of $\theta^n$  in the virtual space $\Hh$.

Error estimates for the terms $\etantheta$ and $\etanpsi$ are given by 
Propositions~\ref{approx-virtual-temp} and \ref{prop:engy:bihar}, respectively. 
Therefore, we will focus on the terms $\varnpsi$ and $\varntheta$.

The following result establishes an error estimate for the fully-discrete virtual scheme~\eqref{eq:VEM2}.

\begin{theorem}\label{converg:fullydisc}
Suppose that the external forces satisfy 
$\fb_{\psi} \in \L^{\infty}(0,T;\HH^{s}(\O))$, $f_{\theta} \in \L^{\infty}(0,T;\H^{r}(\O))$ 
and $\gb \in \L^{\infty}(0,T;\WW^{\min\{s,r\}}_{\infty}(\O))$, with $\frac{1}{2}<s \leq k-1$ and $1\leq  r \leq \ell$. 
Let $(\psi^n,\theta^n)\in \H_0^2(\Omega)\times \H_0^1(\Omega)$ be the solution
of problem~\eqref{BSE-stream2} at time $t=t_n$. Moreover, assume that
\begin{equation*}
\begin{split}	
\psi \in \L^{\infty}(0,T;\H^{2+s}(\O)),& \qquad \partial_t  \psi\in \L^{1}(0,T;\H^{1+s}(\O)), \qquad \partial_{tt}\psi \in \L^{1}(0,T;\H^{1}(\O)),\\	
\theta \in \L^{\infty}(0,T;\H^{1+r}(\O)\cap\W^1_{\infty}(\O)),& \qquad   \partial_{t}\theta \in \L^{1}(0,T;\H^{r}(\O)),
\quad\qquad \partial_{tt}\theta \in \L^{1}(0,T;\L^2(\O)).
\end{split}
\end{equation*} 
 Let $(\psi^n_h,\theta^n_h) \in \Wh\times \Hh$ be the virtual element solution generated by scheme~\eqref{eq:VEM2}. 
Then, the following estimate holds
\begin{equation*}
\|(\psi^n-\psi^n_h,\theta^n-\theta^n_h)\|^2_{\H^1(\O) \times \L^2(\O)}  + \Delta t \sum_{j=1}^{n} \|(\psi^j-\psi^j_h,\theta^j-\theta^j_h)\|^2_{\H^2(\O) \times \H^1(\O)} \leq C(h^{2\min\{s,r\}} + \Delta t^2),
\end{equation*} 
where the constant $C$ is positive and depends on the physical parameters $\nu, \kappa$, final time $T$,  mesh regularity parameter, the regularity of the Boussinesq solution fields $(\psi,\theta)$ and the external forces $\fb_{\psi},f_{\theta},\gb$, but is independent of mesh	size $h$ and time steps $\Delta t$.
\end{theorem}
\begin{proof}
We will establish the proof in four steps. In Step $1$, by using the energy 
operator~\eqref{dis:engy:bilap} and the interpolant presented in Proposition~\ref{approx-virtual-temp},
we establish error equations for the momentum and energy identities in~\eqref{eq:VEM2}. 
In Steps $2$ and $3$, we derive error estimates for the error equations of Step $1$. 
Finally, in Step $4$, we combine the results obtained in Steps $2$ and $3$, 
then by employing the discrete Gronwall inequality we derive the desired result. 
\paragraph{Step 1: Establishing error equations of the momentum and energy identities.}
By using the fully-discrete scheme \eqref{eq:VEM2},
the continuous weak formulation \eqref{BSE-stream2} and the biharmonic energy
projection $\mS_h$ defined in \eqref{dis:engy:bilap}, %for all $\phi_h \in \Wh$ 
we have the following error equation for the momentum identity
(where we have taken $\phi_h=\varnpsi \in \Wh$)
\begin{equation}\label{error:eq:NS}
\begin{split}
M_F^h\left( \frac{\varphi_{\psi}^n-\varphi_{\psi}^{n-1}}{\Delta t}, \varnpsi\right)
&+\nu A_F^h(\varnpsi, \varnpsi) 
 = \Big(F_{\psi}^h(\varnpsi)-F_{\psi}(\varnpsi)\Big)
+\Big(B_F(\psi^n;\psi^n,\varnpsi)-B_F^h(\psi_h^n;\psi_h^n,\varnpsi)\Big)\\
& +\left( M_F(\partial_t\psi^n, \varnpsi)-M_F^h\Big(\frac{\mS_h \psi^n
-\mS_h \psi^{n-1}}{\Delta t},\varnpsi\Big) \right) 
%\\& \quad 
+\Big(C^h(\theta^n_h,\varnpsi)-C(\theta^n,\varnpsi)\Big)\\
&=: T_{F} + T_B + T_M + T_C.
\end{split}
\end{equation} 

Analogously, recalling that $\varntheta=\theta^n_h- \theta_I^n$, and using the definition of 
the continuous and discrete problems (cf. \eqref{BSE-stream2} and \eqref{eq:VEM2}, respectively)
for the energy equation, we have that 
\begin{equation}\label{error:eq:heat}
\begin{split}
M_T^h\left( \frac{\varphi_{\theta}^n-\varphi_{\theta}^{n-1}}{\Delta t}, \varntheta\right)
&+\kappa A_T^h(\varntheta, \varntheta) 
 = \Big(F_{\theta}^h(\varntheta)-F_{\theta}(\varntheta)\Big)
+\Big(B_{\skew}(\psi^n;\theta^n,\varntheta)-B_{\skew}^h(\psi_h^n;\theta_h^n,\varntheta)\Big)\\
& +\left( M_T(\partial_t\theta^n, \varntheta)-M_T^h\Big(\frac{\theta_I^n
-\theta_I^{n-1}}{\Delta t},\varntheta\Big) \right) %\\& \quad 
+ \kappa \left(A_{T}(\theta^n,\varntheta)- A^h_{T}(\theta_I^n,\varntheta)\right) \\
&=: I_{F} + I_B + I_M +I_A.
\end{split}
\end{equation} 

\paragraph{Step 2: Deriving  error estimate for the momentum equation~\eqref{error:eq:NS}.}
In this step we will establish bounds for each term in \eqref{error:eq:NS}. Indeed,
by using the definition of the functionals $F_{\psi}(\cdot)$ and $F_{\psi}^h(\cdot)$, the Cauchy-Schwarz and Young  inequalities for the term $T_F$ holds
\begin{equation}\label{term:TF}
\begin{split}
T_F &\leq  \frac{C}{2\epsilon} h^{2s}\|\fb_{\psi}\|^2_{\L^{\infty}(t_{n-1},t_{n}; \H^s(\O_h))} +  \frac{\epsilon}{2} \|\varnpsi\|^2_{1,\O}.
	\end{split}
\end{equation} 
For the term $T_M$, we proceed similarly as in \cite[Theorem 5.6]{AMNS2021-M2AN} to obtain
\begin{equation}\label{term:TM}
\begin{split}
T_M&:=M_F(\partial_t\psi^n,\varnpsi)-M_F^h \left( \frac{\mS_h \psi^n-\mS_h \psi^{n-1}}{\Delta t}, \varnpsi \right)=M_F\left(\partial_t\psi^n-\frac{\psi^n-\psi^{n-1}}{\Delta t}, \varnpsi \right) \\
& \quad + \sum_{\E\in\CT_h}M^{\E}_F\left(\frac{\psi^n-\psi^{n-1} }{\Delta t}
-\left( \frac{\PiK(\psi^n-\psi^{n-1} )}{\Delta t}\right), \varnpsi \right) \\
& \quad +\sum_{\E\in\CT_h}M_F^{\E,h} \left( \left( \frac{\PiK ( \psi^n-\psi^{n-1})}{\Delta t}\right)
-\frac{\mS_h \psi^n-\mS_h \psi^{n-1}}{\Delta t}, \varnpsi\right)\\
&\leq C \|\partial_{tt} \psi \|_{\L^{1}(t_{n-1},t_n; \H^{1}(\O))}  \|\varnpsi\|_{1,\O} +
\frac{C}{\Delta t } h^{s} \|\partial_t \psi \|_{\L^1(t_{n-1},t_n; \H^{1+s}(\O))} \|\varnpsi\|_{1,\O}. 
	\end{split}
\end{equation}

Next, to estimate $T_C$, we add and subtract the term $C^h(\theta^n, \varnpsi)$ to get
\begin{equation}\label{term:TC}
\begin{split}
T_C&:=C^h(\theta^n_h,\varnpsi)-C(\theta^n,\varnpsi) 
= C^h(\theta^n_h-\theta^n,\varnpsi)+ (C^h(\theta^n,\varnpsi)-C(\theta^n,\varnpsi))\\
&= (C^h(\varntheta,\varnpsi)- C^h(\etantheta,\varnpsi)) + 
(C^h(\theta^n,\varnpsi)-C(\theta^n,\varnpsi))\\
& \leq \|\gb^n\|_{\infty,\O} \left(\|\varntheta\|_{0,\O} + \|\etantheta\|_{0,\O} \right)\|\varnpsi\|_{1,\O} 
+Ch^{\min\{s,r\}}\|\gb^n\|_{\W^{\min\{s,r\}}_{\infty}(\O)}\|\theta\|_{r,\O}
\|\varnpsi\|_{1,\O}\\
& \leq C\|\gb^n\|^2_{\infty,\O}(\|\varntheta\|^2_{0,\O} + \|\varnpsi\|^2_{1,\O})+ 
Ch^{2\min\{s,r\}}\|\gb^n\|^2_{\W^{\min\{s,r\}}_{\infty}(\O)}\|\theta\|^2_{r,\O} + 
\epsilon\|\varnpsi\|^2_{1,\O},
\end{split}
\end{equation}
where we have used the H\"{o}lder inequality, bound~\eqref{prop:C-Ch} (with $\gamma=\min\{s,r\}$) and the Young inequality. 

For the term $T_B$, we have
\begin{equation}\label{preterm:TB}
\begin{split}
T_B&:=B_F(\psi^n;\psi^n,\varnpsi)-B_F^h(\psi_h^n;\psi_h^n,\varnpsi) = 
\Big( B_F(\psi^n;\psi^n,\varnpsi)-B_F^h(\psi^n;\psi^n,\varnpsi)\Big)\\
& \quad+ \Big(  B^h_F(\psi^n;\psi^n,\varnpsi)-B_F^h(\psi_h^n;\psi_h^n,\varnpsi) \Big)%\\&
=: T_{B1} + T_{B2}.
\end{split}
\end{equation}  
Now, we will bound the terms $T_{B1}$ and $T_{B2}$. Indeed, from Lemma~\ref{lemma:crimenBF} and the 
Young inequality we have that 
\begin{equation}\label{TB1}
\begin{split}
T_{B1} &:= B_F(\psi^n;\psi^n,\varnpsi)-B_F^h(\psi^n;\psi^n,\varnpsi) \leq C \: h^s (\|\psi^n\|_{2+s,\O}+\|\psi^n\|_{2,\O})\|\psi^n\|_{2+s,\O}\|\varnpsi\|_{2,\O}\\
& \leq \frac{4C}{\aAFh\nu } h^{2s} \|\psi^n\|^2_{2+s,\O}+ \frac{\aAFh\nu}{8} \|\varnpsi\|^2_{2,\O}\\
& \leq C \nu^{-1}h^{2s} \|\psi^n\|^2_{2+s,\O}+ \frac{\aAFh\nu}{8} \|\varnpsi\|^2_{2,\O},
	\end{split}
\end{equation}
where we have included the term $(\|\psi^n\|_{2+s,\O}+\|\psi^n\|_{2,\O})$ 
in the constant $C$ in order to shorten the inequality.

On the other hand, to bound the term $T_{B2}$, we apply Lemma~\ref{lemma-tec2}, 
recall that $\varnpsi= \psi_h^n-\mS_h\psi^n$  and $\etanpsi= \psi^n-\mS_h\psi^n$,
 to arrive
\begin{equation}\label{TB2}
\begin{split}
T_{B2} &:= B_F^h(\psi^n; \psi^n,\varnpsi) -B_F^h(\psi_h^n;\psi_h^n,\varnpsi) \\	
&= B_F^h(\psi^n; \psi^n-\psi_h^n+\varnpsi,\varnpsi)+
B_F^h( \psi^n-\psi_h^n+\varnpsi;\psi_h^n,\varnpsi)
-B_F^h(\varnpsi;\psi_h^n,\varnpsi)\\
& =  B_F^h(\psi^n;\etanpsi ,\varnpsi)+ B_F^h(\etanpsi;\psi^n_h,\varnpsi)- B_F^h(\varnpsi;\psi_h^n,\varnpsi).
	\end{split}
\end{equation}

By using Lemma~\ref{lemma:disc:form}, together with the Young inequality, we have 
\begin{equation*}
\begin{split}
B_F^h(\psi^n;\etanpsi,\varnpsi) 	
\leq \frac{\aAFh\nu}{8}  \|\varnpsi\|^2_{2,\O}
+ C \nu^{-1} \|\psi^n\|^2_{2,\O}\|\etanpsi\|^2_{2,\O}.
\end{split}
\end{equation*}
Now, adding and subtracting suitable terms, employing Lemma~\ref{lemma:disc:form} along with 
the Young inequality, we  obtain
\begin{equation*}
\begin{split}
B_F^h(\etanpsi;\psi^n_h,\varnpsi) &= B_F^h(\etanpsi;\psi^n+(\psi^n_h-\psi^n),\varnpsi)\\
&= B_F^h(\etanpsi;\psi^n,\varnpsi) +B_F^h(\etanpsi; \varnpsi- \etanpsi,\varnpsi)\\
&=B_F^h(\etanpsi;\psi^n,\varnpsi) -B_F^h(\etanpsi;\etanpsi,\varnpsi)\\
& \leq  \CBFh \left(\|\psi^n\|_{2,\O}
+\|\etanpsi\|_{2,\O} \right) \|\etanpsi\|_{2,\O} \|\varnpsi\|_{2,\O} \\
& \leq \frac{\aAFh\nu}{8}   \|\varnpsi\|^2_{2,\O} 
+ C \nu^{-1}(\|\psi^n\|^2_{2,\O}+\|\etanpsi\|^2_{2,\O}) \|\etanpsi\|^2_{2,\O}.
\end{split}
\end{equation*}

Once again adding and subtracting suitable terms, using Lemma~\ref{boundBFh:two} 
and the Young inequality, we get
\begin{equation*}
\begin{split}
-B_F^h(\varnpsi;\psi^n_h,\varnpsi) &= B_F^h(\varnpsi;(\psi^n-\psi^n_h)-\psi^n,\varnpsi)
= B_F^h(\varnpsi;\etanpsi,\varnpsi)-B_F^h(\varnpsi; \psi^n,\varnpsi)\\
& \leq \CBFh \|\varnpsi\|_{2,\O} \left(\|\etanpsi\|_{2,\O} +\|\psi^n\|_{2,\O} \right) 
\|\varnpsi\|^{\frac{1}{2}}_{2,\O} \|\varnpsi\|^{\frac{1}{2}}_{1,\O}\\
& \leq \frac{\aAFh\nu}{16} \|\varnpsi\|^2_{2,\O} + 2C\nu^{-1} \left(\|\etanpsi\|^2_{2,\O} +\|\psi^n\|^2_{2,\O} \right) \|\varnpsi\|_{2,\O} \|\varnpsi\|_{1,\O}\\
& \leq \frac{\aAFh\nu}{16}\|\varnpsi\|^{2}_{2,\O}+ 2\nu^{-2}C_4\nu^{-1} \left(\|\etanpsi\|^2_{2,\O} +\|\psi^n\|^2_{2,\O} \right)^2 \|\varnpsi\|^{2}_{1,\O} 
+\frac{\aAFh\nu}{16}\|\varnpsi\|^{2}_{2,\O} \\
& \leq \frac{\aAFh\nu}{8}\|\varnpsi\|^{2}_{2,\O} + 4C_4\nu^{-3} \left(\|\etanpsi\|^4_{2,\O} +\|\psi^n\|^4_{2,\O} \right) 
\|\varnpsi\|^{2}_{1,\O}.
	\end{split}
\end{equation*}

Combining the estimates \eqref{preterm:TB}-\eqref{TB2} and the three previous inequalities, we have
\begin{equation}\label{term:TB}
	\begin{split}
T_B  &\leq C_1\nu^{-1} h^{2s} \|\psi^n\|^2_{2+s,\O}\|\psi^n\|^2_{2,\O}
+ \frac{\aAFh\nu}{2} \|\varnpsi\|^2_{2,\O} + C \nu^{-1} \|\psi^n\|^2_{2,\O}\|\etanpsi\|^2_{2,\O}\\
& \qquad+ C \nu^{-1} (\|\psi^n\|^2_{2,\O}+\|\etanpsi\|^2_{2,\O}) \|\etanpsi\|^2_{2,\O}+ C_4 \nu^{-3} \left(\|\etanpsi\|^4_{2,\O} +\|\psi^n\|^4_{2,\O} \right) \|\varnpsi\|^{2}_{1,\O}.
	\end{split}
\end{equation} 

Now, from estimates \eqref{error:eq:NS}, \eqref{term:TF}-\eqref{term:TC} and \eqref{term:TB}, the definition and equivalence of the norm 
$||| \cdot |||_{F,h}$ (cf. \eqref{norm:heatfluid}), together  with  the coercivity of bilinear form  $A^h_F(\cdot,\cdot)$ we have 
\begin{equation}\label{error:estima:NS}
\begin{split}
\frac{1}{2 \Delta t}& \left(|||\varnpsi|||^2_{F,h}- |||\varphi^{n-1}_{\psi}|||^2_{F,h} \right) 
 + \frac{\aAFh\nu}{2}\|\varnpsi\|^{2}_{2,\O}
  \leq  C \left[ 1 +  \nu^{-3}\left(\|\etanpsi\|^4_{2,\O} +\|\psi^n\|^4_{2,\O} \right)  \right] |||\varnpsi|||^2_{F,h} \\
& \quad + C \left[ \nu^{-1} (\|\psi^n\|^2_{2,\O}+\|\etanpsi\|^2_{2,\O})  \right] \|\etanpsi\|^{2}_{2,\O}  
+ \|\gb^n\|^2_{\infty,\O} (\|\etantheta\|^{2}_{0,\O}+ \|\varntheta\|^{2}_{0,\O})\\
&\quad +	Ch^{2s}\left( \|\fb_{\psi}\|^2_{\L^{\infty}(t_{n-1},t_{n}; \H^s(\O_h))}+ \nu^{-1} \|\psi^n\|^2_{2+s,\O} \right)  +h^{2\min\{s,r\}}\|\gb\|^2_{\W^{\min\{s,r\}}_{\infty}(\O)}\|\theta^n\|^2_{r,\O} \\
& \quad +C \|\partial_{tt} \psi \|_{\L^{1}(t_{n-1},t_n; \H^{1}(\O))} |||\varnpsi|||_{F,h} +\frac{C}{\Delta t } h^{s} \|\partial_t \psi \|_{\L^1(t_{n-1},t_n; \H^{1+s}(\O))}|||\varnpsi|||_{F,h}.
\end{split}
\end{equation}

\paragraph{Step 3: Deriving  error estimate for the energy equation~\eqref{error:eq:heat}.}
In this step we will establish estimates for each terms in the error equation~\eqref{error:eq:heat}. 
We start with the term $I_F$, which is bounded by using the Cauchy-Schwarz inequality 
and approximation properties of projection $\Pi^{\ell}_{\E}$, as follows:
\begin{equation}\label{term:IF}
\begin{split}
I_F &:=F_{\theta}^h(\varntheta)-F_{\theta}(\varntheta) 
\leq \frac{C}{2\epsilon} h^{2r}\|f_{\theta}\|^2_{\L^{\infty}(t_{n-1},t_{n}; \H^r(\O_h))} + \frac{\epsilon}{2} \|\varntheta\|^2_{0,\O}.
	\end{split}
\end{equation} 
For the  term $I_M$, we proceed similarly as in \cite[Theorem 3.3]{VB15} to obtain
\begin{equation}\label{term:IM}
	\begin{split}
I_M&:=M_T(\partial_t\theta^n, \varntheta)-M_F^h\Big(\frac{\theta_I^n
-\theta_I^{n-1}}{\Delta t},\varntheta\Big)	\\
&\leq C\|\partial_{tt} \theta\|_{\L^{1}(t_{n-1},t_n; \L^{2}(\O))}\|\varntheta\|_{0,\O} %\\& \quad
+\frac{C}{\Delta t} h^{r} \|\partial_t \theta \|_{\L^1(t_{n-1},t_n; \H^{r}(\O))}\|\varntheta\|_{0,\O}.
\end{split}
\end{equation}
Analogously, as in \eqref{preterm:TB} we split the term $I_B$ as follows:
\begin{equation}\label{preterm:IB}
\begin{split}
I_B&:=B_{\skew}(\psi^n;\theta^n,\varntheta)-B_{\skew}^h(\psi_h^n;\theta_h^n,\varnpsi) 
= \Big( B_{\skew}(\psi^n;\theta^n,\varntheta)-B_{\skew}^h(\psi^n;\theta^n,\varntheta)\Big)\\
& \qquad+ \Big(  B^h_{\skew}(\psi^n;\theta^n,\varntheta)-B_{\skew}^h(\psi_h^n;\theta_h^n,\varntheta) \Big)
=: I_{B1} + I_{B2}.
\end{split}
\end{equation} 
Now, applying the bound~\eqref{prop:Bskew-Bskewh}, with $\gamma=\min\{s,r\}$
and using the Young inequality, we obtain
\begin{equation}\label{IB1}
\begin{split}
I_{B1} &:=B_{\skew}(\psi^n;\theta^n,\varntheta)-B_{\skew}^h(\psi^n;\theta^n,\varntheta) \leq Ch^{\min\{s,r\}}\|\psi^n\|_{2+s,\O}\|\theta^n\|_{1+r,\O}\|\varntheta\|_{1,\O}\\
&\leq C\kappa^{-1}h^{2\min\{s,r\}}\|\psi^n\|^2_{2+s,\O} \|\theta^n\|^2_{1+r,\O} + \frac{\aATh \kappa}{10} \|\varntheta\|^2_{1,\O}. 
\end{split}	
\end{equation} 
On the other hand, similarly as in \eqref{TB2} and \eqref{term:TB}, we can derive
\begin{equation}\label{IB2}
\begin{split}
I_{B2}
& =  B_{\skew}^h(\psi^n;\etantheta ,\varntheta)+ B_{\skew}^h(\etanpsi;\theta^n_h,\varntheta)- B_{\skew}^h(\varnpsi;\theta_h^n,\varntheta)\\
& \leq \frac{\aATh \kappa}{10} \|\varntheta\|^2_{1,\O} + C \kappa^{-1} \|\psi^n\|^2_{2,\O}\|\etantheta\|^2_{1,\O} 
+  \frac{\aATh \kappa}{10} \|\varntheta\|^2_{1,\O} 
+ C \kappa^{-1}(\|\theta^n\|^2_{1,\O}+\|\etantheta\|^2_{1,\O}) \|\etanpsi\|^2_{2,\O}\\
& \qquad - B_{\skew}^h(\varnpsi;\theta_h^n,\varntheta).
	\end{split}
\end{equation}
However, since the discrete trilinear form $B^h_{\skew}(\cdot;\cdot,\cdot)$ does not satisfies an analogous property to Lemma~\ref{boundBFh:two}, we will bound the last term in \eqref{IB2} by a different way. 
Indeed, adding and subtracting adequate terms, using the definition of 
trilinear form, the H\"{o}lder inequality and employing the continuity
of the $\L^2$-projections involved,  we obtain  
\begin{equation}\label{III}
\begin{split}
-B_{\skew}^h(\varnpsi;\theta_h^n,\varntheta) 
&= B_{\skew}^h(\varnpsi;\etantheta,\varntheta)+ B_{\skew}^h(\varnpsi;-\theta^n,\varntheta)\\ 
&= \frac{1}{2} \left( B_{T}^h(\varnpsi;\etantheta,\varntheta) -B_{T}^h(\varnpsi;\varntheta,\etantheta) \right) + B_{\skew}^h(\varnpsi;-\theta^n,\varntheta)\\
& \leq C\sum_{\E \in \CT_h} \|\PiellK \nabla\etantheta\|_{\L^{\infty}(\E)} \| \curl\varnpsi\|_{0,\E}\|\varntheta\|_{0,\E} \\ 
& \quad + C\sum_{\E \in \CT_h}\|\Pi^{\ell-1}_{\E} \etantheta\|_{\L^{\infty}(\E)} \| \curl\varnpsi\|_{0,\E}\|\nabla \varntheta\|_{0,\E}
 + B_{\skew}^h(\varnpsi;-\theta^n,\varntheta).
	\end{split}
\end{equation}

Now, applying an inverse inequality for polynomials, the continuity
of $\PiellK$, and Proposition~\ref{approx-virtual-temp}, 
for $r\geq 1$ we get
\begin{equation*}
\begin{split}
\|\PiellK \nabla\etantheta\|_{\L^{\infty}(\E)} 
&\leq C h^{-1}_{\E}\|\PiellK \nabla\etantheta\|_{0,\E} \leq C h^{-1}_{\E}\|\etantheta\|_{1,\E} \leq C \|\theta^n\|_{1+r,\E} \leq \CR.   
\end{split}
\end{equation*}
Analogously, we have that 
\begin{equation*}
\|\Pi^{\ell-1}_{\E}\etantheta\|_{\L^{\infty}(\E)}  \leq C \|\theta^n\|_{1+r,\E} \leq \CR. 
\end{equation*}
Next, under assumption $\theta^n \in \W^1_{\infty}(\O)$, the definition of the form 
$B_{\skew}^h(\cdot;\cdot,\cdot)$ and the Cauchy-Schwarz inequality, we get 
\begin{equation*}%\label{IB2-3}
\begin{split}
 B_{\skew}^h(\varnpsi;-\theta^n,\varntheta) \leq C_1\|\theta^n\|_{\W^1_{\infty}(\O)} \|\varnpsi\|_{1,\O}\|\varntheta\|_{0,\O} \leq \CR \|\varnpsi\|_{1,\O}\|\varntheta\|_{0,\O}.
\end{split}
\end{equation*}
Inserting the above estimates in \eqref{III}, and applying the Cauchy-Schwarz and Young inequalities,
it follows 
\begin{equation}\label{III-2}
\begin{split}
-B_{\skew}^h(\varnpsi;\theta_h^n,\varntheta) &\leq 3\CR\|\varnpsi\|_{1,\O}\|\varntheta\|_{1,\O} 
\leq C\kappa^{-1}\|\varnpsi\|^2_{1,\O}+ \frac{\aATh\kappa}{10}\|\varntheta\|^2_{1,\O}.
\end{split}
\end{equation}
Then, combining the estimates \eqref{preterm:IB}, \eqref{IB1}, \eqref{IB2} and  \eqref{III-2}, we obtain
\begin{equation}\label{term:IB}
\begin{split}
I_B  &\leq C\kappa^{-1}h^{2\min\{s,r\}}\|\psi^n\|^2_{2+s,\O} \|\theta^n\|^2_{1+r,\O}  +C \kappa^{-1} \|\psi^n\|^2_{2,\O}\|\etantheta\|^2_{1,\O} \\
 & \quad  + C \kappa^{-1}(\|\theta^n\|^2_{1,\O}+\|\etantheta\|^2_{1,\O}) \|\etanpsi\|^2_{2,\O}
  + \frac{4\aATh\kappa}{10}\|\varntheta\|^2_{1,\O}+C(\kappa^{-1}+1) \|\varnpsi\|^2_{1,\O}.\\
\end{split}
\end{equation}

Now, for the term $I_A$, we add and subtract $\theta^n_{\pi} \in \P_{\ell}(\E)$ such that Proposition~\ref{app-pol} 
holds true, then applying the consistency property of $A^{h,\E}_T(\cdot,\cdot)$, the triangle inequality and  Proposition~\ref{approx-virtual-temp}, we have that 
\begin{equation}\label{term:IA}
\begin{split}
I_A & = \kappa \sum_{\E\in\CT_h}\left(A^{\E}_{T}(\theta^n,\varntheta)- A^{h,\E}_{T}(\theta_I^n,\varntheta)\right)
= \kappa  \sum_{\E\in\CT_h}\left(A^{\E}_{T}(\theta^n-\theta^n_{\pi},\varntheta)+ A^{h,\E}_{T}(\theta^n_{\pi}-\theta_I^n,\varntheta)\right)\\
& \leq C\kappa h^{r}\|\theta^n\|_{1+r,\O} \|\varntheta\|_{1,\O}\\
& \leq Ch^{2r}\|\theta^n\|^2_{1+r,\O} + \frac{\aATh\kappa}{10} \|\varntheta\|^2_{1,\O}. 	 
\end{split}
\end{equation}

Now, from bounds \eqref{error:eq:heat}, \eqref{term:IF}, \eqref{term:IM}, \eqref{term:IB} and \eqref{term:IA}, the definition and equivalence of the norms  
$||| \cdot |||_{T,h}$ (cf.~\eqref{norm:heatfluid}) and $\|\cdot\|_{0,\O}$, together  with  the coercivity of bilinear form $A^h_T(\cdot,\cdot)$, we have 
\begin{equation}\label{error:estima:heat}
\begin{split}
\frac{1}{2 \Delta t}& \left(|||\varntheta|||^2_{T,h}- |||\varphi^{n-1}_{\theta}|||^2_{T,h} \right) 
+  \frac{\aATh \kappa}{2} \|\varntheta\|^2_{1,\O} 
\leq  C\|\varntheta\|^{2}_{0,\O} +  \kappa^{-1} \|\theta^n\|^2_{1,\O} \|\etanpsi\|^{2}_{2,\O}\\
& \quad + C \left[ \kappa^{-1} (\|\psi^n\|^2_{2,\O}  +\|\etanpsi\|^2_{2,\O}) \right] \|\etantheta\|^{2}_{1,\O}  +C|||\varntheta|||^2_{T,h} \\
&\quad +	Ch^{2r} \|f_{\theta}\|^2_{\L^{\infty}(t_{n-1},t_{n}; \H^r(\O_h))} +C_1\kappa^{-1}h^{2\min\{s,r\}} \|\psi^n\|^2_{2+s,\O}\|\theta^n\|^2_{2+r,\O} \\
&  \quad +	C\|\partial_{tt} \theta\|_{\L^{1}(t_{n-1},t_n; \L^{2}(\O))} |||\varntheta|||_{T,h} +\frac{C}{\Delta t} h^{r} \|\partial_t \theta \|_{\L^1(t_{n-1},t_n; \H^{r}(\O))}|||\varntheta|||_{T,h}. 
	\end{split}
\end{equation}

\paragraph{Step 4: Combining the steps 2, 3 and  the discrete Gronwall inequality.}
In this  last part, we  combine Steps $2$ and $3$. Indeed, we proceed to multiply by $2\Delta t$ the estimates \eqref{error:estima:NS} and \eqref{error:estima:heat}, then employing the Young inequality to the resulting bounds
and  iterating $j=0,\ldots,n$, we have 

\begin{equation*}%\label{error:I}
\begin{split}
|||\varnpsi|||^2_{F,h} & +|||\varntheta|||^2_{T,h} 
+ \Delta t \sum_{j=0}^n  \|\varphi_{\psi}^j\|^2_{2,\O} 
+  \Delta t \sum_{j=0}^n \|\varphi_{\theta}^j\|^2_{1,\O} \\
& \leq  C \Delta t \sum_{j=0}^n\left[ 1  + \nu^{-3}\left(\|\eta_{\psi}^j \|^4_{2,\O} +\|\psi^j\|^4_{2,\O} \right)  \right] |||\varphi_{\psi}^j|||^{2}_{F,h} %\\& \quad 
+  C \Delta t \sum_{j=0}^n\left[ 1 +\|\gb^j\|^2_{\infty,\O} \right] ||| \varphi_{\theta}^j|||^{2}_{T,h} \\
& \quad + C \Delta t \sum_{j=0}^n\left[ \nu^{-1} (\|\psi^j\|^2_{2,\O}+\|\eta_{\psi}^j\|^2_{2,\O})+ \kappa^{-1} \|\theta^j\|^2_{1,\O}\right] \|\eta_{\psi}^j\|^{2}_{2,\O} \\
& \quad 
+ C\Delta t \sum_{j=0}^n \left[ \kappa^{-1} (\|\psi^j\|^2_{2,\O}  +\|\eta^j_{\psi}\|^2_{2,\O})+\|\gb^j\|^2_{\infty,\O} \right] \|\eta^j_{\theta}\|^{2}_{1,\O}\\
&\quad +C\Delta t h^{2s}\left( \|\fb_{\psi}\|^2_{\L^{\infty}(0,t_{n}; \H^s(\O_h))}+  \|\partial_t \psi \|^2_{\L^1(0,t_n; \H^{1+s}(\O))} + \nu^{-1} \|\psi\|^2_{\L^{\infty}(0,t_n;\H^{2+s}(\O))} \right)\\
& \quad  +\Delta t h^{2\min\{s,r\}}\left(\|\gb\|^2_{\L^{\infty}(0,t_n;\W^{\min\{s,r\}}_{\infty}(\O))} \right) \|\theta\|^2_{\L^{\infty}(0,t_n;\H^{r}(\O))} \\
&\quad +C \Delta th^{2r}\left( \|f_{\theta}\|^2_{\L^{\infty}(0,t_{n}; \H^r(\O_h))} +\|\partial_t \theta \|^2_{\L^1(0,t_n; \H^{r}(\O))}\right) \\
& \quad +C \Delta t \kappa^{-1}h^{2\min\{s,r\}}\left(  \|\psi\|^2_{\L^{\infty}(0,t_n; \H^{2+s}(\O))} + \|\theta\|^2_{\L^{\infty}(0,t_n; \H^{1+r}(\O))} \right) \\
&  \quad +	C \Delta t^2\left( \|\partial_{tt} \theta\|^2_{\L^{1}(0,t_n; \L^{2}(\O))} + \|\partial_{tt} \psi \|^2_{\L^{1}(0,t_n; \H^{1}(\O))} \right)
+ \aMFh \|\varphi^{0}_{\psi}\|^2_{1,\O} + \aMTh \|\varphi^{0}_{\theta}\|^2_{0,\O}.
	\end{split}
\end{equation*}

Thus, applying the discrete Gronwall inequality (cf. Lemma \ref{discrete:gronwall}), choosing $(\psi^0_h,\theta^0_h)=(\psi_I(0),\theta_I(0))$ and using   Propositions~\ref{approx-virtual} and \ref{approx-virtual-temp} along with the equivalence of norms, we have
\begin{equation*}
\begin{split}
( \|\varphi_{\psi}^n\|^2_{1,\O} +   \|\varphi_{\theta}^n\|^2_{0,\O}) +  \Delta t \sum_{j=1}^n (\|\varphi_{\psi}^j\|^2_{2,\O}+  \|\varphi_{\theta}^j\|^2_{1,\O})	
 \leq C(h^{2\min\{s,r\}} + \Delta t^2),
	\end{split}
\end{equation*}
with  $\frac{1}{2}<s \leq k-1$, $1\leq r \leq \ell$ and $C>0$ is independent of mesh
size $h$ and time step $\Delta t$.

Finally, the desired result follows from the above estimate, 
triangular inequality, together with Propositions~\ref{approx-virtual-temp} 
and~\ref{prop:engy:bihar}.
\end{proof}

\begin{remark}
In the present framework, the main advantage of using an energy projector $\mS_h\psi^n$,
as we do for the stream-function space, is to obtain a shorter proof.
Nevertheless, for the temperature variable we do not use an energy projector, but resort to a standard interpolant $\theta_I^n$. 
The reason is that we need also some local approximation properties for the temperature field that the energy projection operator, 
being global in nature, would not have.
\end{remark}

% -------------------------------------------------------------------------------------
\setcounter{equation}{0}	
\section{Numerical results}\label{Numerical:experiments}
% -------------------------------------------------------------------------------------

In this section we carry out numerical experiments in order 
to support our analytical results and illustrate the 
performance of the proposed fully-discrete virtual scheme~\eqref{eq:VEM2}
for the Boussinesq system. In all examples, we use
the lowest order virtual element spaces $\W_2^h$ and $\H_1^h$, for 
the stream-function and temperature fields, respectively. 
At each  discrete time, the nonlinear fully-discrete system \eqref{eq:VEM2} is linearized 
by using the Newton method. 
 For the first time step, we take as initial guess 
$(\psi_h^{{\tt in}},\theta_h^{{\tt in}})=(0,0)$, and  for all $n \geq 1$ 
we take	$(\psi_h^{{\tt in}},\theta_h^{{\tt in}}) = (\psi_h^{n-1},\theta_h^{n-1})$.
The iterations are finalized when the 
$\ell^{\infty}$-norm of the global incremental discrete solution drop 
below a fixed tolerance of $\Tol=10^{-8}$.

 The domain $\Omega$ is partitioned using the following sequences of polygonal meshes (an example for each family is shown in Figure~\ref{fig:meshsamples}):
\begin{multicols}{2}
\begin{itemize}
	\item $\CT^1_h$: Distorted quadrilaterals meshes;
	\item $\CT^2_h$: Triangular meshes;
	\item $\CT^3_h$: Voronoi meshes; %	Tessellation;
	\item $\CT^4_h$: Distorted concave rhombic quadrilaterals.
\end{itemize}  
\end{multicols}

%%___________________ Figure 1
\begin{figure}[!h]%[htpb!]
	\centering % [scale=0.3]
	\subfigure[mesh $\CT^1_h$]{\includegraphics[width=0.23\linewidth, height=0.23\textwidth]{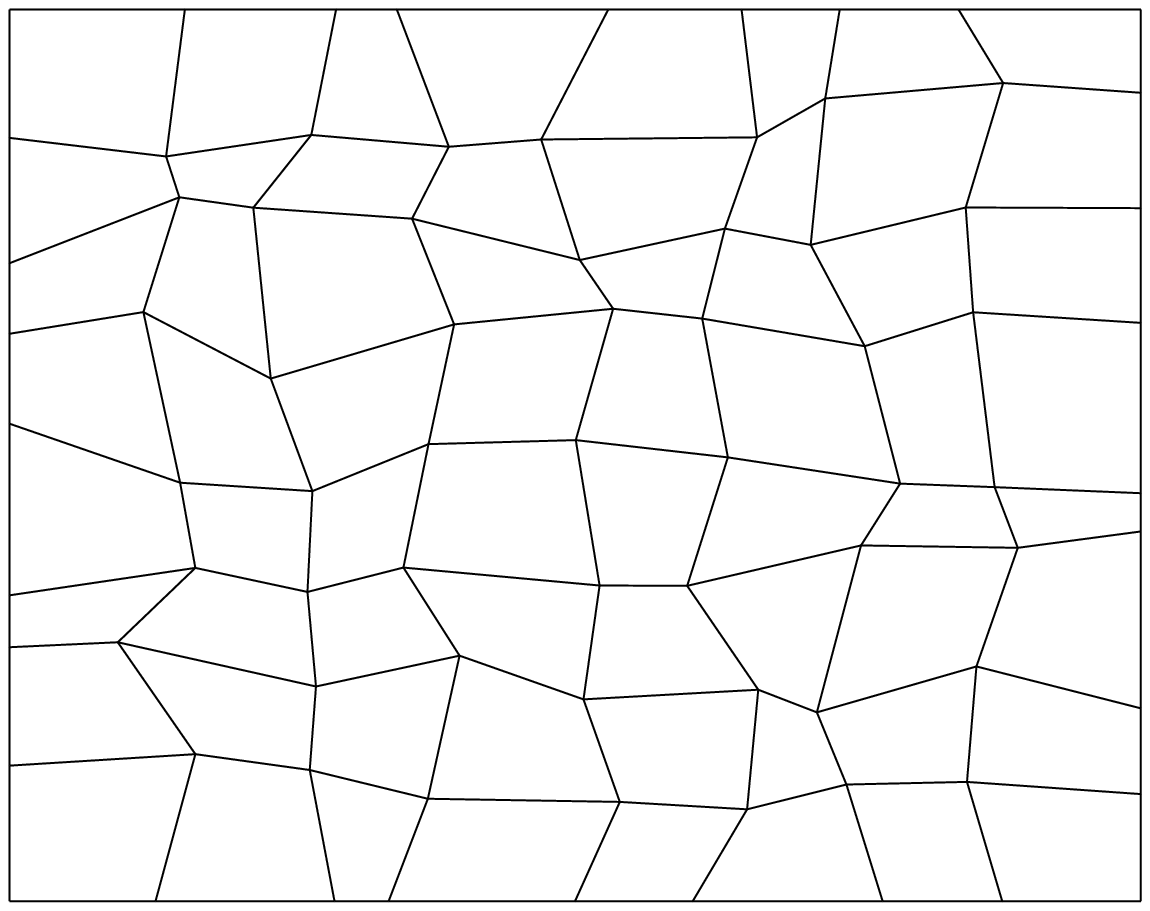}}\hspace*{-0.35cm}
	\subfigure[mesh $\CT^2_h$]{\includegraphics[width=0.23\linewidth, height=0.23\textwidth]{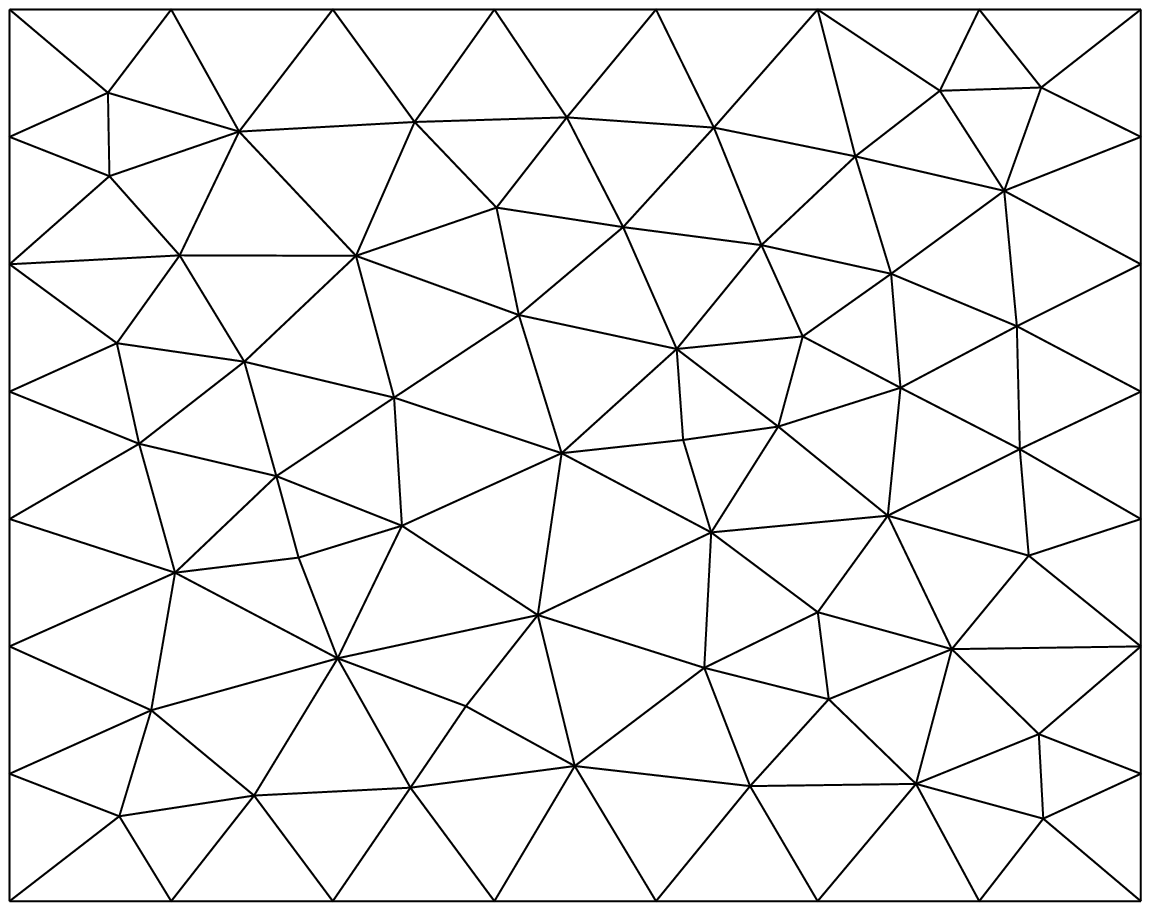}}\hspace*{-0.35cm}
	\subfigure[mesh $\CT^3_h$]{\includegraphics[width=0.23\linewidth, height=0.23\textwidth]{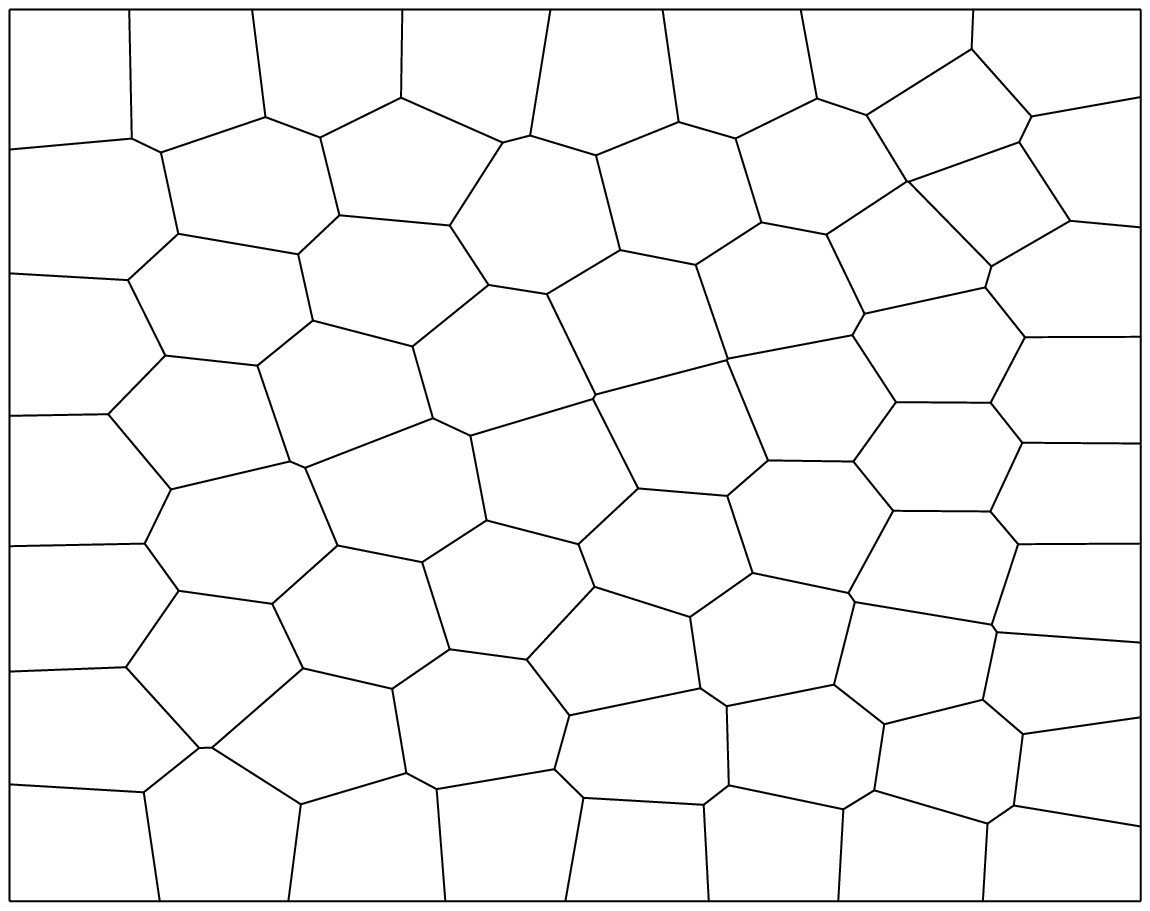}}\hspace*{-0.35cm}
	\subfigure[mesh $\CT^4_h$]{\includegraphics[width=0.23\linewidth, height=0.23\textwidth]{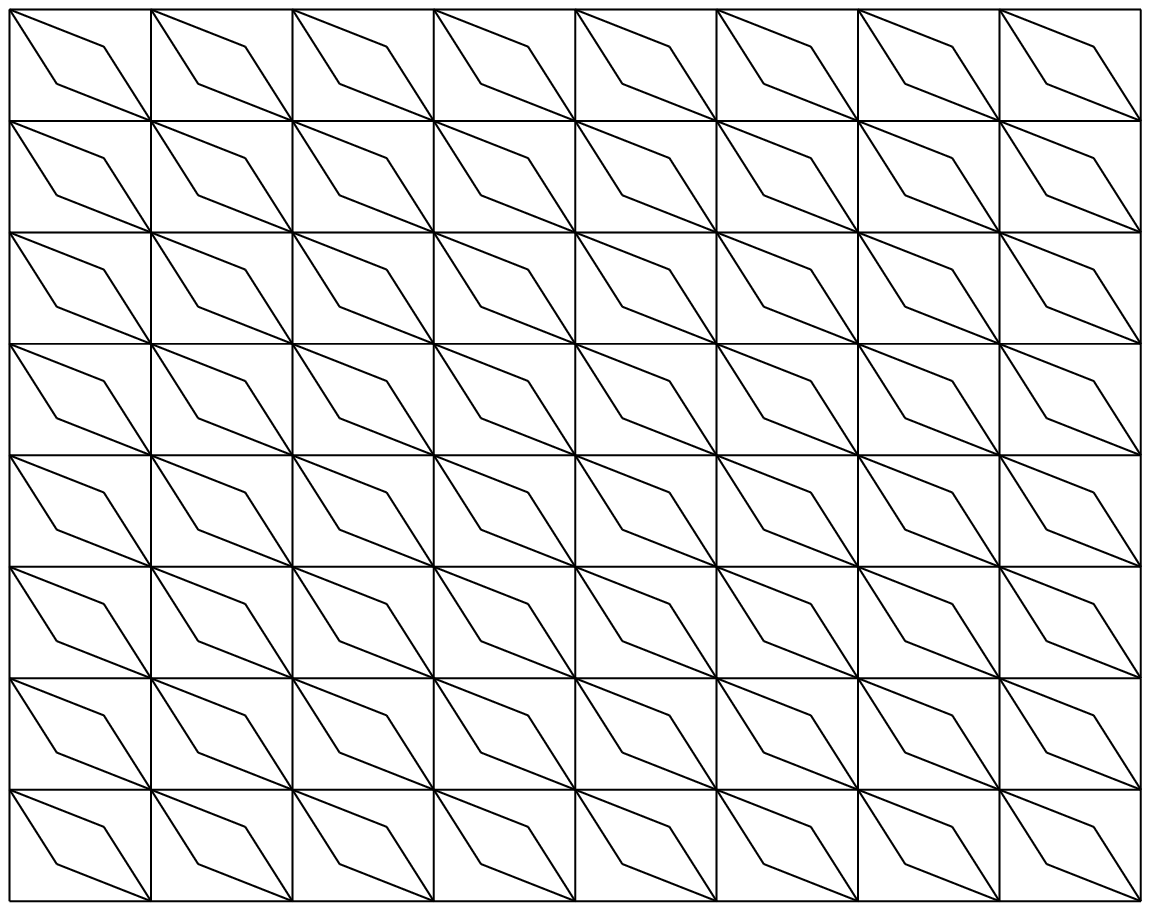}}
	\caption{Domain discretized with different meshes.}
	\label{fig:meshsamples}
\end{figure} 
%%____________________________________

In order to test the convergence properties of the proposed VEM,
we measure some errors as the difference between
the exact solutions $(\psi, \theta)$ and adequate projections
of the numerical solution $(\psi^n_h,\theta^n_h)$. More precisely, 
we consider the following quantities:
\begin{align}
	{\tt E}(\psi,\L^2, \H^2)
	&:=\Big(\Delta t  \sum_{n=1}^{N} |\psi(t_n)-
	\Pi^{\D,2} \psi^n_h|^2_{2,h} \Big)^{1/2},\qquad
	{\tt E}(\theta,\L^2, \H^1)
	:=\Big(\Delta t  \sum_{n=1}^{N} |\theta(t_n)-
	\Pi^{\nabla,1} \theta^n_h|^2_{1,h} \Big)^{1/2}, \label{error:L2Hi}\\
	{\tt E}(\psi,\L^{\infty}, \H^1)
	&:=|\psi(T)-\Pi^{\D,2} \psi^N_h|_{1,h},\qquad \qquad \qquad \quad
	{\tt E}(\theta,\L^{\infty}, \L^2)
	:=\|\theta(T)-\Pi^{\nabla,1} \theta^N_h\|_{0,\O}. \label{error:LinfityHi}
\end{align}
%%% ___________________________________

\noindent
Accordingly to Theorem \ref{converg:fullydisc}, the expected convergence rate for the sum of the above norms is $\mathcal{O}(h + \Delta t)$.

\subsection{Accuracy assessment}
In our first example, we illustrate the accuracy in space and time of the proposed  
VEM~\eqref{eq:VEM2}, considering a manufactured exact solution on the square 
domain $\O := (0,1)^2$, the time interval $[0,1]$ and force per unit mass 
$\gb=(0,-1)^T$. We solve the Boussinesq system~\eqref{Bouss:Eq}, 
taking the load terms $\fb_{\psi}$ and $f_{\theta}$,  boundary and initial 
conditions in such a way that the analytical solution is given by:
\begin{equation*}
	\bu(x,y,t)= \begin{pmatrix}
		u_1(x,y,t) \\
		u_2(x,y,t)
	\end{pmatrix}=\begin{pmatrix}
		\: \:(e^{10(t-1)}-e^{-10})\:x^2(1-x)^2(2y-6y^2+4y^3)\\
		-\: (e^{10(t-1)}-e^{-10})y^2(1-y)^2(2x-6x^2+4x^3)
	\end{pmatrix},  %  \qquad
\end{equation*}
\begin{equation*}
	p(x,y,t) =(e^{10(t-1)}-e^{-10})( \sin(x)\cos(y) + (\cos(1)-1)\sin(1) ), 
\end{equation*}
\begin{equation*}
	\psi(x,y,t)=(e^{10(t-1)}-e^{-10})x^2(1-x)^2y^2(1-y)^2 \qquad  \text{and} \qquad \theta(x,y,t)=u_1(x,y,t)+u_2(x,y,t).
\end{equation*}

In order to see the linear trend of the stream-function and temperature 
errors~\eqref{error:L2Hi},  predicted  by Theorem~\ref{converg:fullydisc}, 
we  refine simultaneously in space and time. 
More precisely, for each mesh family we consider the  mesh refinements with
$h=1/4, 1/8, 1/16, 1/32$, and we use the same uniform refinements  
for the time variable. In particular, for the mesh $\O^1_h$, it can be seen 
along the diagonal of Table~\ref{tabla1-1}, the expected first order  
convergence for the stream-function and temperature errors~\eqref{error:L2Hi}.

In Figure~\ref{fig1-1}, we display the errors~\eqref{error:L2Hi}  for the same simultaneous time and space
refinements ($h=\Delta t =2^{-i}$, with $i=2,\ldots,5$), using the four mesh families. We notice that the rates of convergence predicted in Theorem~\ref{converg:fullydisc}  are attained by both unknowns.

%%%____________________________ Table 1 ___________________________
\begin{table}[!h] 
\begin{center}
%{\small \begin{tabular}{l|l|c|c|c|c|c|c|c|c|c|c}
{\small \begin{tabular}{llcccccccccc}
\hline
\hline \\
\multicolumn{7}{c}{${\tt E}(\psi,\L^2;\H^2)$} \\
\hline
\hline
$\dofs$ & \diagbox[innerwidth=0.8cm]{$h$}{$\Delta t$}& $1/4$ &$1/8$& $1/16$&$1/32$&$1/64$\\
%dof &$h$ &$\Delta t_0$ & $\Delta t_0/2$ & $\Delta t_0/4$& $\Delta t_0/8$ &$\Delta t_0/16 $ \\
\hline  
\hline % $\fbox{1.076690e-3}$
36   &$1/4$  &\fbox{1.88912e-2} &1.42183e-2 &1.16131e-2 &1.02912e-2 &\fbox{9.63665e-3}  \\ %\hline
196  &$1/8$  &1.11333e-2 &\fbox{8.42107e-3} &6.91546e-3 &6.15400e-3 &\fbox{5.77765e-3} \\ %\hline 
900  &$1/16$ &4.92223e-3 &3.53363e-3 &\fbox{2.85747e-3} &2.54826e-3 &\fbox{2.40427e-3}\\ %\hline
3844 &$1/32$ &3.61175e-3 &2.11884e-3 &1.46063e-3 &\fbox{1.21158e-3} &\fbox{1.11670e-3}\\ %\hline
15876&$1/64$ &\fbox{3.21002e-3} &\fbox{1.64565e-3} &\fbox{9.22443e-4}&6.49802e-4 &\fbox{5.59824e-4} \\
\hline\hline 				\\
%&&& \qquad ${\tt E}(\theta,\L^2;\H^1)$ &&\\ [0.5ex]
%\hline
\multicolumn{7}{c}{${\tt E}(\theta,\L^2;\H^1)$} \\
\hline\hline
36   &$1/4$  &\fbox{1.74892e-2} &1.34200e-2 &1.11391e-2 &9.96756e-3  &\fbox{9.38232e-3} \\
196  &$1/8$  &1.02277e-2 &\fbox{7.88174e-3} &6.66404e-3 &6.05736e-3  &\fbox{5.75702e-3} \\
900  &$1/16$ &5.32067e-3 &3.65373e-3 &\fbox{2.93777e-3} &2.64415e-3  &\fbox{2.51594e-3} \\
3844 &$1/32$ &3.80377e-3 &2.18463e-3 &1.49484e-3 &\fbox{1.24874e-3}  &\fbox{1.16084e-3} \\
15876&$1/64$ &\fbox{3.37157e-3} &\fbox{1.71644e-3} &\fbox{9.52229e-4}&6.64250e-4 &\fbox{5.69713e-4}\\
\hline 				
\hline
\end{tabular}}
\end{center} % ${\tt E}(\psi,\L^2;\H^2)$ and  ${\tt E}(\theta,\L^2;\H^1)$
\caption{Accuracy assessment. Errors~\eqref{error:L2Hi}  using the VEM~\eqref{eq:VEM2}, with polynomial 
	degrees $(k,\ell)=(2,1)$, physical parameters $\nu=\kappa=1$ and the mesh family mesh $\CT^1_h$.}
\label{tabla1-1}
\end{table}

%%%______________________ figure 2 ___________________________________

\begin{figure*}[!h]% [hpbt] 
	\centering
	\subfigure[Stream-function errors]{%
		\includegraphics[width=0.40\linewidth, height=0.43\textwidth]{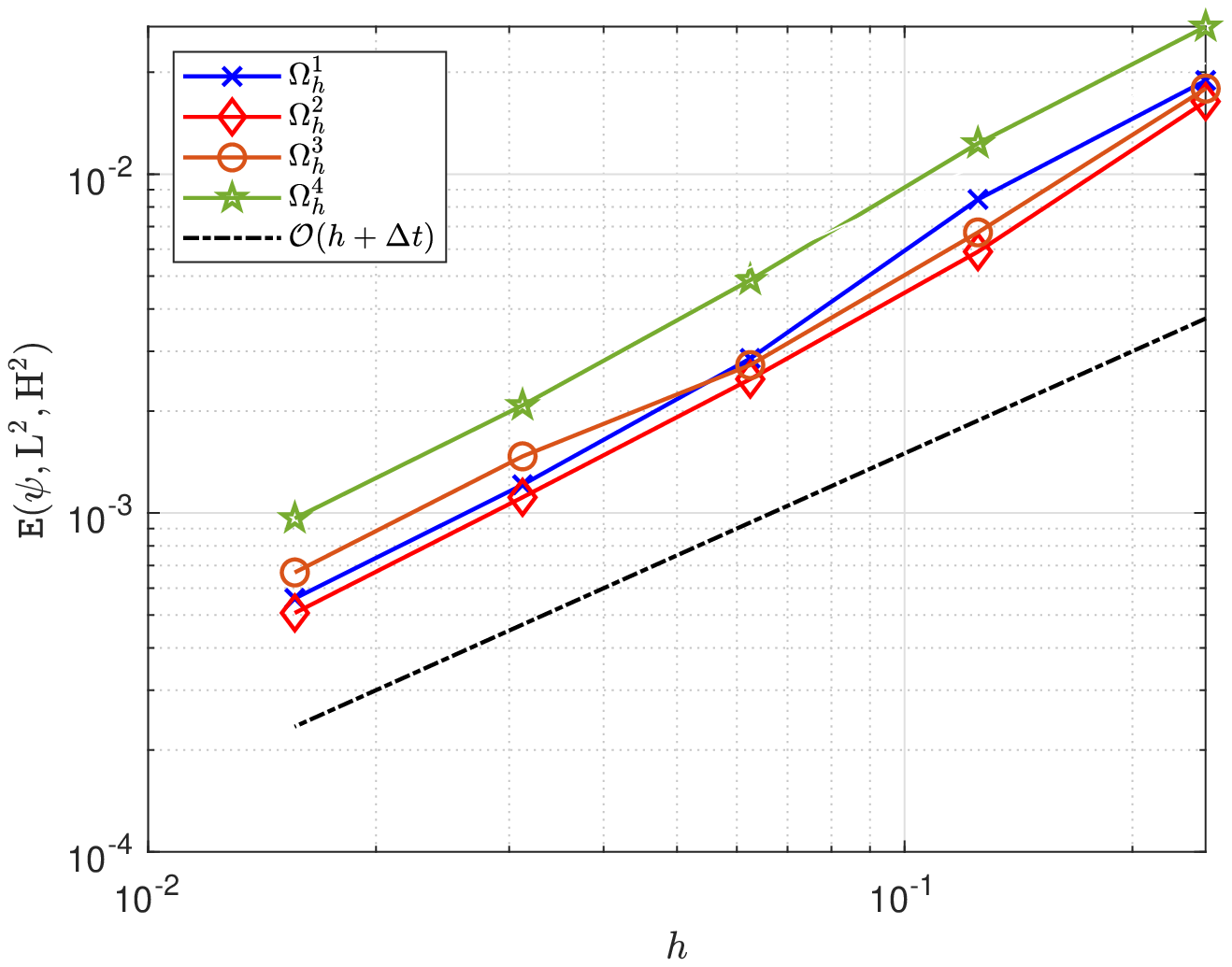}
		\label{errors-stream} } 
	\subfigure[Temperature errors]{%{0.3\textwidth}
		\includegraphics[width=0.40\linewidth, height=0.43\textwidth]{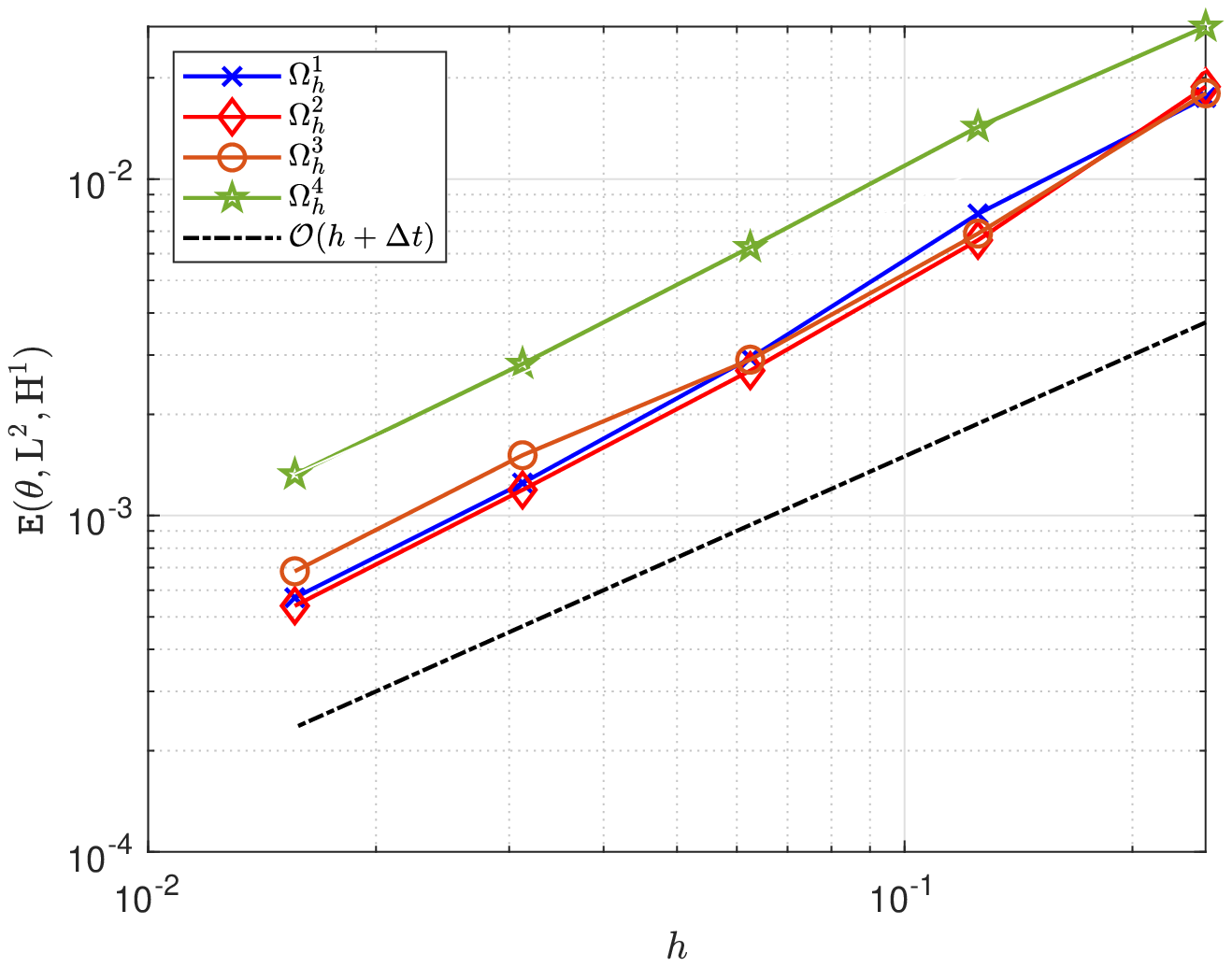}
		\label{errors-temp}} % Stream-function (a) and temperature (b) e
 	\caption{Accuracy assessment. Errors~\eqref{error:L2Hi} for simultaneous space and time refinements,  
 		using the VEM~\eqref{eq:VEM2}  with polynomial degrees  $(k,\ell)=(2,1)$, physical parameters 
 		$\nu=\kappa=1$ and the mesh families $\O^i_h$, $i=1,\ldots, 4$.}
	\label{fig1-1}
\end{figure*}
%%%%___________________

In order to study the trend of the stream-function and temperature errors~\eqref{error:LinfityHi}, 
we show in Table~\ref{tabla1-2} the results considering again the mesh $\O^1_h$, with
$h=\Delta t =2^{-i}$, with $i=2,\ldots,5$. In particular, we can observe that the rate of convergence in the mesh size $h$ seems higher than one; this is not fully surprising, since standard interpolation estimates (in space) for the norms in~\eqref{error:LinfityHi} indicate that, potentially, the discrete space could approximate the exact solution with order $\mathcal{O}(h^2)$. In order to better investigate this aspect, in Figure~\ref{fig1-3} we display the errors~\eqref{error:LinfityHi}  for space and time
refinements given by $h =2^{-i}$ and $\Delta t =4^{-i}$, with  $i=2,\ldots,5$, 
respectively,  using the four mesh families. We notice that the rates of convergence seem indeed quadratic with respect to $h$.

%%%____________________________ Table 2 ___________________________
\begin{table}[!h] 
\begin{center}
%{\small \begin{tabular}{l|l|c|c|c|c|c|c|c|c|c|c}
{\small \begin{tabular}{llcccccccccc}
%\hline
%&&& \qquad ${\tt E}(\psi,\L^2;\H^2)$ &&\\ [1ex] 
\hline
\hline \\
\multicolumn{7}{c}{${\tt E}(\psi,\L^{\infty};\H^1)$} \\
\hline
\hline
$\dofs$ & \diagbox[innerwidth=0.8cm]{$h$}{$\Delta t$}& $1/4$ &$1/8$& $1/16$&$1/32$&$1/64$\\
%dof &$h$ &$\Delta t_0$ & $\Delta t_0/2$ & $\Delta t_0/4$& $\Delta t_0/8$ &$\Delta t_0/16 $ \\
\hline  
\hline  %\fbox{}
36   &$1/4$  &4.30301e-3  & 4.50090e-3 & 4.65255e-3& 4.74590e-3 &4.79749e-3\\ %\hline
196  &$1/8$  &2.03865e-3  & 2.20110e-3 & \fbox{2.33234e-3}& 2.41662e-3 &2.46443e-3\\ %\hline 
900  &$1/16$ &2.38767e-4  & 2.11074e-4 & 3.61809e-4& 4.80109e-4 &\fbox{5.49619e-4}\\ %\hline
3844 &$1/32$ &7.26027e-4  & 4.35284e-4 & 2.05347e-4& 6.71747e-5 &4.99331e-5\\ %\hline
15876&$1/64$ &8.16241e-4  & 5.20174e-4 & 2.84604e-4& 1.34953e-4 &5.10645e-5 \\
\hline\hline 				\\
\multicolumn{7}{c}{${\tt E}(\theta,\L^{\infty};\L^2)$} \\
\hline\hline
36   &$1/4$  &3.44760e-3  & 3.94792e-3 & 4.28939e-3 & 4.48462e-3 &4.58811e-3 \\
196  &$1/8$  &9.85211e-4  & 1.44875e-3 & \fbox{1.82900e-3} & 2.06308e-3 &2.19159e-3\\
900  &$1/16$ &5.96219e-4  & 2.98014e-4 & 3.26274e-4 & 4.64998e-4 &\fbox{5.57065e-4} \\
3844 &$1/32$ &8.26668e-4  & 4.90632e-4 & 2.31786e-4 & 9.52686e-5 &9.44396e-5 \\
15876&$1/64$ &8.90387e-4  & 5.68492e-4 & 3.13988e-4 & 1.53393e-4 &6.48063e-5\\
\hline 				
\hline
\end{tabular}}
\end{center} % ${\tt E}(\psi,\L^2;\H^2)$ and  ${\tt E}(\theta,\L^2;\H^1)$
\caption{Accuracy assessment. Errors~\eqref{error:LinfityHi}  using the VEM~\eqref{eq:VEM2}, with polynomial 
degrees $(k,\ell)=(2,1)$, the physical parameters $\nu=\kappa=1$ and the mesh family $\CT^1_h$.}
\label{tabla1-2}
\end{table}
%%%___________________

%%%___________________ Figure 3 ________________________________

\begin{figure*}[!h]% [hpbt] 
	\centering
	\subfigure[Stream-function errors]{
		\includegraphics[width=0.40\linewidth, height=0.43\textwidth]{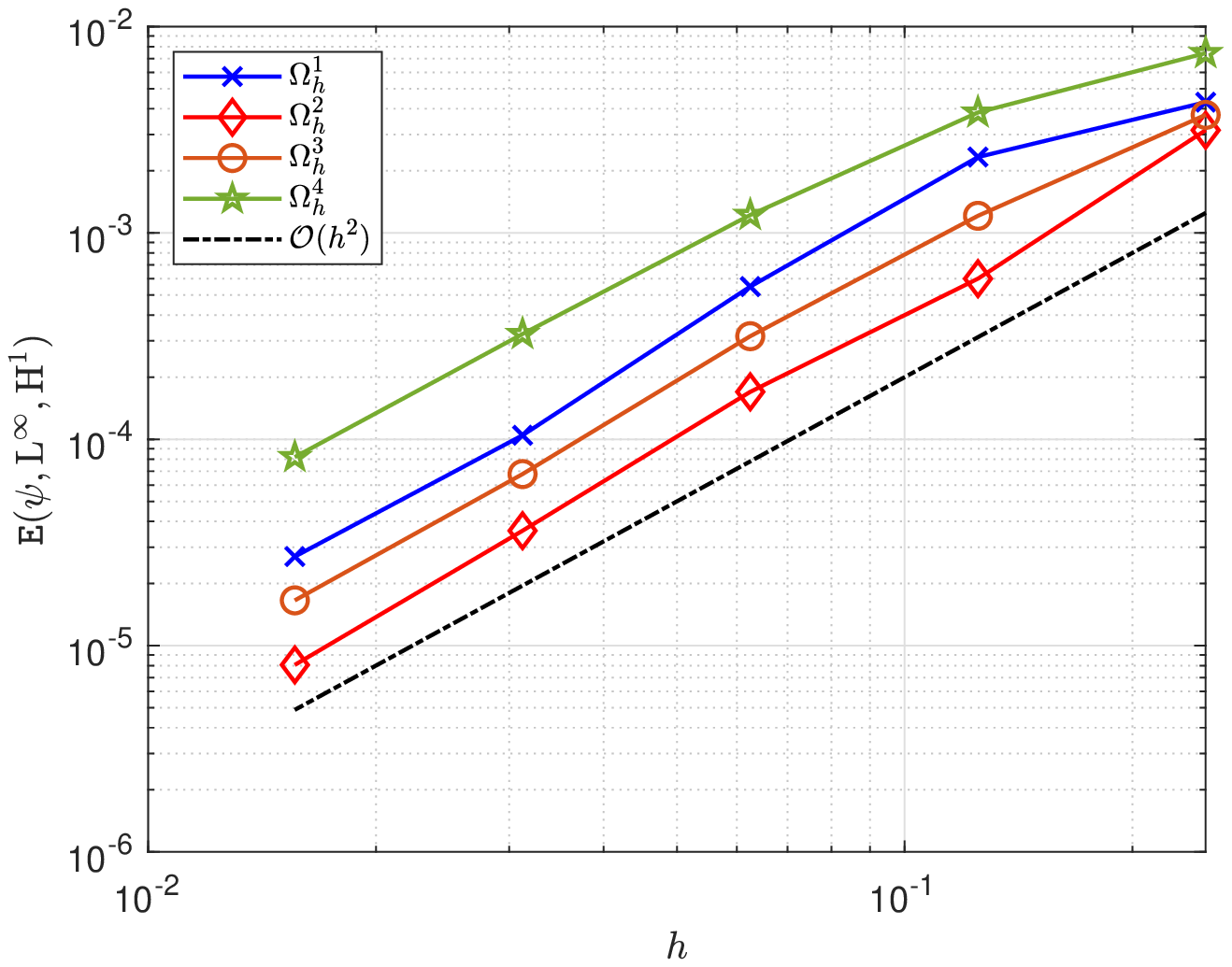}
		\label{errorsfinal-stream} } 
	\subfigure[Temperature errors]{%{0.3\textwidth}
		\includegraphics[width=0.40\linewidth, height=0.43\textwidth]{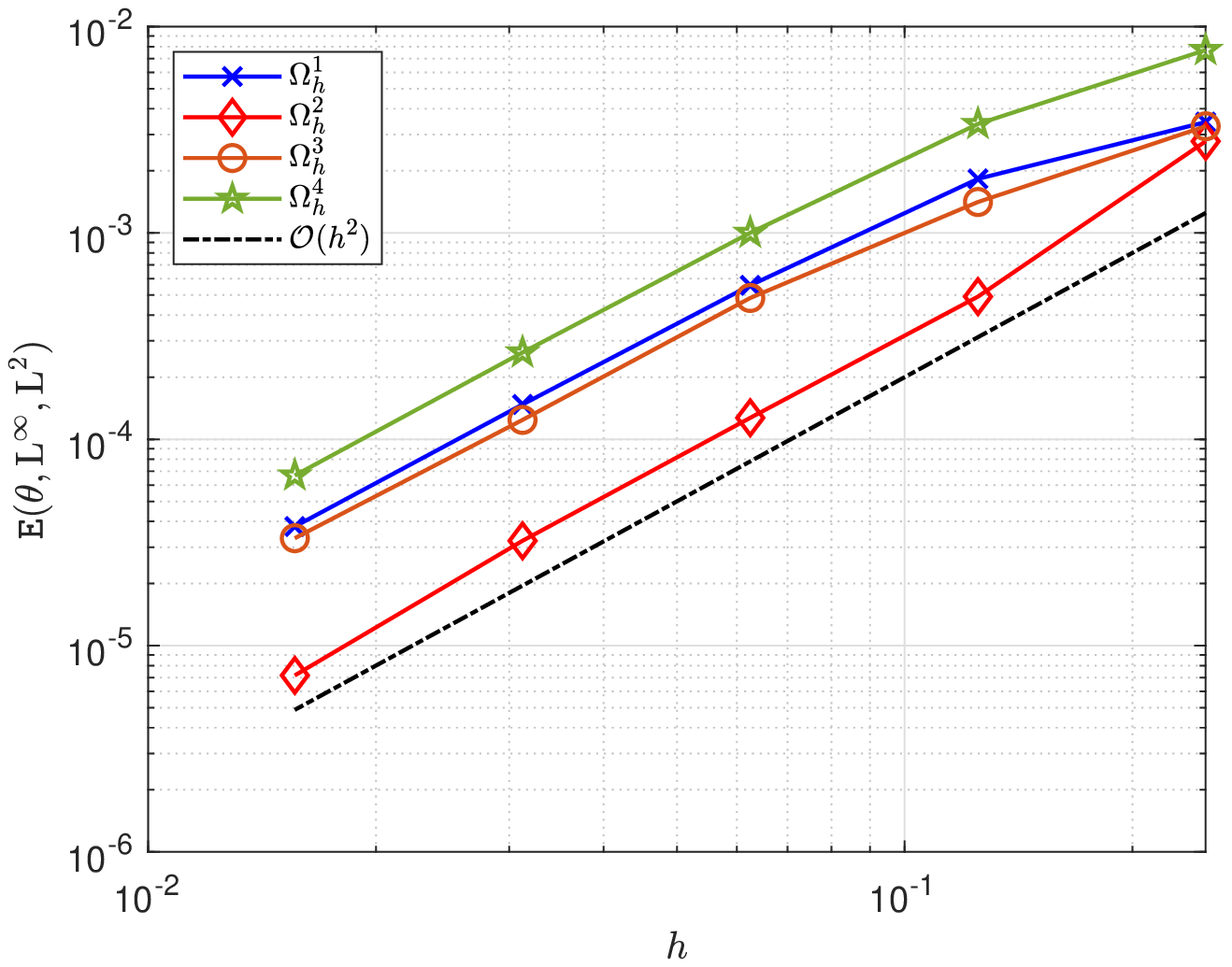}
		\label{errorsfinal-temp}} 
	\caption{Accuracy assessment. Errors~\eqref{error:LinfityHi},  using the VEM~\eqref{eq:VEM2}  with polynomial degrees  $(k,\ell)=(2,1)$, the physical parameters $\nu=\kappa=1$ and the mesh families $\O^i_h$, $i=1,\ldots, 4$.}
	\label{fig1-3}
\end{figure*}
%%%%_____________________________________

\subsection{Performance of the VEM for small viscosity}
In this test we consider the square domain $\O := (0,1)^2$, 
the time interval $[0,1]$  and  force per unit mass  $\gb=(0,-1)^T$. 
We solve the Boussinesq system~\eqref{Bouss:Eq}, taking the load terms $\fb_{\psi}$ 
and $f_{\theta}$, boundary and initial conditions in such 
a way that the analytical solution is given by:
\begin{equation*}
	\bu(x,y,t)= \begin{pmatrix}
		u_1(x,y,t) \\
		u_2(x,y,t)
	\end{pmatrix}=\begin{pmatrix}
		-\cos(t)\sin(\pi x)\sin(\pi y)\\
		-\cos(t)\cos(\pi x)\cos(\pi y)
	\end{pmatrix},  %  \qquad
\end{equation*}
\begin{equation*}
	p(x,y,t) = \cos(t)(\sin(\pi x)+\cos(\pi y)-2/\pi), 
\end{equation*}
\begin{equation*}
	\psi(x,y,t)= \frac{1}{\pi}\cos(t)\sin(\pi x)\cos(\pi y) \qquad  \text{and} \qquad \theta(x,y,t)=u_1(x,y,t)+u_2(x,y,t).
\end{equation*}	

The purpose of this experiment is to investigate the performance of the VEM~\eqref{eq:VEM2} 
for small viscosity parameters. In Figure~\ref{fig2-1}, we post the errors~\eqref{error:L2Hi} of the stream-function variable 
obtained with the mesh sizes $h=1/4,1/8,1/16$ of $\O^2_h$, considering different values of 
$\nu$ and fixing the time step $\Delta t$ as $1/8$  and  $1/16$ (see Figure~\ref{errors-stream-Deltat:1half8} 
and Figure~\ref{errors-stream-Deltat:1half16}, respectively).  
It can observed that the solutions of our VEM are accurate even for small
values of $\nu$. Larger stream-function errors appear for very small viscosity values. 

We observe that this results are in accordance with the general observation that exactly divergence-free Galerkin methods are more robust with respect to small diffusion parameters, see for instance \cite{SL-paper} (and also \cite{BLV-NS18} in the VEM context). On the other hand, note that the scheme here proposed has no explicit stabilization of the convection term since this is not the focus of the present work (for instance, the natural norm associated to the stability of the discrete problem does not guarantee a robust control on the convection).

%%___________________ Figure 3
%%%______________________ figure 2 ___________________________________

\begin{figure*}[hpbt]  % [!h] 
	\centering
	\subfigure[Stream-function errors with $\Delta t =1/8$]{%
		\includegraphics[width=0.5\linewidth, height=0.3\textwidth]{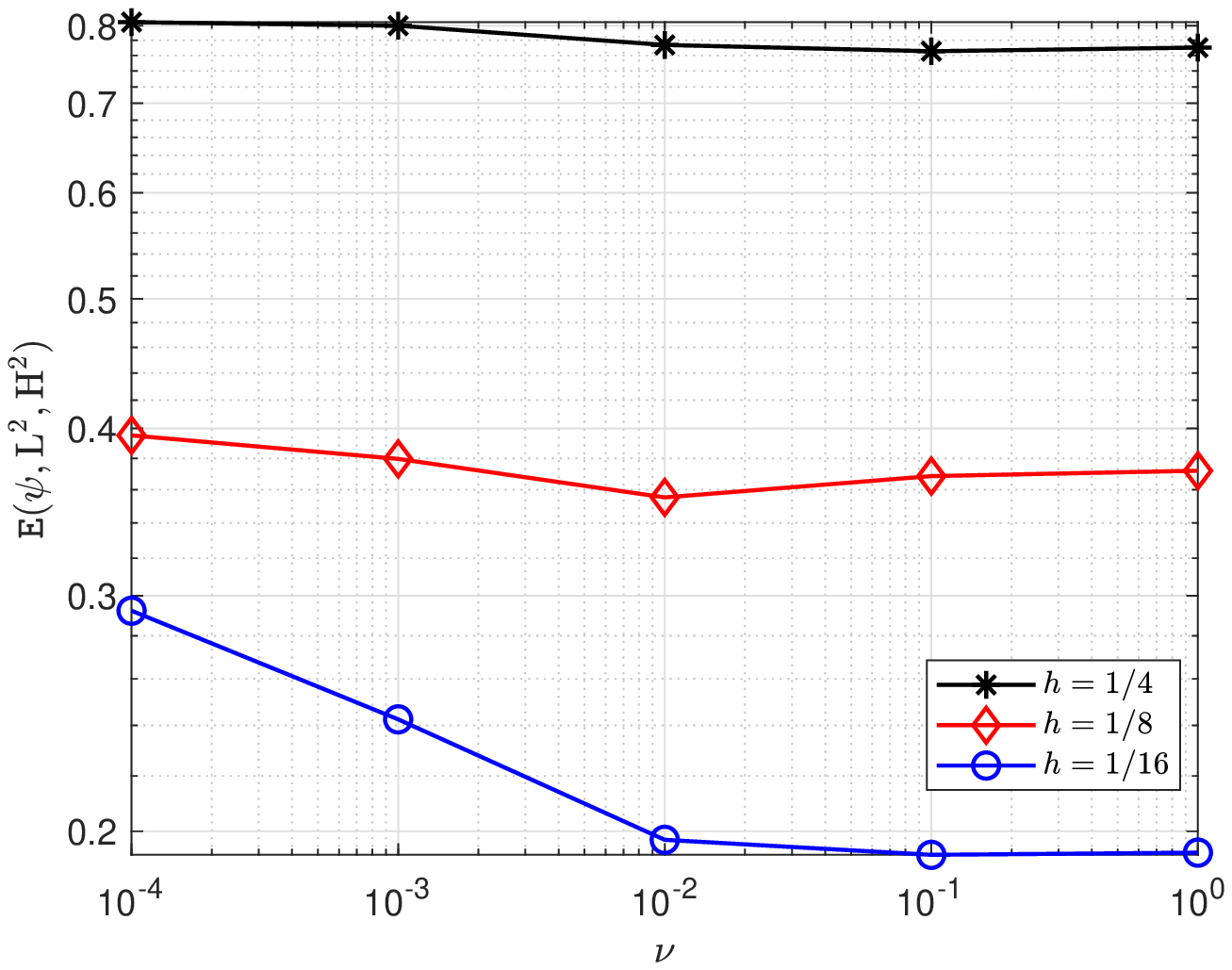}
		\label{errors-stream-Deltat:1half8} } \hspace{-0.7cm}
	\subfigure[Stream-function errors with $\Delta t =1/16$]{%{0.3\textwidth}
		\includegraphics[width=0.5\linewidth, height=0.3\textwidth]{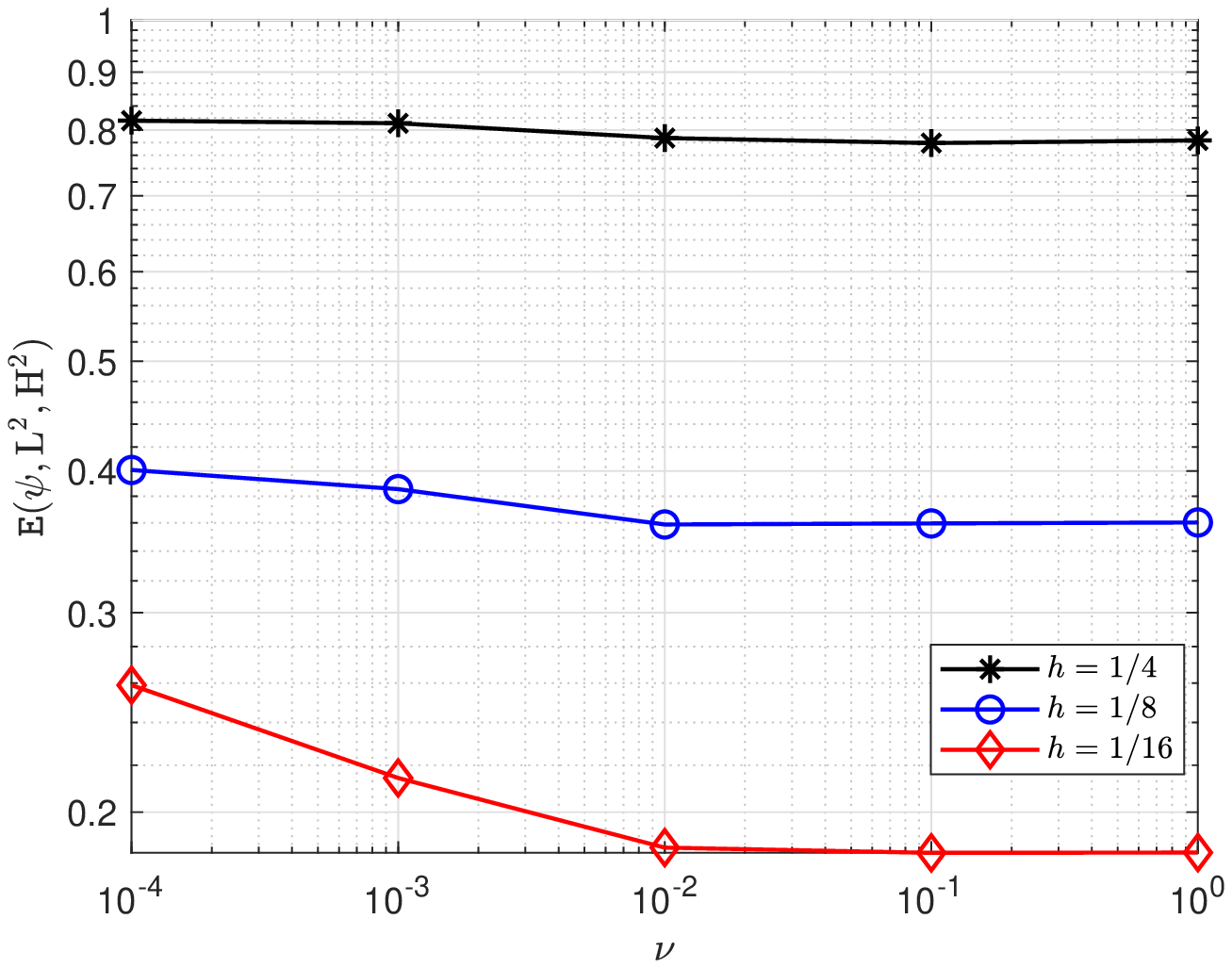}
		\label{errors-stream-Deltat:1half16}} % Stream-function (a) and temperature (b) e
	\caption{Small viscosity test. Errors~\eqref{error:L2Hi} of the VEM~\eqref{eq:VEM2}, for different values of $\nu$ and $\kappa=1$, using the meshes $\O^2_h$, polynomial degrees  $(k,\ell)=(2,1)$.}
	\label{fig2-1}
\end{figure*}
%%%%___________________

\subsection{Natural convection in a cavity with the left wall heating}	
In this last example we consider the $2$D natural convection benchmark problem, 
describing  the behaviour of a incompressible flow in a squared cavity, 
which is heated at the left wall (see \cite{ZHZ2014_camwa,Vahl83,MNZ98,Manzari99,WPW2001}). 
In particular, we consider the  unitary 
square domain $\O=(0,1)^2$. The boundary conditions are given as follows: 
the temperature in the left and right walls are $\theta_{L}=1$ and $\theta_{R}=0$, respectively, 
while in  the horizontal walls is $\partial_{\bn } \theta = 0$ 
(i.e.,  insulated, there is no heat transfer through these walls), no-slip boundary 
conditions are imposed for the fluid flow at all walls. In terms of the stream-function 
these conditions are given by: $\psi = \partial_x\psi =\partial_y \psi=0$ on  $\Gamma$, as shown in Figure~\ref{fig:geometry}. The initial conditions are chosen as  $\psi_0=-x+y$ and $\theta_0=1$ (so that the initial data does not satisfy the boundary conditions). 

We consider the forces $\fb_{\psi}=\0$, $f_{\theta}=0$ and $\gb=\Pr\Ra (0,1)^T$, 
where $\Pr$ and  $\Ra$ denote  the Prandtl and Rayleigh numbers, respectively.
For the numerical experiment, we set the physical parameters as: $\nu=\Pr =0.71$, 
$\Ra \in [10^3,10^6]$ and $\kappa =1$.  

\begin{figure*}[!h]
	\centering
	\subfigure[Boundary conditions]{
\begin{tikzpicture}[scale=0.415]			
	% creating the square
		\draw[-] (0,0) -- (12,0)--(12,12)--(0,12)--cycle;
		
	% writing the boundary conditions for the horizontal walls	
		\node[] at (6.0,-0.8) {$\psi=\partial_x\psi=\partial_y\psi=0$, $\partial_{\bn} \theta=0$};
		\node[] at (6.0,12.5) {$\psi=\partial_x\psi=\partial_y\psi=0$, $\partial_{\bn} \theta=0$};

			% writing the boundary conditions for the verical walls
		\node[] at (-2.0,7.0) {$\quad \psi=0$};
		\node[] at (-2.0,6.0) {$\partial_x\psi=0$};
        \node[] at (-2.0,5.0) {$\partial_y\psi=0$};
        \node[] at (-2.0,4.0) {$\: \:\theta_L=1$};
        
        \node[] at (14.0,7.0) {$\quad\psi=0$};
        \node[] at (14.0,6.0) {$\partial_x\psi=0$};
        \node[] at (14.0,5.0) {$\partial_y\psi=0$};
        \node[] at (14.0,4.0) {$\: \: \theta_R=0$};

	\end{tikzpicture}
\label{BC} }
\subfigure[mesh $\O^5_h$ with $h=1/8$]{%{0.3\textwidth}
	\includegraphics[width=0.37\linewidth, height=0.35\textwidth]{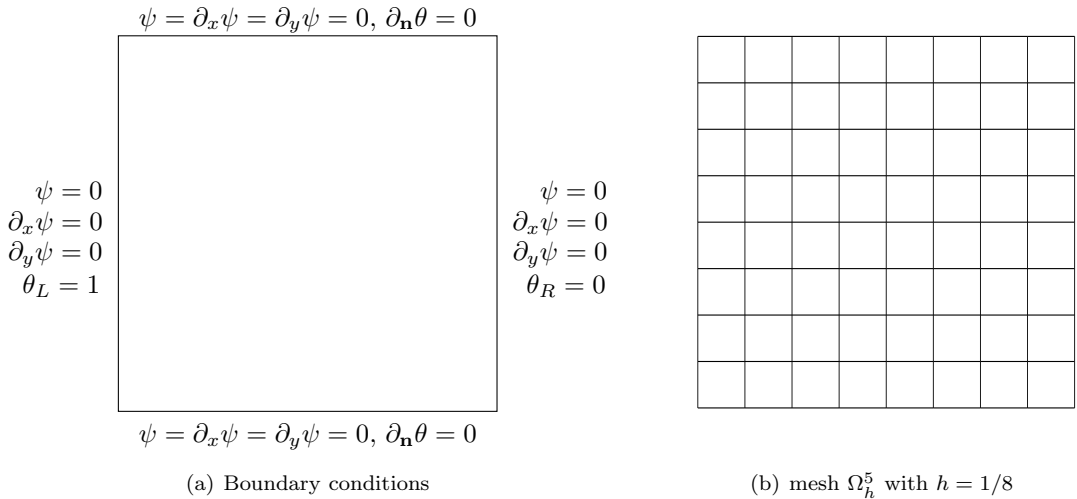}
	\label{mesh:square}} 
	\vspace*{0.2cm}
	\caption{Natural convection cavity. Boundary conditions and domain discretized with mesh $\O^5_h$.}
	\label{fig:geometry}
\end{figure*}

In order to compare our results with the existing bibliographic, 
we decompose the domain $\O$ using  mesh $\O^5_h$ conformed by uniform squares (see Figure~\ref{mesh:square}).
Moreover, the time step is $\Delta t = 10^{-3}$ and final time $T=1$.

Streamlines  and isotherms of the discrete solution obtained with our VEM~\eqref{eq:VEM2}
are posted in Figure~\ref{stream-isotherm}, using $\Ra = 10^3, 10^4, 10^5, 10^6$ and mesh size $h=1/64$.
The results show well agreement with the results presented in the benchmark solutions in~\cite{ZHZ2014_camwa,Vahl83,MNZ98,Manzari99,WPW2001}.

\begin{figure}[!h]
	\begin{center}
		\includegraphics[height=4.6cm, width=4.6cm]{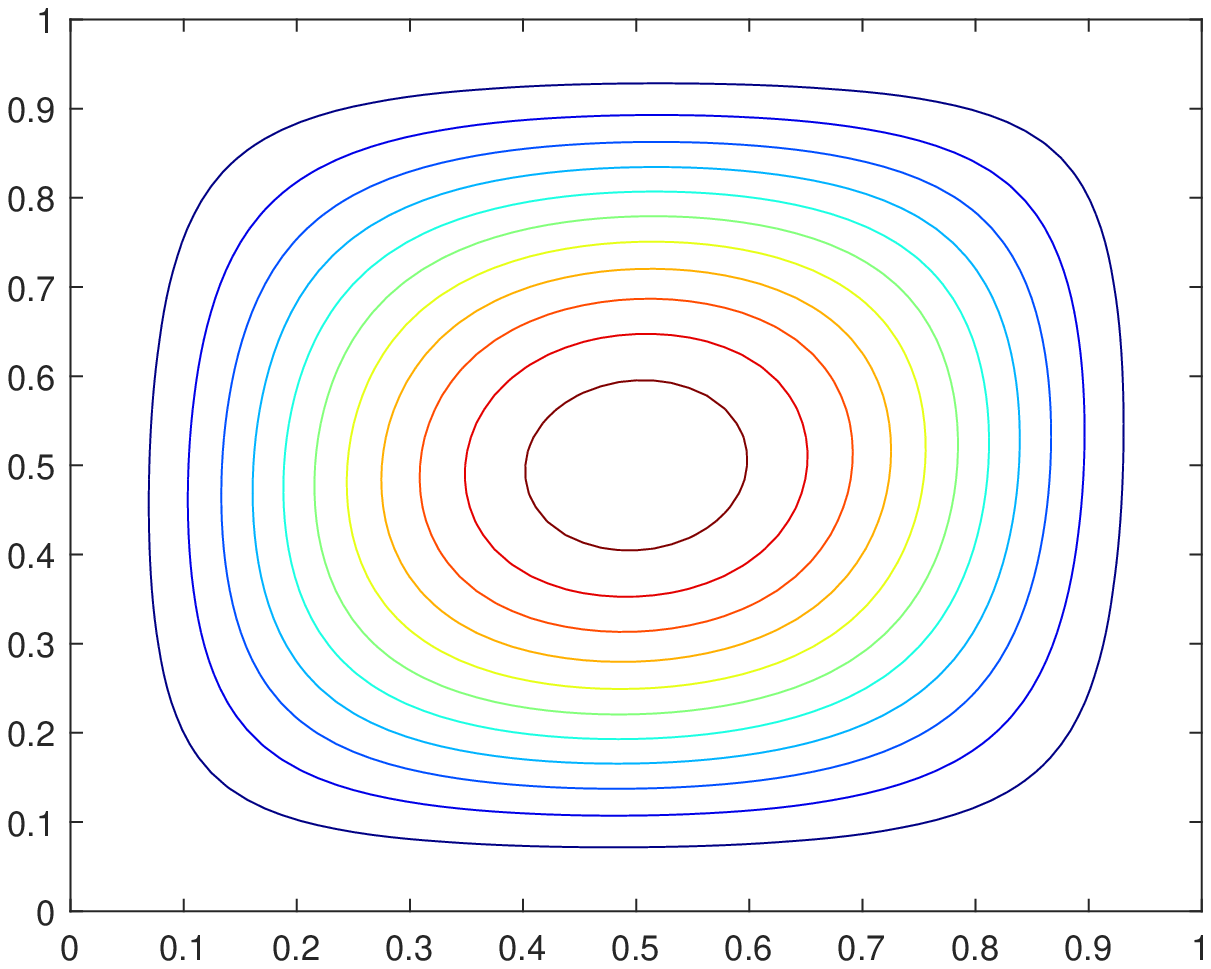} \hspace*{-0.6cm}
		\includegraphics[height=4.6cm, width=4.6cm]{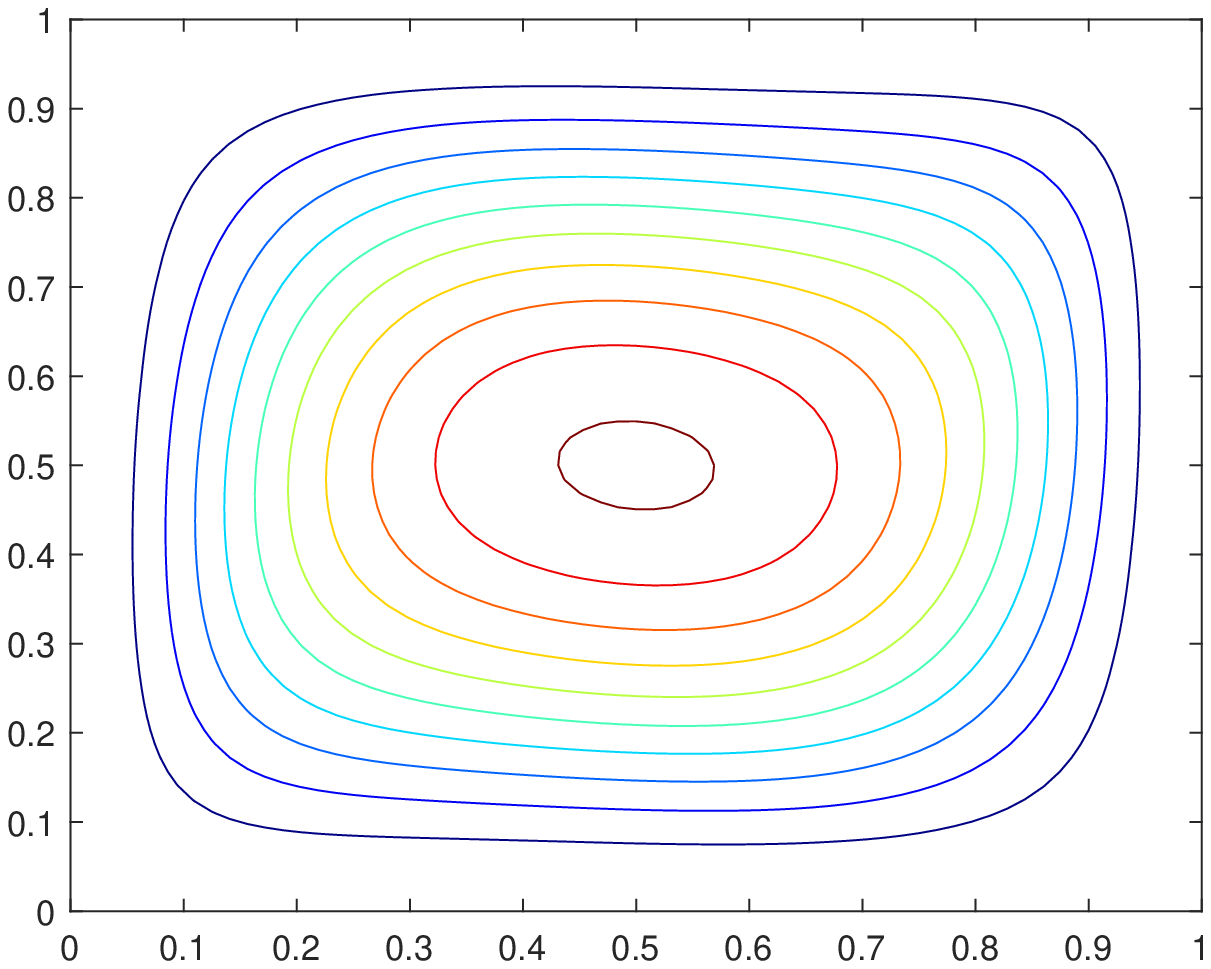} \hspace*{-0.6cm}
		\includegraphics[height=4.6cm, width=4.6cm]{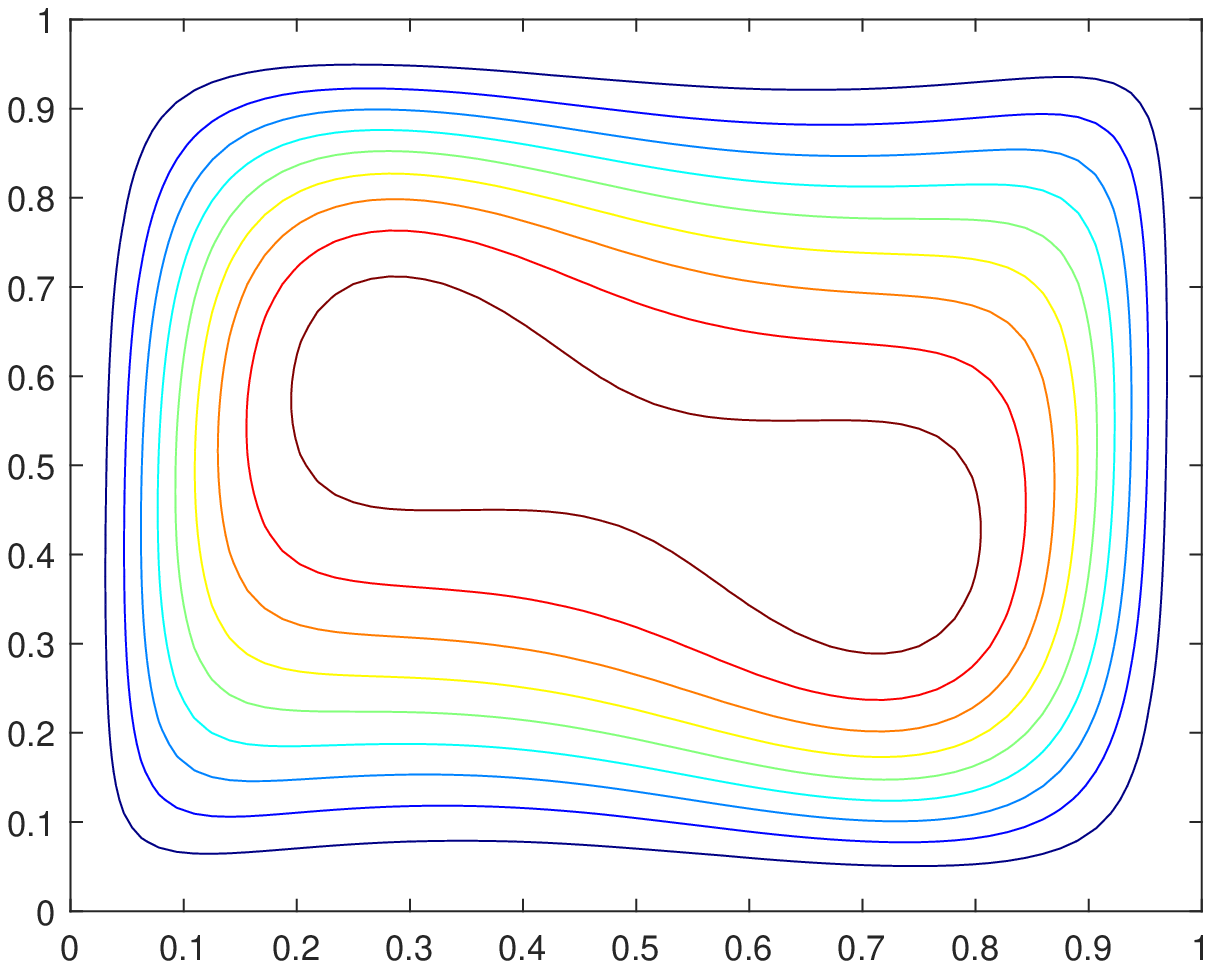} \hspace*{-0.6cm}
		\includegraphics[height=4.6cm, width=4.6cm]{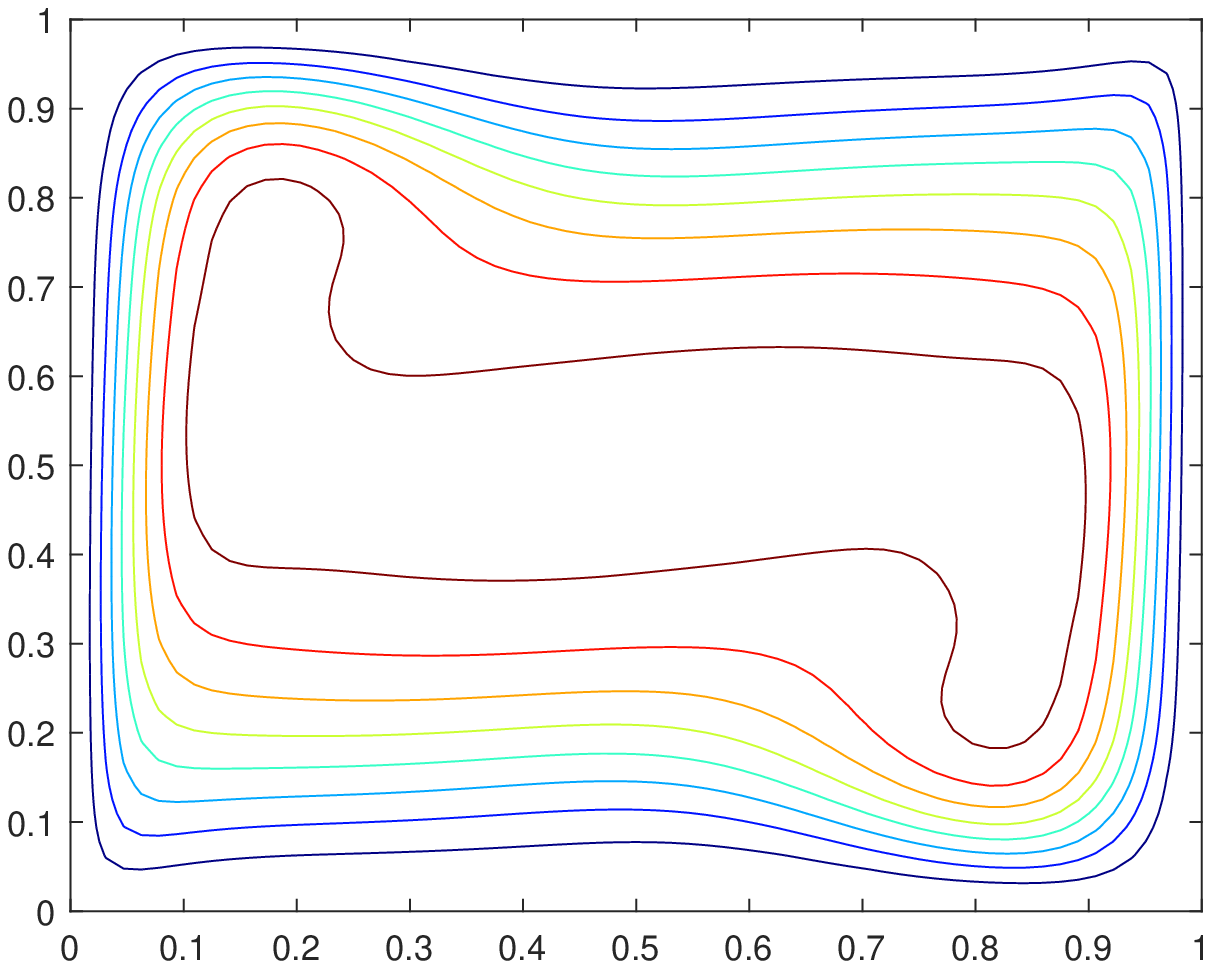}\\
		\includegraphics[height=4.6cm, width=4.6cm]{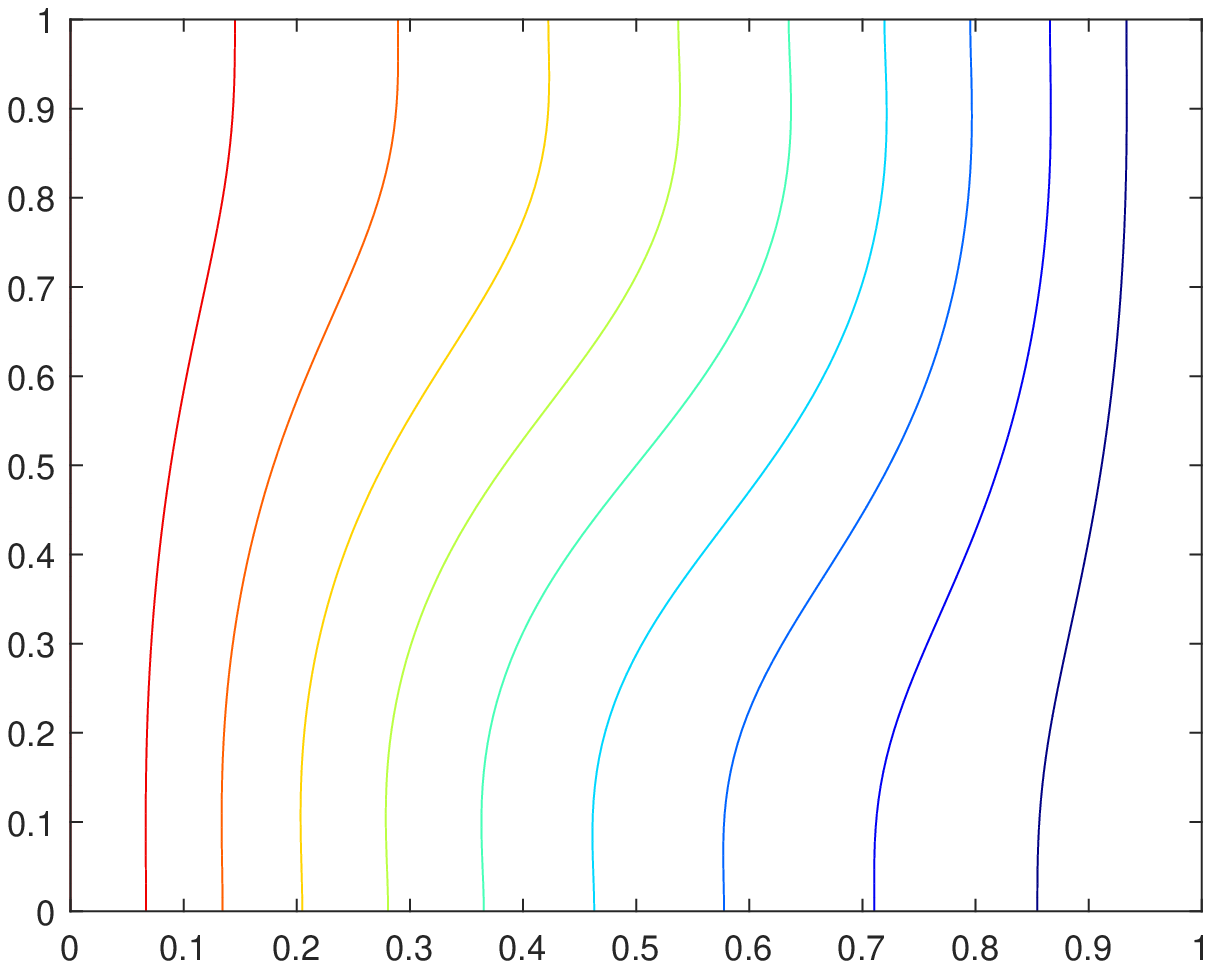} \hspace*{-0.6cm}
		\includegraphics[height=4.6cm, width=4.6cm]{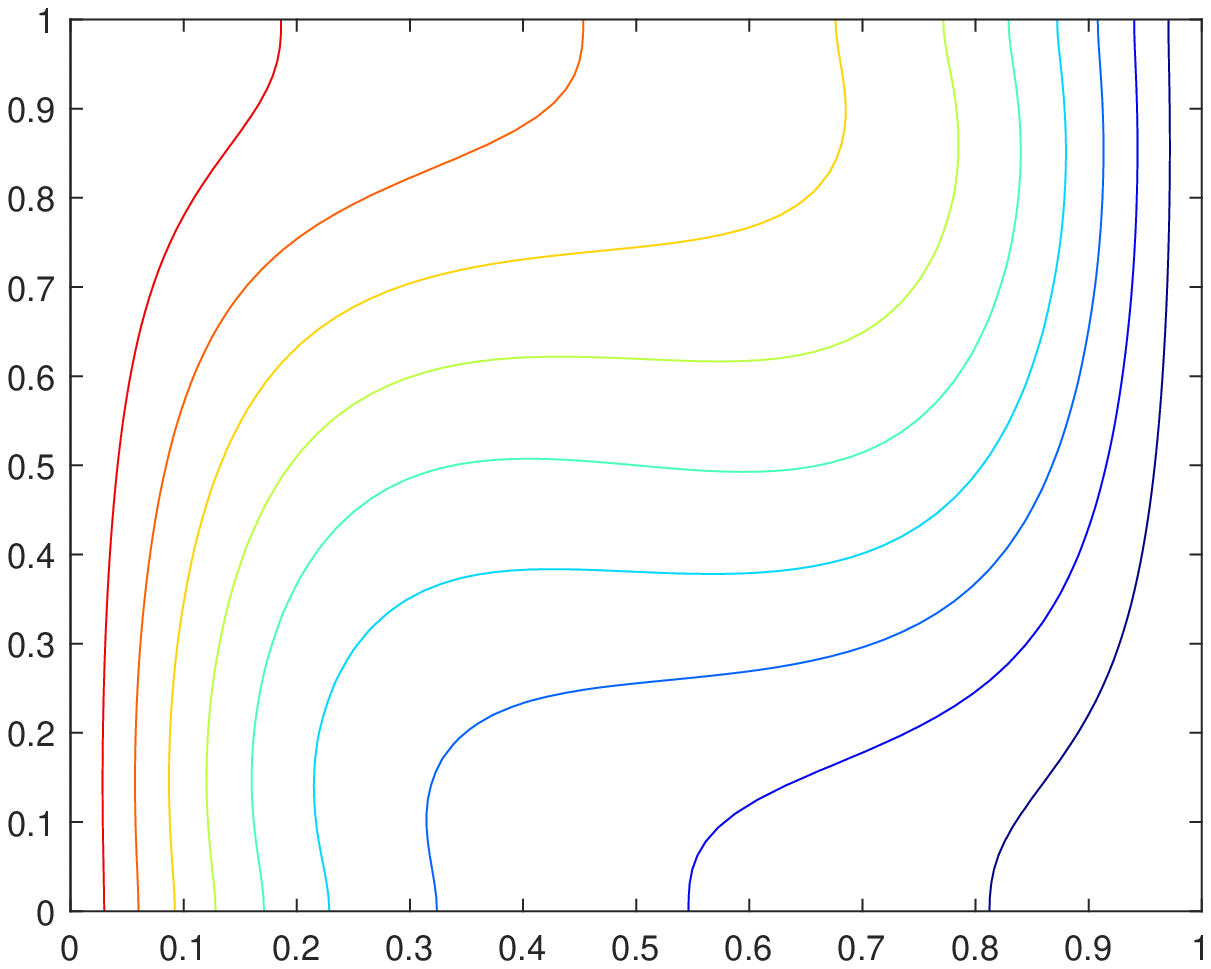} \hspace*{-0.6cm}
		\includegraphics[height=4.6cm, width=4.6cm]{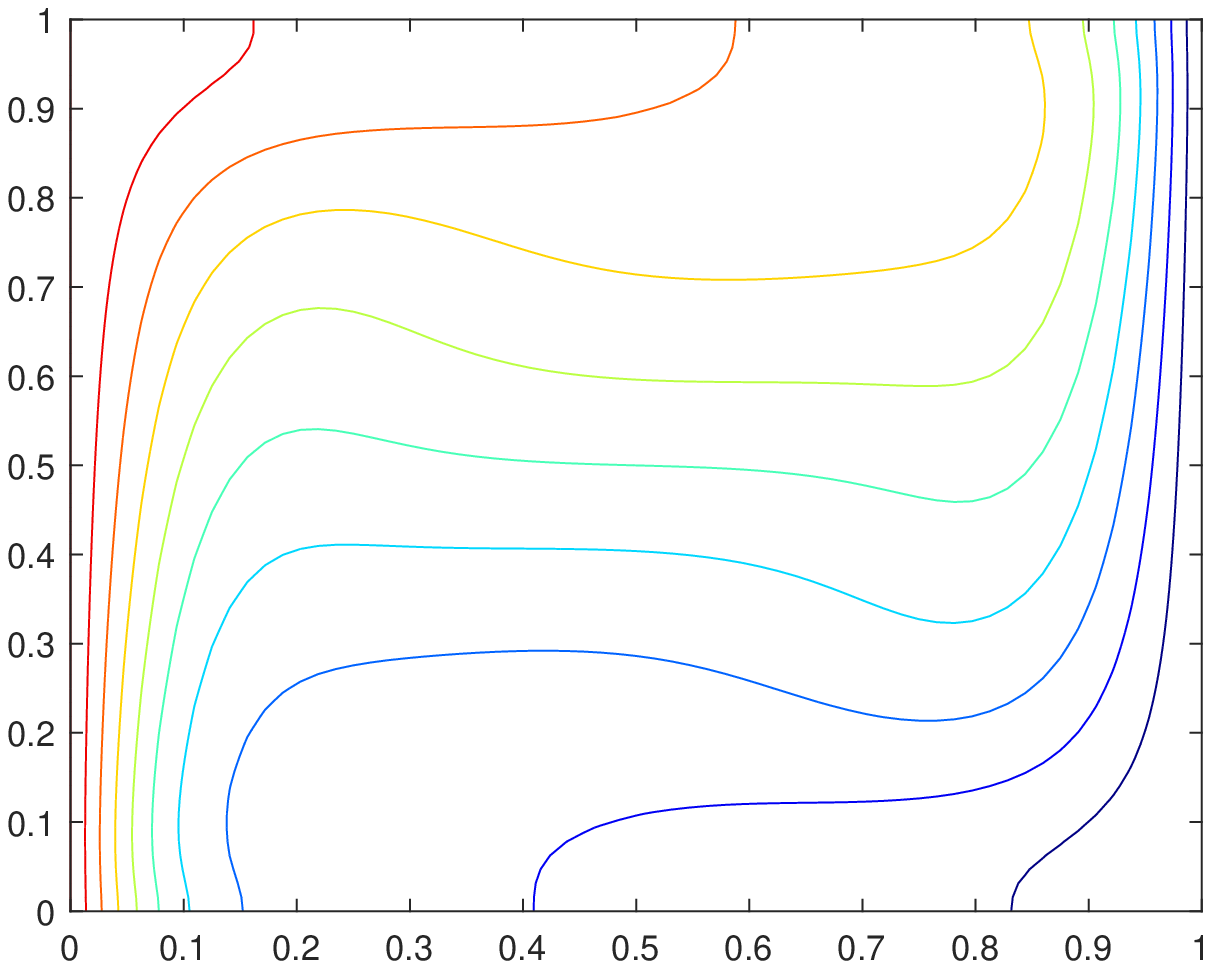} \hspace*{-0.6cm}
		\includegraphics[height=4.6cm, width=4.6cm]{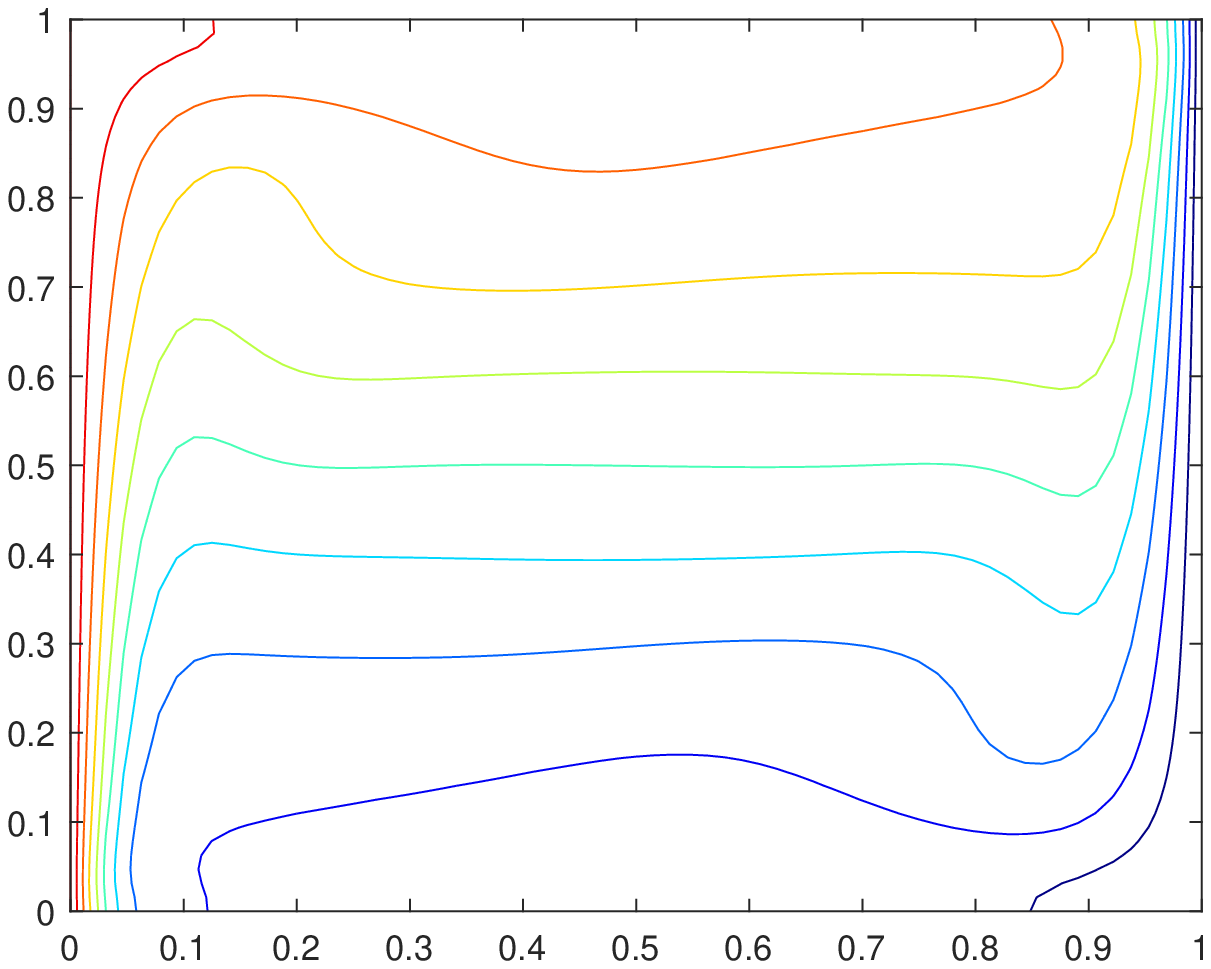}
	\end{center}
	\caption{Natural convection cavity:  streamlines (top panels) and isotherms (bottom panels), for ${\rm Ra}= 10^3, 10^4, 10^5$ and $10^6$,  respectively (from left to right).}
	\label{stream-isotherm} 
\end{figure}

Tables~\ref{tabla3-1} and~\ref{tabla3-2} present a
quantitative comparison between our results and those obtained by the 
benchmark solutions in the above papers. 
Table~\ref{tabla3-1}  shows the  maximum vertical velocity at $y= 0.5$,  
for $\Ra=10^4, 10^5$ and  $10^6$, while Table~\ref{tabla3-2} shows the 
maximum horizontal velocity at  $x=0.5$, using the same values of the 
Rayleigh number. Here the numbers in the parenthesis 
indicate the mesh size used by the respective reference.
We can observe that the results show good agreement, even for higher Rayleigh numbers.

%%%____________________________ Table 3 ___________________________
\begin{table}[!h] 
\begin{center}
		{\small \begin{tabular}{lllllllllll}
		\hline  
		\hline
		${\rm Ra}$ & $\text{VEM}$ & $\text{Ref}$ \cite{ZHZ2014_camwa} &  $\text{Ref}$ \cite{Vahl83} & $\text{Ref}$ \cite{MNZ98} & $\text{Ref}$ \cite{Manzari99} & $\text{Ref}$ \cite{WPW2001} \\
		\hline

$10^4$ & 19.56(64) & 19.63(64) & 19.51(41) & 19.63(71) &19.90(71)& 19.79(101)\\
$10^5$ & 68.46(64) & 68.48(64) &68.22(81)  &68.85(71)  &70.00(71) & 70.63(101) \\
$10^6$ & 216.37(64)& 220.46(64)& 216.75(81)& 221.6(71) & 228.0(71) &227.11(101) \\
\hline  
\hline 
\end{tabular}}
\end{center} 
\caption{Natural convection cavity. Comparison of maximum vertical velocity $u_{1h}:=\Pi_h^1 \partial_y \psi$ at $y = 0.5$ with the VEM~\eqref{eq:VEM2} and mesh $\O^5_h$ ($h=1/64$).}
\label{tabla3-1}
		\end{table}
		%%%___________________

%%%____________________________ Table 4 ___________________________
\begin{table}[!h] 
\begin{center}
{\small \begin{tabular}{lllllllllll}%{llccccccccc}
\hline
\hline
${\rm Ra}$ & $\text{VEM}$ & $\text{Ref}$ \cite{ZHZ2014_camwa} &  $\text{Ref}$ \cite{Vahl83} & $\text{Ref}$ \cite{Manzari99} & $\text{Ref}$ \cite{WPW2001} \\
\hline
$10^4$ & 16.15(64) &16.19(64) & 16.18(41)& 16.10(71) &16.10(101)\\
$10^5$ & 34.80(64) &34.74(64) & 34.81(81)& 34.0(71)  &34.00(101) \\
$10^6$ & 65.91(64) & 64.81(64)& 65.33(81)& 65.40(71) &65.40(101) \\
\hline  
\hline 
\end{tabular}}
\end{center}
	\caption{Natural convection cavity. Comparison of maximum horizontal velocity $u_{2h}:=-\Pi_h^1 \partial_x \psi$ at $x = 0.5$  with the VEM~\eqref{eq:VEM2} and mesh $\O^5_h$ ($h=1/64$).}
	\label{tabla3-2}
\end{table}
%%%_______________________________________________________________________________

Finally, for the natural convection problem we investigate the heat transfer coefficient 
along the vertical walls of the cavity in terms of the local Nusselt number $(\Nu)$, 
which is defined by: $\Nu(x,y) := - \partial_{\bn} \theta(x,y)$.   
Figure~\ref{fig3-1} describes the variation of local Nusselt
number at hot wall and cold wall, for different values of the Rayleigh number.  
It can be seen that the results show good agreement with the results presented in~\cite{ZHZ2014_camwa,Vahl83,MNZ98,Manzari99,WPW2001}.

%%%________________________________________________________________________________
\begin{figure*}[!h]% [hpbt] 
	\centering
	\subfigure[Nusselt in the hot wall]{%
		\includegraphics[width=0.40\linewidth, height=0.43\textwidth]{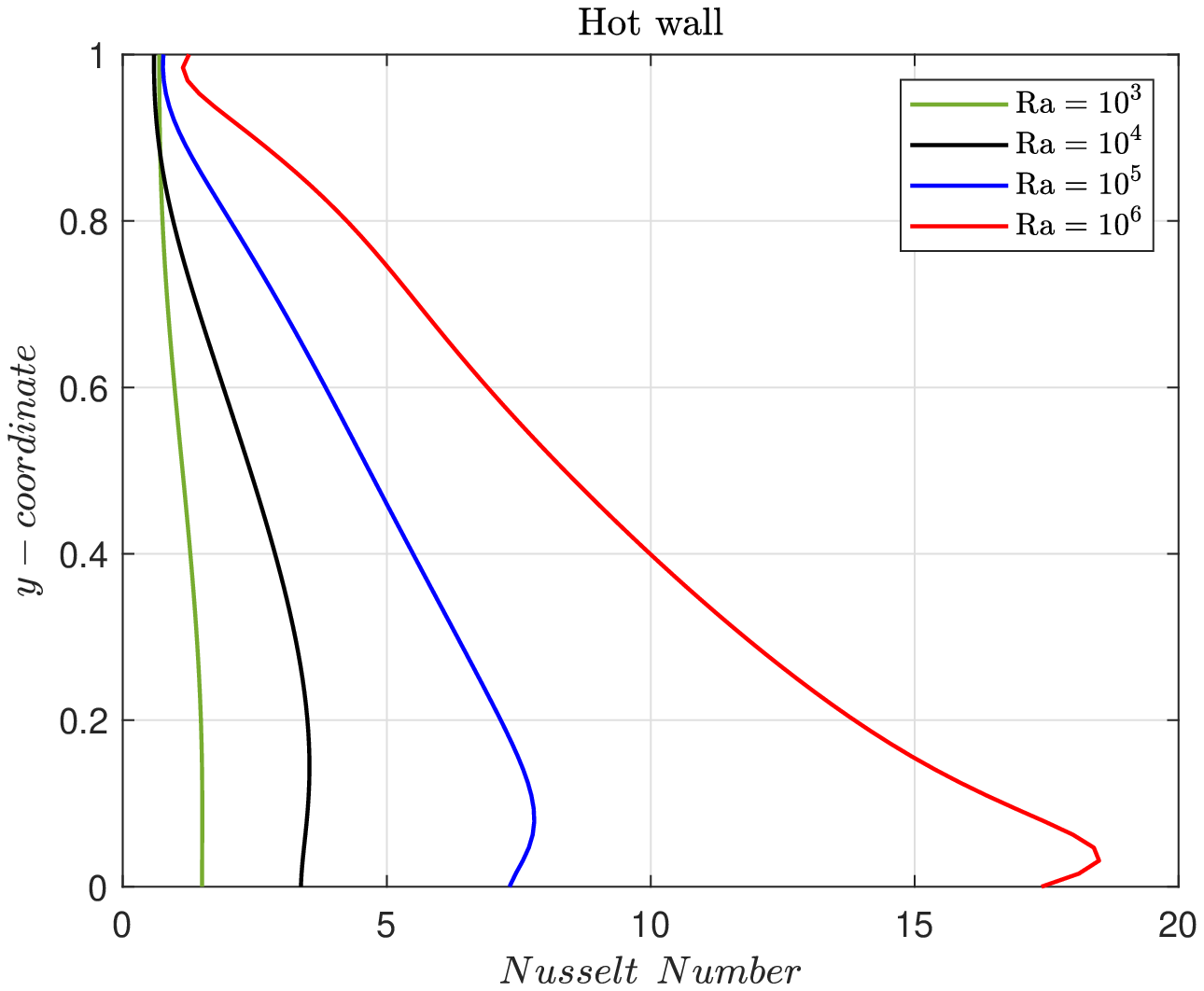}
		\label{nusselt_hot} } 
	\subfigure[Nusselt in the cold wall]{%{0.3\textwidth}
		\includegraphics[width=0.40\linewidth, height=0.43\textwidth]{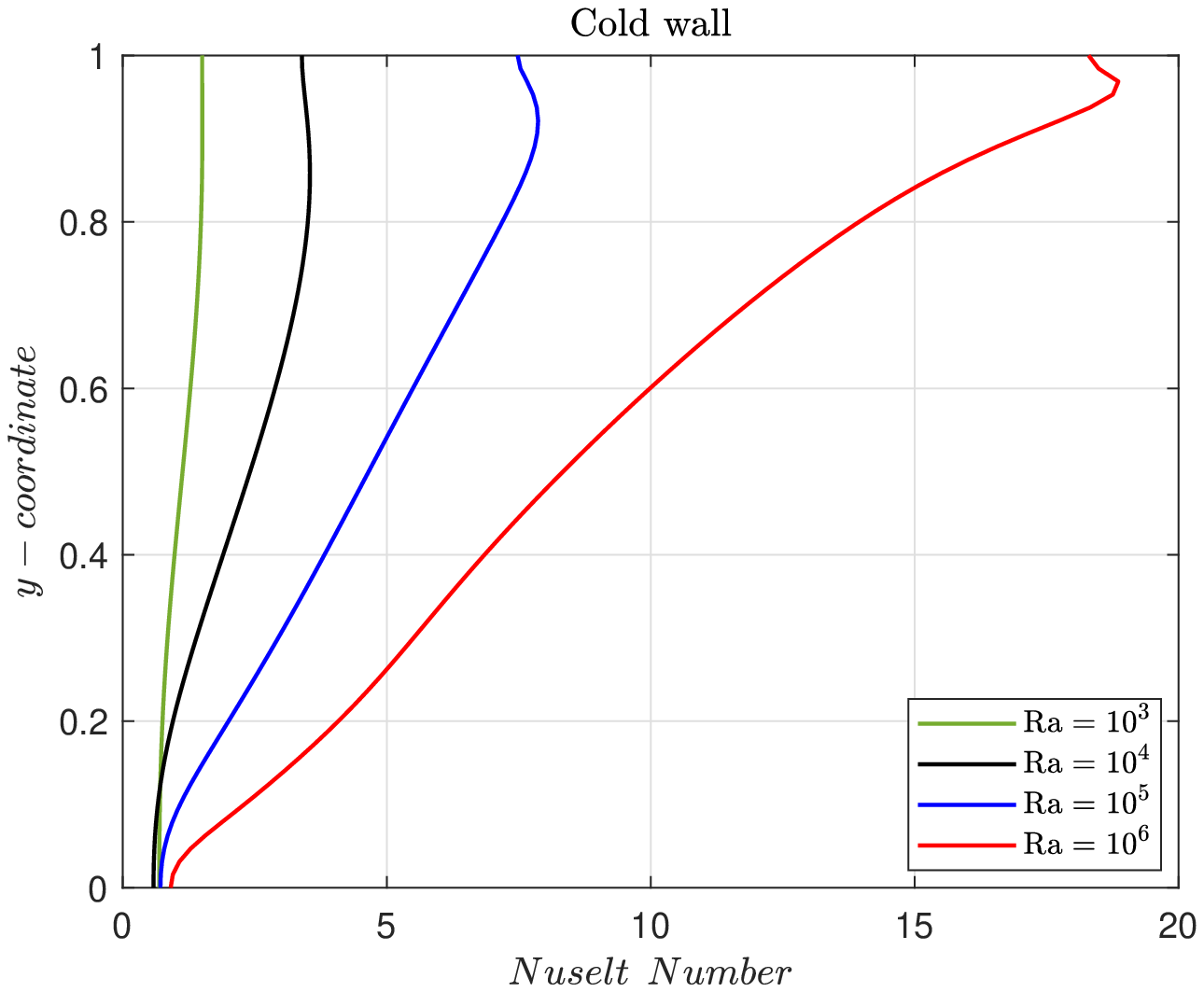}
		\label{nusselt_cold}} 
	\caption{Natural convection cavity.  Nusselt number along the hot wall (left) and the cold wall (right) for
		varying Rayleigh numbers, using the VEM~\eqref{eq:VEM2} and mesh $\O^5_h$, with $h=1/64$.}
\label{fig3-1}
\end{figure*}
%%%%___________________

%%%% _________________________________________________________________________
\small
\section*{Acknowledgements} 	
	The first author a was partially supported by the Italian MIUR through the 
	PRIN grants n. 905	201744KLJL.
	The second author was partially supported by the National Agency
	for Research and Development, ANID-Chile through FONDECYT project 1220881,
	by project {\sc Anillo of Computational Mathematics for Desalination Processes} ACT210087,
	and by project Centro de Modelamiento Matem\'atico (CMM), FB210005,
	BASAL funds for centers of excellence.
	The third author was supported by the National Agency for Research
	and Development, ANID-Chile, Scholarship Program,
	Doctorado Becas Chile 2020, 21201910.	
	%-----------------------------------------------------------------------

\end{document}